\documentclass[11pt]{article}
\usepackage{amssymb,amsfonts}
\usepackage{amsmath,amsthm,amsxtra}

\usepackage{graphicx}
\usepackage{amscd}
\usepackage{ulem}
\usepackage{makeidx}
\usepackage{fancybox}

\newcommand{\ccc}{{\mathbf C}}

\newcommand{\nnn}{{\mathbf N}}

\newcommand{\rrr}{{\mathbf R}}
\newcommand{\zzz}{{\mathbf Z}}

\newtheorem{thm}{Theorem}[section]
\newtheorem{prop}{Proposition}[section]
\newtheorem{lemma}{Lemma}[section]

\newtheorem{rem}{Remark}[section]
\newtheorem{note}{Note}[section]

\numberwithin{equation}{section}

\begin{document}

\title{Mock theta functions and indefinite modular forms}

\author{\footnote{12-4 Karato-Rokkoudai, Kita-ku, Kobe 651-1334, 
Japan, \qquad
wakimoto.minoru.314@m.kyushu-u.ac.jp, \hspace{5mm}
wakimoto@r6.dion.ne.jp 
}{ Minoru Wakimoto}}

\date{\empty}

\maketitle

\begin{center}
Abstract
\end{center}
In the explicit formula for the signed mock theta functions 
$\Phi^{(-)[m,s]}$ obtained from the coroot lattice of $D(2,1;a)$, 
functions with indefinite quadratic forms naturally take place. 
We compute their modular transformation properties by
applying the Zwegers' modification theory of mock theta functions 
and show that the $\ccc$-linear span of these functions is 
$SL_2(\zzz)$-invariant.

\tableofcontents

\section{Introduction}

For $m \in \frac12 \nnn$ and $s \in \frac12 \zzz$, 
let $\Phi^{(\pm)[m,s]}_i$ $(i=1,2)$ and $\Phi^{(\pm)[m,s]}$ be 
the (signed) mock theta functions (cf. \cite{KW2016b}) defined by 
\begin{subequations}
\begin{eqnarray}
\Phi^{(\pm)[m,s]}_1(\tau, z_1, z_2,t) &:=& e^{-2\pi imt} 
\sum_{j \in \zzz} \, (\pm 1)^j \, 
\frac{e^{2\pi imj(z_1+z_2)+2\pi isz_1} q^{mj^2+sj}}{1-e^{2\pi iz_1}q^j}
\label{n3-2:eqn:2022-111a}
\\[1mm]
\Phi^{(\pm)[m,s]}_2(\tau, z_1, z_2,t) &:=& e^{-2\pi imt} 
\sum_{j \in \zzz} \, (\pm 1)^j \, 
\frac{e^{-2\pi imj(z_1+z_2)-2\pi isz_2} q^{mj^2+sj}}{1-e^{-2\pi iz_2}q^j}
\label{n3-2:eqn:2022-111b}
\\[1mm]
\Phi^{(\pm)[m,s]}(\tau, z_1, z_2,t) &:=& 
\Phi^{(\pm)[m,s]}_1(\tau, z_1, z_2,t) -\Phi^{(\pm)[m,s]}_2(\tau, z_1, z_2,t) 
\label{n3-2:eqn:2022-111c}
\end{eqnarray}
\end{subequations}
where $q:=e^{2\pi i\tau} \,\ (\tau \in \ccc_+)$ and $z_1, z_2, t \in \ccc$. 
We write simply $\Phi^{[m,s]}_i$ and $\Phi^{[m,s]}$ 
for $\Phi^{(+)[m,s]}_i$ and $\Phi^{(+)[m,s]}$ respectively. 
For $m \in \frac12 \nnn$ and $j \in \frac12 \zzz$, 
we define the function $R^{(\pm)}_{j,m}(\tau,z)$ by 
{\allowdisplaybreaks
\begin{eqnarray}
R^{(\pm)}_{j,m}(\tau,w) &:= & 
\sum_{\substack{n \in \frac12 \zzz \\[1mm]
n -j \in 2m\zzz}}
(\pm 1)^{\frac{n-j}{2m}}\bigg\{
{\rm sgn}\Big(n-\frac12-j+2m\Big)
\nonumber
\\[1mm]
& &
-E\bigg(\Big(n-2m\frac{{\rm Im}(w)}{{\rm Im}(\tau)}\Big)
\sqrt{\frac{{\rm Im}(\tau)}{m}}\bigg)\bigg\} \, 
e^{-\frac{\pi i\tau}{2m}n^2+2\pi inw}
\label{m1:eqn:2022-512a}
\end{eqnarray}}
where $(\tau,w) \in \ccc_+ \times \ccc$ and 
$E(z):= 2\int_0^z e^{-\pi t^2}dt$.
For $m \in \frac12 \nnn$ and $s \in \frac12 \zzz$, we put 
\begin{subequations}
{\allowdisplaybreaks
\begin{eqnarray}
& &\hspace{-15mm}
\Phi^{(\pm)[m,s]}_{1, {\rm add}}(\tau, z_1,z_2,t) 
\, := \, 
- \, \frac12 \, e^{-2\pi imt} \hspace{-2mm}
\sum_{\substack{\\ k \in s+\zzz \\[1mm] 
s \leq k < s+2m }} \hspace{-2mm}
R^{(\pm)}_{k,m}\Big(\tau, \frac{z_1-z_2}{2}\Big) 
\theta_{k,m}^{(\pm)}(\tau, \, z_1+z_2) 
\label{m1:eqn:2022-512b1a}
\\[1mm]
& &\hspace{-15mm}
\Phi^{(\pm)[m,s]}_{2, {\rm add}}(\tau, z_1,z_2,t) 
\, := \, 
- \, \frac12 \, e^{-2\pi imt} \hspace{-2mm}
\sum_{\substack{\\ k \in s+\zzz \\[1mm] 
s \leq k < s+2m }} \hspace{-2mm}
R^{(\pm)}_{k,m}\Big(\tau, \frac{z_1-z_2}{2}\Big) 
\theta_{-k,m}^{(\pm)}(\tau, \, z_1+z_2) 
\label{m1:eqn:2022-512b1b}
\\[1mm]
& &\hspace{-15mm}
\Phi^{(\pm)[m,s]}_{\rm add}(\tau, z_1,z_2,t) 
\,\ := \,\ 
\Phi^{(\pm)[m,s]}_{1, {\rm add}}(\tau, z_1,z_2,t) 
-
\Phi^{(\pm)[m,s]}_{2, {\rm add}}(\tau, z_1,z_2,t) 
\nonumber
\\[2.5mm]
&=&
- \, \frac12 \, e^{-2\pi imt} \hspace{-2mm}
\sum_{\substack{\\ k \in s+\zzz \\[1mm] 
s \leq k < s+2m }} \hspace{-2mm}
R^{(\pm)}_{k,m}\Big(\tau, \frac{z_1-z_2}{2}\Big) 
\big[\theta_{k,m}^{(\pm)}-\theta_{-k,m}^{(\pm)}\big]
(\tau, \, z_1+z_2) 
\label{m1:eqn:2022-512b1}
\end{eqnarray}}
\end{subequations}
and
\begin{subequations}
{\allowdisplaybreaks
\begin{eqnarray}
& & \hspace{-10mm}
\widetilde{\Phi}^{(\pm)[m,s]}_i(\tau, z_1,z_2,t) 
\, := \, 
\Phi^{(\pm)[m,s]}_i(\tau, z_1,z_2,t) 
+
\Phi^{(\pm)[m,s]}_{i,{\rm add}}(\tau, z_1,z_2,t) \hspace{5mm}
(i=1,2)
\label{m1:eqn:2022-512b2a}
\\[2mm]
& & \hspace{-10mm}
\widetilde{\Phi}^{(\pm)[m,s]}(\tau, z_1,z_2,t) 
\, := \, 
\Phi^{(\pm)[m,s]}(\tau, z_1,z_2,t) 
+
\Phi^{(\pm)[m,s]}_{\rm add}(\tau, z_1,z_2,t) 
\label{m1:eqn:2022-512b2}
\end{eqnarray}}
\end{subequations}
where $\theta_{k,m}^{(\pm)}$ is the Jacobi's theta function:
\begin{equation}
\theta_{k,m}^{(\pm)}(\tau,z) \, := \, 
\sum_{j \in \zzz} (\pm 1)^j e^{2\pi im(j+\frac{k}{2m})z} \, 
q^{m(j+\frac{k}{2m})^2}
\label{m1:eqn:2022-512c}
\end{equation}

\noindent
We call $\widetilde{\Phi}^{(\pm)[m,s]}_i$ (resp. $\Phi^{(\pm)[m,s]}_{i,{\rm add}}$)
the \lq \lq modification" (resp. \lq \lq correction term") of 
$\Phi^{(\pm)[m,s]}_i$.

\medskip

It is known that the Jacobi's theta functions satisfy the following properties:

\begin{lemma} 
\label{m1:lemma:2022-526a}
Let $m \in \frac12 \nnn$, $j \in \frac12 \zzz$ and $a \in \rrr$ 
such that $am \in \frac12 \zzz$. Then
\begin{enumerate}
\item[{\rm 1)}] \,\ $\theta_{j,m}^{(\pm)}(\tau, z+a\tau) 
\, = \, q^{-\frac14ma^2} e^{-\pi imaz}
\theta_{j+am,m}^{(\pm)}(\tau, z)$
\item[{\rm 2)}] \,\ $\theta_{j,m}^{(\pm)}(\tau, z+a) 
\,\ = \, \left\{
\begin{array}{lcl}
e^{\pi ija} \, \theta_{j,m}^{(\pm)}(\tau, z) \quad & \text{if} & 
am \, \in \, \zzz \\[2mm]
e^{\pi ija} \, \theta_{j,m}^{(\mp)}(\tau, z) & \text{if} & 
am \, \in \, \frac12 \zzz_{\rm odd}
\end{array} \right. $
\end{enumerate}
\end{lemma}

\begin{lemma} 
\label{m1:lemma:2022-526b}
$\theta_{j+2am,m}^{(\pm)}(\tau,z) = (\pm 1)^a \theta_{j,m}^{(\pm)}(\tau,z)$
\,\ if $m \in \frac12 \nnn$, $j \in \frac12 \zzz$ and $a \in \zzz$.
\end{lemma}

\begin{lemma} 
\label{m1:lemma:2022-526c}
Let $m \in \frac12 \nnn$, $j \in \frac12 \zzz$. Then
\begin{enumerate}
\item[{\rm 1)}] \,\ $\theta_{j,m}^{(+)}
\Big(-\dfrac{1}{\tau}, \dfrac{z}{\tau}\Big)
\, = \,
\dfrac{(-i\tau)^{\frac12} }{\sqrt{2m}} \, 
e^{\frac{\pi imz^2}{2\tau}} \, 
\times \, \left\{
\begin{array}{ccl}
\sum\limits_{\substack{k \in \zzz \\[1mm] 0 \leq k <2m}}
e^{\frac{\pi ijk}{m}} \, \theta_{k,m}^{(+)}(\tau,z) \quad 
&\text{if}& j \, \in \, \zzz
\\[7.5mm]
\sum\limits_{\substack{k \in \zzz \\[1mm] 0 \leq k <2m}}
e^{\frac{\pi ijk}{m}} \, \theta_{k,m}^{(-)}(\tau,z) \quad 
&\text{if}& j \, \in \, \frac12 \zzz_{\rm odd}
\end{array}\right.  $
$$
= \, \dfrac{(-i\tau)^{\frac12}}{\sqrt{2m}} \, 
e^{\frac{\pi imz^2}{2\tau}} \, 
\times \left\{
\begin{array}{ccl}
\sum\limits_{k \in \zzz/2m\zzz}
e^{\frac{\pi ijk}{m}} \, \theta_{k,m}^{(+)}(\tau,z) \quad 
&\text{if}& j \, \in \, \zzz
\\[4mm]
\sum\limits_{k \in \zzz/2m\zzz}
e^{\frac{\pi ijk}{m}} \, \theta_{k,m}^{(-)}(\tau,z) \quad 
&\text{if}& j \, \in \, \frac12 \zzz_{\rm odd}
\end{array}\right. 
$$
\item[{\rm 2)}]  $\theta_{j,m}^{(-)}
\Big(-\dfrac{1}{\tau}, \dfrac{z}{\tau}\Big)
\,\ = \,\ \dfrac{1}{\sqrt{2m}} \, (-i\tau)^{\frac12} \, 
e^{\frac{\pi imz^2}{2\tau}}
\times \left\{
\begin{array}{ccl}
\sum\limits_{\substack{k \in \frac12 \zzz_{\rm odd} \\[1mm] 0 < k <2m}}
e^{\frac{\pi ijk}{m}} \, \theta_{k,m}^{(+)}(\tau,z) \,\ 
&\text{if}& j \, \in \, \zzz
\\[7.5mm]
\sum\limits_{\substack{k \in \frac12 \zzz_{\rm odd} \\[1mm] 0 < k <2m}}
e^{\frac{\pi ijk}{m}} \, \theta_{k,m}^{(-)}(\tau,z) \,\ 
&\text{if}& j \, \in \, \frac12 \zzz_{\rm odd}
\end{array}\right. $
$$
= \dfrac{(-i\tau)^{\frac12}}{\sqrt{2m}} \, 
e^{\frac{\pi imz^2}{2\tau}} \times
\left\{
\begin{array}{ccl}
\sum\limits_{k \in \zzz/2m\zzz}
e^{\frac{\pi ij}{m}(k+\frac12)} \, \theta_{k+\frac12,m}^{(+)}(\tau,z) \,\ 
&\text{if}& j \in \zzz
\\[4mm]
\sum\limits_{k \in \zzz/2m\zzz}
e^{\frac{\pi ij}{m}(k+\frac12)} \, \theta_{k+\frac12,m}^{(-)}(\tau,z) \,\ 
&\text{if}& j \in \frac12 \zzz_{\rm odd}
\end{array}\right. 
$$
\end{enumerate}
\end{lemma}

\begin{lemma} 
\label{m1:lemma:2022-526d}
Let $m \in \frac12 \nnn$, $j \in \frac12 \zzz$. Then
$$
\theta_{j,m}^{(\pm)}(\tau+1, z) \,\ = \,\ \left\{
\begin{array}{ccl}
e^{\frac{\pi ij^2}{2m}} \, \theta_{j,m}^{(\pm)}(\tau,z) \quad 
&\text{if}& m+j \, \in \, \zzz \\[2mm]
e^{\frac{\pi ij^2}{2m}} \, \theta_{j,m}^{(\mp)}(\tau,z) \quad 
&\text{if}& m+j \, \in \, \frac12 \zzz_{\rm odd}
\end{array} \right.
$$
\end{lemma}

\medskip

We note that the Mumford's theta functions $\vartheta_{ab}(\tau,z)$ 
(cf. \cite{Mum}) are related to the Jacobi's theta functions 
by the following formulas:

\begin{note} \quad 
\label{m1:note:2022-526a}
$\left\{
\begin{array}{lcr}
\vartheta_{00}(\tau,z) &=& \theta_{0, \frac12}^{(+)}(\tau, 2z) \\[3mm]
\vartheta_{01}(\tau,z) &=& \theta_{0, \frac12}^{(-)}(\tau, 2z)
\end{array}\right. \hspace{10mm} \left\{
\begin{array}{lcr}
\vartheta_{10}(\tau,z) &=& \theta_{\frac12, \frac12}^{(+)}(\tau, 2z) \\[3mm]
\vartheta_{11}(\tau,z) &=& i \, \theta_{\frac12, \frac12}^{(-)}(\tau, 2z)
\end{array}\right. $
\end{note}

\noindent
In this paper, we write simply $\theta_{j,m}$ for \, $\theta_{j,m}^{(+)}$.

\medskip

This paper is organized as follows. 
In section 2, we recall and prepare some of basic properties of 
mock theta functions $\Phi^{(\pm)[m,s]}_i$ and $\Phi^{(\pm)[m,s]}$ and 
$\widetilde{\Phi}^{(\pm)[m,s]}$, which are used in this paper.

In section 3, we deduce explicit formulas for $\Phi^{(-)[m,\frac12]}$ 
by using similar method as in \cite{W2022b} and, in sections 4 and 5,
we compute Zwegers' modification for these functions.

In sections 6 and 7, we study the functions $\Xi^{(i)[m,p]}(\tau,z)$ 
and $\Upsilon^{(i)[m,p]}(\tau,z)$ which have good modular properties.
In section 8, we compute the modular transformation properties 
of indefinite modular forms $g^{(i)[m,p]}_j(\tau)$ and show that 
their $\ccc$-linear span is $SL_2(\zzz)$-invariant.

\section{Basic properties of $\Phi^{(\pm)[m,s]}$}
\label{m1:sec:basic-properties}

Some of basic properties of $\Phi^{[m,s]}_i$ and $\Phi^{[m,s]}$
are shown in section 2 of \cite{W2022a} and \cite{W2022b}.
In this section, we note that similar formulas hold 
for signed mock theta functions $\Phi^{(\pm)[m,s]}_i$ and 
$\Phi^{(\pm)[m,s]}$ and $\widetilde{\Phi}^{(\pm)[m,s]}_i$
and $\widetilde{\Phi}^{(\pm)[m,s]}$.

\begin{note}
\label{m1:note:2022-830a}
Let $m \in \frac12 \nnn$ and $s \in \frac12 \zzz$. Then
\begin{enumerate}
\item[{\rm 1)}] \,\ $\Phi^{(\pm)[m,s]}_i(\tau, -z_2, -z_1,t)
\, = \, \Phi^{(\pm)[m,s]}_j(\tau, z_1, z_2,t)$ \,\ for \,\ 
$i,j \in \{1,2\}$ \,\ such that \,\ $i \ne j$.
\item[{\rm 2)}] \,\ $\Phi^{(\pm)[m,s]}(\tau, -z_2, -z_1,t)
\, = \, - \, \Phi^{(\pm)[m,s]}(\tau, z_1, z_2,t)$
\item[{\rm 3)}] \,\ $\Phi^{(\pm)[m,s]}(\tau, z, -z,t) \, = \, 0$
\end{enumerate}
\end{note}

\vspace{0mm}

\begin{lemma} 
\label{m1:lemma:2022-512a}
Let $m \, \in \frac12 \nnn$ and $s, \, s' \in \frac12 \zzz$ 
and $j \in \{1,2\}$. Then
\begin{enumerate}
\item[{\rm 1)}] $\widetilde{\Phi}^{(\pm)[m,s]}_j(\tau, z_1,z_2,t) 
\, = \, 
\widetilde{\Phi}^{(\pm)[m,s']}_j(\tau, z_1,z_2,t)$ \quad if 
$s-s' \in \zzz$, 
\item[{\rm 2)}] $\widetilde{\Phi}^{(-)[m,\frac12]}_j\Big(
-\dfrac{1}{\tau}, \dfrac{z_1}{\tau},\dfrac{z_2}{\tau},t\Big) 
\, = \, \tau \, 
e^{\frac{2\pi im}{\tau}z_1z_2} \, 
\widetilde{\Phi}^{(-)[m,\frac12]}_j(\tau, z_1,z_2,t)$
\item[{\rm 3)}] $\widetilde{\Phi}^{(\pm)[m,s]}_j(\tau+1, z_1,z_2,t)
\, = \, 
\widetilde{\Phi}^{(\pm)[m,s]}_j(\tau, z_1,z_2,t)$ \quad if \,\ 
$m+s \in \zzz$.
\end{enumerate}
\end{lemma}

\vspace{0mm}

\begin{lemma} 
\label{m1:lemma:2022-512b}
Let $m \in \frac12 \nnn$, $s \in \frac12 \zzz$ and $p \in \zzz$. Then 
\begin{enumerate}
\item[{\rm 1)}]
\begin{enumerate}
\item[{\rm (i)}] $\Phi_1^{(\pm)[m, s]}(\tau, z_1, z_2+p\tau, t) 
= e^{-2\pi impz_1}  
\Phi_1^{(\pm)[m, s+mp]}(\tau, z_1, z_2, t)$
\item[{\rm (ii)}] $\Phi_2^{(\pm)[m, s]}(\tau, z_1, z_2+p\tau, t) 
= (\pm 1)^p e^{-2\pi impz_1} 
\Phi_2^{(\pm)[m, s+mp]}(\tau, z_1, z_2, t)$
\end{enumerate}
\vspace{0mm}
\item[{\rm 2)}] $\Phi^{(\pm)[m, s]}(\tau, z_1, z_2+p\tau, t) 
= e^{-2\pi impz_1} 
\Phi^{(\pm)[m, s+mp]}(\tau, \, z_1, \, z_2, \, t)$ \quad 
if \,\ $(\pm 1)^p =1$.
\end{enumerate}
\end{lemma}

\begin{proof} The proof is similar with that of Lemma 2.4 in \cite{W2022b}.
\end{proof}

\vspace{0mm}

\begin{rem} 
\label{m1:rem:2022-512a}
Lemmas 2.2 and 2.5 in \cite{W2022a} and Lemma 2.1 and formula 
(2.1) in \cite{W2022b} hold as well for $\Phi^{(\pm)[m,s]}$
if we replace $\Phi^{[m,s]}$ and $\theta_{j,m}$
with $\Phi^{(\pm)[m,s]}$ and $\theta_{j,m}^{(\pm)}$ respectively.
\end{rem}

\medskip

We note also that the following formulas hold.

\begin{lemma} 
\label{m1:lemma:2022-524a}
Let $m \in \frac12 \nnn$, $s \in \frac12 \zzz$ and $a \in \zzz$. Then
\begin{enumerate}
\item[{\rm 1)}]
\begin{enumerate}
\item[{\rm (i)}] $\Phi^{(\pm)[m,s]}_1
(\tau, \, z_1+a\tau, \, z_2-a\tau, \, t)
= \, (\pm 1)^a e^{2\pi ima(z_1-z_2)} q^{ma^2} 
\Phi^{(\pm)[m,s-2am]}_1(\tau, z_1,z_2,t) $
\item[{\rm (ii)}] $\Phi^{(\pm)[m,s]}_2
(\tau, \, z_1+a\tau, \, z_2-a\tau, \, t)
= \, 
(\pm 1)^a e^{2\pi ima(z_1-z_2)} q^{ma^2} 
\Phi^{(\pm)[m,s-2am]}_2(\tau, z_1,z_2,t) $
\end{enumerate}
\item[{\rm 2)}] \,\ $\Phi^{(\pm)[m,s]}
(\tau, \, z_1+a\tau, \, z_2-a\tau, \, t)
= \,\ (\pm 1)^a \, e^{2\pi ima(z_1-z_2)} \, q^{ma^2} \, 
\Phi^{(\pm)[m,s-2am]}(\tau, z_1,z_2,t) $

\vspace{-5mm}

\begin{equation} 
\label{m1:eqn:2022-524b}
\end{equation}
\end{enumerate}
\end{lemma}

\begin{proof} 1) (i) By definition \eqref{n3-2:eqn:2022-111a} of 
$\Phi^{(\pm)[m,s]}_1$, we have
{\allowdisplaybreaks
\begin{eqnarray*}
\Phi^{(\pm)[m,s]}_1
(\tau, \, z_1+a\tau, \, z_2-a\tau, \, 0)
&=& 
\sum_{j \in \zzz}(\pm 1)^j \, \frac{
e^{2\pi imj(z_1+z_2)+2\pi is(z_1+a\tau)} \, q^{mj^2+sj}
}{1-e^{2\pi i(z_1+a\tau)} q^j}
\\[1mm]
&=& 
\sum_{j \in \zzz}(\pm 1)^j \, \frac{
e^{2\pi imj(z_1+z_2)+2\pi isz_1} \, q^{mj^2+s(j+a)}
}{1-e^{2\pi iz_1} q^{j+a}}
\end{eqnarray*}}
Putting $j=k-a$, this becomes
{\allowdisplaybreaks
\begin{eqnarray*}
&=&
\sum_{k \in \zzz}(\pm 1)^{k-a} \, \frac{
e^{2\pi im(k-a)(z_1+z_2)+2\pi isz_1} \, q^{m(k-a)^2+sk}
}{1-e^{2\pi iz_1} q^k} 
\\[1mm]
&=&
(\pm 1)^a e^{2\pi ima(z_1-z_2)} q^{ma^2} 
\sum_{k \in \zzz} (\pm 1)^k 
\frac{e^{2\pi imk(z_1+z_2)+2\pi i(s-2am)z_1} 
q^{mk^2+(s-2am)k}
}{1-e^{2\pi iz_1} q^k}
\\[1mm]
&=&
(\pm 1)^a e^{2\pi ima(z_1-z_2)} q^{ma^2} \, 
\Phi^{(\pm)[m,s-2am]}_1(\tau, z_1,z_2,0)
\end{eqnarray*}}

\vspace{-4mm}

\noindent
proving (i).  1) (ii) is obtained from (i) by using Lemma 2.2 in 
\cite{W2022a} with Remark \ref{m1:rem:2022-512a}. \\ 
2) follows immediately from 1).
\end{proof}

\vspace{0mm}

\begin{lemma} \,\ 
\label{m1:lemma:2022-524b}
Let $m \in \frac12 \nnn$, $s \in \frac12 \zzz$ and $a \in \zzz_{\geq 0}$. 
Then
\begin{enumerate}
\item[{\rm 1)}] \,\ $\Phi^{(\pm)[m,s]}(\tau, \, z_1+a\tau, \, z_2-a\tau, \, 0) $
{\allowdisplaybreaks
\begin{eqnarray}
&=& (\pm 1)^a \, e^{2\pi ima(z_1-z_2)} \, q^{ma^2} \, \bigg\{
\Phi^{(\pm)[m,s]}(\tau, z_1,z_2,0)
\nonumber
\\[1mm]
& &
- \sum_{\substack{k \, \in \zzz \\[1mm] 1 \, \leq \, k \, \leq \, 2am}}
e^{-\pi i(k-s)(z_1-z_2)} \, 
q^{- \frac{(k-s)^2}{4m}} \, 
\big[\theta^{(\pm)}_{k-s, \, m} - \theta^{(\pm)}_{-(k-s), \, m}
\big] (\tau, z_1+z_2) \bigg\}
\label{m1:eqn:2022-524c}
\end{eqnarray}}
\item[{\rm 2)}]
\begin{enumerate}
\item[{\rm (i)}] \,\ $\Phi^{(\pm)[m,s]}_1(\tau, \, z_1+a\tau, \, z_2-a\tau, \, 0) $
{\allowdisplaybreaks
\begin{eqnarray*}
&=& (\pm 1)^a \, e^{2\pi ima(z_1-z_2)} \, q^{ma^2} \, \bigg\{
\Phi^{(\pm)[m,s]}_1(\tau, z_1,z_2,0)
\\[1mm]
& &
- \sum_{\substack{k \, \in \zzz \\[1mm] 1 \, \leq \, k \, \leq \, 2am}}
e^{-\pi i(k-s)(z_1-z_2)} \, 
q^{- \frac{(k-s)^2}{4m}} \, 
\theta^{(\pm)}_{k-s, \, m} (\tau, z_1+z_2) \bigg\}
\end{eqnarray*}}
\item[{\rm (ii)}] \,\ $\Phi^{(\pm)[m,s]}_2(\tau, \, z_1+a\tau, \, z_2-a\tau, \, 0) $
{\allowdisplaybreaks
\begin{eqnarray*}
&=& (\pm 1)^a \, e^{2\pi ima(z_1-z_2)} \, q^{ma^2} \, \bigg\{
\Phi^{(\pm)[m,s]}_2(\tau, z_1,z_2,0)
\\[1mm]
& &
- \sum_{\substack{k \, \in \zzz \\[1mm] 1 \, \leq \, k \, \leq \, 2am}}
e^{-\pi i(k-s)(z_1-z_2)} \, 
q^{- \frac{(k-s)^2}{4m}} \, 
\theta^{(\pm)}_{-(k-s), \, m}(\tau, z_1+z_2) \bigg\}
\end{eqnarray*}}
\end{enumerate}
\end{enumerate}
\end{lemma}

\begin{proof} 1) \,\ By Lemma 2.1 in \cite{W2022b} with Remark 
\ref{m1:rem:2022-512a}, we have
{\allowdisplaybreaks
\begin{eqnarray*}
& & \hspace{-5mm}
\Phi^{(\pm)[m,s-a]}(\tau, z_1,z_2,0) \,\ - \,\ 
\Phi^{(\pm)[m,s]}(\tau, z_1,z_2,0)
\\[1mm]
&= &
\sum_{\substack{k \, \in \zzz \\[1mm] -a \, \leq \, k \, \leq \, -1}}
e^{\pi i(s+k)(z_1-z_2)} \, 
q^{- \frac{(s+k)^2}{4m}} \, \big[
\theta^{(\pm)}_{s+k, \, m} - \theta^{(\pm)}_{-(s+k), \, m}\big]
(\tau, z_1+z_2)
\end{eqnarray*}}
Letting \, $a \rightarrow  2am$, we have
{\allowdisplaybreaks
\begin{eqnarray}
& & \hspace{-5mm}
\Phi^{(\pm)[m,s-2am]}(\tau, z_1,z_2,0) \, - \, 
\Phi^{(\pm)[m,s]}(\tau, z_1,z_2,0)
\nonumber
\\[1mm]
&= &
\sum_{\substack{k \, \in \zzz \\[1mm] -2am \, \leq \, k \, \leq \, -1}}
e^{\pi i(s+k)(z_1-z_2)} 
q^{- \frac{(s+k)^2}{4m}} \big[
\theta^{(\pm)}_{s+k, \, m} - \theta^{(\pm)}_{-(s+k), \, m}\big]
(\tau, z_1+z_2)
\nonumber
\\[1mm]
&=&
- \sum_{\substack{l \, \in \zzz \\[1mm] 1 \, \leq \, k \, \leq \, 2am}}
e^{\pi i(s-k)(z_1-z_2)} \, 
q^{- \frac{(s-k)^2}{4m}} 
\big[\theta^{(\pm)}_{k-s, \, m} - \theta^{(\pm)}_{-(k-s), \, m}
\big](\tau, z_1+z_2)
\label{m1:eqn:2022-524f}
\end{eqnarray}}
Substituting this equation \eqref{m1:eqn:2022-524f} into 
\eqref{m1:eqn:2022-524b}, we obtain the formula 
\eqref{m1:eqn:2022-524c}.

\medskip

The proof of 2) is quite similar.
\end{proof}

\vspace{1mm}

The following Lemma \ref{m1:lemma:2022-830a} is obtained 
immediately from Note \ref{m1:note:2022-830a} and 
Lemma \ref{m1:lemma:2022-512b}.

\vspace{0mm}

\begin{lemma}
\label{m1:lemma:2022-830a}
Let $m \in \frac12 \nnn, \,\ s \in \frac12 \zzz$ and $p \in \zzz$. Then 
\begin{enumerate}
\item[{\rm 1)}]
\begin{enumerate}
\item[{\rm (i)}] \quad $\Phi^{[m,s]}(\tau, z, -z+p\tau,t) \, = \, 0$
\item[{\rm (ii)}] \quad $\Phi^{(-)[m,s]}(\tau, z, -z+2p\tau,t) \, = \, 0$
\end{enumerate}
\end{enumerate}
namely,
\begin{enumerate}
\item[{\rm 2)}]
\begin{enumerate}
\item[{\rm (i)}] \quad $\Phi^{[m,s]}_1(\tau, z, -z+p\tau,t) \, = \, 
\Phi^{[m,s]}_2(\tau, z, -z+p\tau,t)$
\item[{\rm (ii)}] \quad $\Phi^{(-)[m,s]}_1(\tau, z, -z+2p\tau,t) \, = \, 
\Phi^{(-)[m,s]}_2(\tau, z, -z+2p\tau,t)$
\end{enumerate}
\end{enumerate}
\end{lemma}

\section{Explicit formula for $\Phi^{(-)[m,\frac12]}$}

Under the same setting and notations with Section 3 in \cite{W2022b}, 
we consider $\widehat{D}(2,1; a= \frac{-m}{m+1})$ where 
$m \in \frac12 \nnn$, and compute 
{\allowdisplaybreaks
\begin{eqnarray*}
F^{(\natural)}_{a\Lambda_0} &:=& \sum_{j,k \in \zzz}
(-1)^k \, t_{j\alpha_2^{\vee}+k\alpha_3^{\vee}}
\left(\frac{e^{a\Lambda_0}}{1-e^{-\alpha_1}}\right)
\\[1mm]
A^{(\natural)}_{a\Lambda_0} &:=& \sum_{w \in \overline{W}}
\varepsilon(w) \, w\big(F^{(\natural)}_{a\Lambda_0}\big)
\end{eqnarray*}}

\vspace{-3mm}

\noindent
Then by just similar arguments with those in the proof of Lemma 3.1 
and Propositions 3.1 and 3.2 and Corollary 3.1 in \cite{W2022b}, 
we obtain the following Lemma \ref{m1:lemma:2022-523a} and 
Propositions \ref{m1:prop:2022-523a} $\sim$ \ref{m1:prop:2022-523c}:

\begin{lemma} 
\label{m1:lemma:2022-523a}
For $m \in \frac12 \nnn$ \, the following formula holds:
{\allowdisplaybreaks
\begin{eqnarray}
& & \,\
\sum_{j \in \zzz}(-1)^jq^{(m+1)j^2}
e^{-2\pi ij(z_1-z_2)+2\pi ijm(z_1-z_3)} \, 
\Phi^{(-)[m,0]}(\tau, \, z_1, \, -z_3+2j\tau, \, 0) 
\nonumber
\\[1mm]
& & \hspace{-3mm}
- \,\ \sum_{j \in \zzz}(-1)^jq^{(m+1)j^2}
e^{2\pi ij(z_1-z_2)+2\pi ijm(z_1-z_3)} 
\Phi^{(-)[m,0]}(\tau, \, z_2, \, z_1-z_2-z_3+2j\tau, \, 0) 
\nonumber
\\[1mm]
&=& \,\ 
\sum_{j \in \zzz}(-1)^jq^{(m+1)j^2}
e^{2\pi ij(z_1-z_2)-2\pi ijm(z_1-z_3)} \, 
\Phi^{[1,0]}(\tau, \, z_1, \, -z_2+2j\tau, \, 0) 
\nonumber
\\[1mm]
& & \hspace{-3mm}
- \,\ \sum_{j \in \zzz}(-1)^jq^{(m+1)j^2}
e^{2\pi ij(z_1-z_2)+2\pi ijm(z_1-z_3)} 
\Phi^{[1,0]}(\tau, \, z_3, \, z_1-z_2-z_3+2j\tau, \, 0) 
\label{m1:eqn:2022-523a} 
\end{eqnarray}}
\end{lemma}

\begin{prop} 
\label{m1:prop:2022-523a}
For $m \in \frac12 \nnn$ \, the following formula holds:
\begin{subequations}
{\allowdisplaybreaks
\begin{eqnarray}
& &
\theta_{0, \, m+1}^{(-)}\Big(\tau, \, \frac{(z_1-z_2)+m(z_1+z_3)}{m+1}\Big)
\, \Phi^{(-)[m,0]}(\tau, \, z_1, \, -z_3, \, 0) 
\nonumber
\\[2mm]
& & \hspace{-5mm}
- \,\ 
\theta_{0, \, m+1}^{(-)}\Big(\tau, \, \frac{(z_1-z_2)+m(z_1-2z_2-z_3)}{m+1}\Big)
\, \Phi^{(-)[m,0]}(\tau, \, z_2, \, z_1-z_2-z_3, \, 0)
\nonumber
\\[2mm]
& & \hspace{-5mm}
+ \, \sum_{j \, > \, 0} \, \sum_{k=1}^{2mj} 
q^{(m+1)j^2} \, q^{-\frac{k^2}{4m}} \, 
\omega^{[m]}_{j,k}(z_1, z_2, z_3) \, 
\big[\theta_{k,m}^{(-)}-\theta_{-k,m}^{(-)}\big](\tau, z_1-z_3)
\nonumber
\\[2mm]
&=&
i \, \theta_{0, m+1}^{(-)}\Big(\tau, \,\ \frac{(z_1+z_2) + m(z_1-z_3)}{m+1}\Big)
\,\ 
\frac{\eta(\tau)^3 \, \vartheta_{11}(\tau, \, z_1-z_2)}{
\vartheta_{11}(\tau, z_1) \, \vartheta_{11}(\tau, z_2)}
\nonumber
\\[2mm]
& &\hspace{-5mm}
+ \,\ i \, 
\theta_{0, m+1}^{(-)}\Big( \tau, \,\ \frac{(z_1-z_2-2z_3) \, + \, m(z_1-z_3)}{m+1}\Big)
\,
\frac{\eta(\tau)^3 \, \vartheta_{11}(\tau, \, z_1-z_2)}{
\vartheta_{11}(\tau, z_3) \, \vartheta_{11}(\tau, \, z_1-z_2-z_3)}
\label{m1:eqn:2022-523b1}
\end{eqnarray}}
\noindent
where
{\allowdisplaybreaks
\begin{eqnarray}
\omega^{[m]}_{j,k}(z_1, z_2, z_3)
&:=&
e^{2\pi ij(z_1-z_2)} \, e^{\pi i(2mj-k)(z_1-2z_2-z_3)}
\, + \, 
e^{-2\pi ij(z_1-z_2)} \, e^{-\pi i(2mj-k)(z_1-2z_2-z_3)}
\nonumber
\\[1mm]
& & \hspace{-5mm}
- \,\ 
e^{2\pi ij(z_1-z_2)} \, e^{\pi i(2mj-k)(z_1+z_3)}
\, - \, 
e^{-2\pi ij(z_1-z_2)} \, e^{-\pi i(2mj-k)(z_1+z_3)}
\label{m1:eqn:2022-523b2}
\end{eqnarray}}
\end{subequations}
\end{prop}

\begin{prop} 
\label{m1:prop:2022-523b}
For $m \in \nnn$, the following formula holds:
{\allowdisplaybreaks
\begin{eqnarray}
& & \hspace{-7mm}
\theta_{0, m+1}^{(-)}\Big(\tau, \frac{-\frac12+m(z_1+z_3)}{m+1} \Big) 
\bigg\{
\Phi^{(-)[m,0]}(\tau, \, z_1, \, -z_3, \, 0) 
-
\Phi^{(-)[m,0]}\Big(\tau, \, z_1+\frac12, \, -z_3-\frac12, \, 0\Big) \bigg\}
\nonumber
\\[1mm]
&=& \hspace{-2mm}
2i \eta(\tau)^2 \eta(2\tau)^2 \Bigg\{- 
\frac{\displaystyle 
\theta_{0, m+1}^{(-)}\Big(\tau, \frac{\frac12+z_1+z_3}{m+1}+ z_1-z_3\Big)
}{\vartheta_{11}(\tau, z_1) \vartheta_{10}(\tau, z_1)} 
+ 
\frac{\displaystyle 
\theta_{0, m+1}^{(-)}\Big(\tau, \frac{\frac12+z_1+z_3}{m+1}- z_1+z_3\Big)
}{\vartheta_{11}(\tau, z_3) \vartheta_{10}(\tau, z_3)} \Bigg\}
\nonumber
\\[1mm]
& & \hspace{-3mm}
+ \, 2 \, \sum_{j=1}^{\infty} \sum_{r=1}^j 
\sum_{\substack{k \in \zzz_{\rm odd} \\[1mm] 1 \leq k \leq m}}
(-1)^{j+r} q^{(m+1)j^2-\frac{1}{4m}(2m(j-r)+k)^2}
\big\{e^{\pi i(2mr-k)(z_1+z_3)}+e^{-\pi i(2mr-k)(z_1+z_3)}\big\}
\nonumber
\\[-7mm]
& & \hspace{50mm}
\times \, 
\big[\theta_{k,m}^{(-)}-\theta_{-k,m}^{(-)}](\tau, z_1-z_3)
\nonumber
\\[1mm]
& & \hspace{-3mm}
- \, 2 \, \sum_{j=1}^{\infty} \sum_{r=0}^{j-1} 
\sum_{\substack{k \in \zzz_{\rm odd} \\[1mm] 1 \leq k <m}}
(-1)^{j+r} q^{(m+1)j^2-\frac{1}{4m}(2m(j-r)-k)^2}
\big\{e^{\pi i(2mr+k)(z_1+z_3)}+e^{-\pi i(2mr+k)(z_1+z_3)}\big\}
\nonumber
\\[-7mm]
& & \hspace{50mm}
\times \, 
\big[\theta_{k,m}^{(-)}-\theta_{-k,m}^{(-)}](\tau, z_1-z_3)
\label{m1:eqn:2022-523c}
\end{eqnarray}}
\end{prop}

\begin{prop} 
\label{m1:prop:2022-523c}
For $m \in \frac12 \nnn$ the following formulas hold:
{\allowdisplaybreaks
\begin{eqnarray}
& & \hspace{-5mm}
\theta_{0, 2m+1}^{(-)}
\Big(\tau, \dfrac{-\frac12+2m(z_1-z_2)}{2m+1} \Big) \, 
\Phi^{(-)[m, \frac12]}(2\tau, \, 2z_1, \, 2z_2, \, 0)
\nonumber
\\[2mm]
& & \hspace{-10mm}
= \,\ 
-i \, \eta(2\tau)^3 \, \Bigg\{
\frac{\displaystyle 
\theta_{0, 2m+1}^{(-)}\Big(\tau, \, \frac{\frac12+z_1-z_2}{2m+1}+ z_1+z_2\Big)
}{\vartheta_{11}(2\tau, 2z_1)} 
+ 
\frac{\displaystyle 
\theta_{0, 2m+1}^{(-)}\Big(\tau, \, \frac{\frac12+z_1-z_2}{2m+1}- z_1-z_2\Big)
}{\vartheta_{11}(2\tau, 2z_2)} \Bigg\}
\nonumber
\\[1mm]
& & \hspace{-10mm}
+ \, \sum_{j=1}^{\infty} \sum_{r=1}^j 
\sum_{\substack{k \in \zzz_{\rm odd} \\[1mm] 1 \leq k \leq 2m}}
(-1)^{j+r} q^{(2m+1)j^2-\frac{1}{8m}(4m(j-r)+k)^2}
\big\{e^{\pi i(4mr-k)(z_1-z_2)}+e^{-\pi i(4mr-k)(z_1-z_2)}\big\}
\nonumber
\\[-7mm]
& & \hspace{50mm}
\times \, 
\big[\theta_{k,2m}^{(-)}-\theta_{-k,2m}^{(-)}](\tau, z_1+z_2)
\nonumber
\\[1mm]
& & \hspace{-10mm}
- \, \sum_{j=1}^{\infty} \sum_{r=0}^{j-1} 
\sum_{\substack{k \in \zzz_{\rm odd} \\[1mm] 1 \leq k < 2m}}
(-1)^{j+r} q^{(2m+1)j^2-\frac{1}{8m}(4m(j-r)-k)^2}
\big\{e^{\pi i(4mr+k)(z_1-z_2)}+e^{-\pi i(4mr+k)(z_1-z_2)}\big\}
\nonumber
\\[-7mm]
& & \hspace{50mm}
\times \, 
\big[\theta_{k,2m}^{(-)}-\theta_{-k,2m}^{(-)}](\tau, z_1+z_2)
\label{m1:eqn:2022-523d}
\end{eqnarray}}
\end{prop}

\medskip

Replacing $\tau$ and $z_i$ with $\frac12 \tau$ and $\frac12 z_i$
respectively and using 
$\theta_{j,m}^{(\pm)}(\frac{\tau}{2}, \frac{z}{2}) = 
\theta_{\frac{j}{2}, \frac{m}{2}}^{(\pm)}(\tau, z)$, this formula 
\eqref{m1:eqn:2022-523d} is rewritten as follows:
{\allowdisplaybreaks
\begin{eqnarray}
& & \hspace{-5mm}
\theta_{0, m+\frac12}^{(-)}
\Big(\tau, \dfrac{-1+2m(z_1-z_2)}{2m+1} \Big) \, 
\Phi^{(-)[m, \frac12]}(\tau, \, z_1, \, z_2, \, 0)
\nonumber
\\[1mm]
& & \hspace{-10mm}
= \,\ 
-i \, \eta(\tau)^3 \, \Bigg\{
\frac{\displaystyle 
\theta_{0, m+\frac12}^{(-)}
\Big(\tau, \, \frac{1+z_1-z_2}{2m+1}+ z_1+z_2\Big)
}{\vartheta_{11}(\tau, z_1)} 
+ 
\frac{\displaystyle 
\theta_{0, m+\frac12}^{(-)}
\Big(\tau, \, \frac{1+z_1-z_2}{2m+1}- z_1-z_2\Big)
}{\vartheta_{11}(\tau, z_2)} \Bigg\}
\nonumber
\\[1mm]
& & \hspace{-10mm}
+ \, \sum_{j=1}^{\infty} \sum_{r=1}^j 
\sum_{\substack{k \in \zzz_{\rm odd} \\[1mm] 1 \leq k \leq 2m}}
(-1)^{j+r} q^{(m+\frac12)j^2-\frac{1}{16m}(4m(j-r)+k)^2}
\big\{e^{\frac{\pi i}{2}(4mr-k)(z_1-z_2)}
+e^{-\frac{\pi i}{2}(4mr-k)(z_1-z_2)}\big\}
\nonumber
\\[-8mm]
& & \hspace{50mm}
\times \, 
\big[\theta_{\frac{k}{2},m}^{(-)}
-\theta_{-\frac{k}{2},m}^{(-)}](\tau, z_1+z_2)
\nonumber
\\[1mm]
& & \hspace{-10mm}
- \, \sum_{j=1}^{\infty} \sum_{r=0}^{j-1} 
\sum_{\substack{k \in \zzz_{\rm odd} \\[1mm] 1 \leq k < 2m}}
(-1)^{j+r} q^{(m+\frac12)j^2-\frac{1}{16m}(4m(j-r)-k)^2}
\big\{e^{\frac{\pi i}{2}(4mr+k)(z_1-z_2)}
+e^{-\frac{\pi i}{2}(4mr+k)(z_1-z_2)}\big\}
\nonumber
\\[-8mm]
& & \hspace{50mm}
\times \, 
\big[\theta_{\frac{k}{2},m}^{(-)}
-\theta_{-\frac{k}{2},m}^{(-)}](\tau, z_1+z_2)
\label{m1:eqn:2022-524a}
\end{eqnarray}}

Applying this formula to $(z_1,z_2)= 
(\frac{z}{2}+\frac{\tau}{2}-\frac12+p\tau, 
\frac{z}{2}-\frac{\tau}{2}+\frac12-p\tau)$, \, 
$(\frac{z}{2}+\frac{\tau}{2}+p\tau, \frac{z}{2}-\frac{\tau}{2}-p\tau)$
and 
$(\frac{z}{2}-\frac12+p\tau, \frac{z}{2}+\frac12-p\tau)$, 
we obtain the following:

\begin{lemma}
\label{m1:lemma:2022-524c}
Let $m \in \frac12 \nnn_{\rm odd}$ and $p \in \zzz_{\geq 0}$. Then
\begin{enumerate}
\item[{\rm 1)}] \,\ $q^{-\frac{m^2(2p+1)^2}{2(2m+1)}} \, 
\theta_{(2p+1)m, m+\frac12}^{(-)}(\tau,0) \, 
\Phi^{(-)[m, \frac12]}\Big(\tau, \,\ 
\dfrac{z}{2}+\dfrac{\tau}{2}-\dfrac12+p\tau, \,\ 
\dfrac{z}{2}-\dfrac{\tau}{2}+\dfrac12-p\tau, \,\ 0\Big)$
{\allowdisplaybreaks
\begin{eqnarray*}
& & \hspace{-8mm}
= \,\ 
- \,\ i \, \eta(\tau)^3 \,\ q^{\frac{m(2p+1)^2}{4(2m+1)}} \cdot 
\frac{\theta_{p+\frac12,m+\frac12}^{(-)}(\tau,z)
\, - \, 
\theta_{-(p+\frac12),m+\frac12}^{(-)}(\tau,z)
}{\theta_{0, \frac12}(\tau,z)}
\\[1mm]
& & \hspace{-3mm}
+ \, \bigg[\sum_{\substack{j, \, r \in \zzz \\[1mm] 0 \leq r <j}} 
- \sum_{\substack{j, \, r \in \zzz \\[1mm] j \leq r<0}} 
\bigg] \, 
\sum_{\substack{k \in \frac12 \zzz_{\rm odd} \\[1mm] 0< k \leq m}} 
\hspace{-2mm}
(-1)^{j} e^{\pi ik} \, 
q^{(m+\frac12)j^2+(2p+1)mj} 
q^{-\frac{1}{4m}(2mr+k)^2-\frac{2p+1}{2}(2mr+k)}
\\[-8mm]
& & \hspace{50mm}
\times \, 
\big[\theta_{k,m}^{(-)}-\theta_{-k,m}^{(-)}\big](\tau, z)
\\[1mm]
& & \hspace{-3mm}
+ \,\ \bigg[
\sum_{\substack{j, \, r \in \zzz \\[1mm] 0 \leq r  \leq j}} 
- \sum_{\substack{j, \, r \in \zzz \\[1mm] j < r <0}} 
\bigg] \, 
\sum_{\substack{k \in \frac12 \zzz_{\rm odd} \\[1mm] 0< k < m}}
(-1)^{j} e^{\pi ik} \, 
q^{(m+\frac12)j^2+(2p+1)mj} 
q^{-\frac{1}{4m}(2mr-k)^2-\frac{2p+1}{2}(2mr-k)}
\\[-8mm]
& & \hspace{50mm}
\times \, 
\big[\theta_{k,m}^{(-)}-\theta_{-k,m}^{(-)}\big](\tau, z)
\\[2mm]
& & \hspace{-3mm}
- \,\ \sum_{j\in \zzz} (-1)^{j} \, q^{(m+\frac12)j^2+(2p+1)mj} 
\sum_{\substack{k \in \frac12 \zzz_{\rm odd} \\[1mm] 0< k <m}}
e^{\pi ik} \, 
q^{-\frac{1}{4m}k^2+(2p+1)\frac{k}{2}}
\big[\theta_{k,m}^{(-)}-\theta_{-k,m}^{(-)}\big](\tau, z)
\end{eqnarray*}}
\item[{\rm 2)}] $q^{-\frac{m^2(2p+1)^2}{2(2m+1)}}  
\theta_{(2p+1)m, m+\frac12}(\tau,0) \cdot 
\Phi^{(-)[m, \frac12]}\Big(\tau, \, 
\dfrac{z}{2}+\dfrac{\tau}{2}+p\tau, \, 
\dfrac{z}{2}-\dfrac{\tau}{2}-p\tau, \, 0\Big)$
{\allowdisplaybreaks
\begin{eqnarray*}
& & \hspace{-8mm}
= \,\ 
(-1)^p \, \eta(\tau)^3 \,\ q^{\frac{m(2p+1)^2}{4(2m+1)}} \cdot 
\frac{\theta_{p+\frac12,m+\frac12}(\tau,z)
\, - \, 
\theta_{-(p+\frac12),m+\frac12}(\tau,z)
}{\theta_{0, \frac12}^{(-)}(\tau,z)}
\\[1mm]
& & \hspace{-3mm}
+ \, \bigg[\sum_{\substack{j, \, r \in \zzz \\[1mm] 0 \leq r<j}} 
- \sum_{\substack{j, \, r \in \zzz \\[1mm] j \leq r <0}} 
\bigg] \, 
\sum_{\substack{k \in \frac12 \zzz_{\rm odd} \\[1mm] 0< k \leq m}} 
\hspace{-2mm}
(-1)^r \, 
q^{(m+\frac12)j^2+(2p+1)mj} 
q^{-\frac{1}{4m}(2mr+k)^2-\frac{2p+1}{2}(2mr+k)}
\\[-8mm]
& & \hspace{50mm}
\times \, 
\big[\theta_{k,m}^{(-)}-\theta_{-k,m}^{(-)}\big](\tau, z)
\\[2mm]
& & \hspace{-3mm}
- \, \bigg[
\sum_{\substack{j, \, r \in \zzz \\[1mm] 0 \leq r \leq j}} 
- \sum_{\substack{j, \, r \in \zzz \\[1mm] j< r <0}} 
\bigg] \, 
\sum_{\substack{k \in \frac12 \zzz_{\rm odd} \\[1mm] 0<k<m}}
(-1)^r \, 
q^{(m+\frac12)j^2+(2p+1)mj} 
q^{-\frac{1}{4m}(2mr-k)^2-\frac{2p+1}{2}(2mr-k)}
\\[-8mm]
& & \hspace{50mm}
\times \, 
\big[\theta_{k,m}^{(-)}-\theta_{-k,m}^{(-)}\big](\tau, z)
\\[2.5mm]
& & \hspace{-3mm}
+ \sum_{j\in \zzz} q^{(m+\frac12)j^2+(2p+1)mj} \hspace{-2mm}
\sum_{\substack{k \in \frac12 \zzz_{\rm odd} \\[1mm] 0< k<m}}
\hspace{-2mm}
q^{-\frac{1}{4m}k^2+(2p+1)\frac{k}{2}}
\big[\theta_{k,m}^{(-)}-\theta_{-k,m}^{(-)}\big](\tau, z)
\end{eqnarray*}}
\item[{\rm 3)}] $q^{-\frac{2m^2p^2}{2m+1}} \, 
\theta_{2pm, m+\frac12}^{(-)}(\tau,0) \cdot 
\Phi^{(-)[m, \frac12]}\Big(\tau, \,\ 
\dfrac{z}{2}-\dfrac12+p\tau, \,\ 
\dfrac{z}{2}+\dfrac12-p\tau, \,\ 0\Big)$
{\allowdisplaybreaks
\begin{eqnarray*}
& & \hspace{-8mm}
= \,\ 
- \,\ i \, \eta(\tau)^3 \,\ q^{\frac{mp^2}{2m+1}} \cdot 
\frac{\theta_{p,m+\frac12}^{(-)}(\tau,z)
\, - \, 
\theta_{-p,m+\frac12}^{(-)}(\tau,z)
}{\theta_{\frac12, \frac12}(\tau,z)}
\\[1mm]
& & \hspace{-3mm}
+ \, \bigg[\sum_{\substack{j, \, r \in \zzz \\[1mm] 0 \leq r <j}} 
- \sum_{\substack{j, \, r \in \zzz \\[1mm] j \leq r <0}} 
\bigg] \, 
\sum_{\substack{k \in \frac12 \zzz_{\rm odd} \\[1mm] 0< k \leq m}} 
\hspace{-2mm}
(-1)^{j} e^{\pi ik} \, 
q^{(m+\frac12)j^2+2pmj} 
q^{-\frac{1}{4m}(2mr+k)^2-p(2mr+k)}
\\[-8mm]
& & \hspace{50mm}
\times \, 
\big[\theta_{k,m}^{(-)}-\theta_{-k,m}^{(-)}\big](\tau, z)
\\[2mm]
& & \hspace{-3mm}
+ \, \bigg[
\sum_{\substack{j, \, r \in \zzz \\[1mm] 0 \leq r \leq j}} 
- \sum_{\substack{j, \, r \in \zzz \\[1mm] j< r<0}} 
\bigg] \, 
\sum_{\substack{k \in \frac12 \zzz_{\rm odd} \\[1mm] 0< k<m}}
(-1)^{j} e^{\pi ik} \, 
q^{(m+\frac12)j^2+2pmj} 
q^{-\frac{1}{4m}(2mr-k)^2-p(2mr-k)}
\\[-8mm]
& & \hspace{50mm}
\times \, 
\big[\theta_{k,m}^{(-)}-\theta_{-k,m}^{(-)}\big](\tau, z)
\\[2.5mm]
& & \hspace{-3mm}
- \,\ \sum_{j\in \zzz} (-1)^{j} \, q^{(m+\frac12)j^2+2pmj} 
\sum_{\substack{k \in \frac12 \zzz_{\rm odd} \\[1mm] 0<k <m}}
e^{\pi ik} \, 
q^{-\frac{1}{4m}k^2+pk} 
\big[\theta_{k,m}^{(-)}-\theta_{-k,m}^{(-)}\big](\tau, z)
\end{eqnarray*}}
\end{enumerate}
\end{lemma}

\begin{proof} 1) \,\ Letting $(z_1,z_2)= 
(\frac{z}{2}+\frac{\tau}{2}-\frac12+p\tau, 
\frac{z}{2}-\frac{\tau}{2}+\frac12-p\tau)$ in \eqref{m1:eqn:2022-524a}, 
we have 
{\allowdisplaybreaks
\begin{eqnarray}
& & \hspace{-15mm}
\theta_{0, m+\frac12}^{(-)}
\Big(\tau, \dfrac{-1+2m((2p+1)\tau-1)}{2m+1} \Big) 
\nonumber
\\[1mm]
& & \hspace{-10mm}
\times \, 
\Phi^{(-)[m, \frac12]}\Big(\tau, \, 
\frac{z}{2}+\frac{\tau}{2}-\frac12+p\tau, \,
\frac{z}{2}-\frac{\tau}{2}+\frac12-p\tau, \, 0\Big)
\,\ = \,\ {\rm (I)} + {\rm (II)}
\label{m1:eqn:2022-526a}
\end{eqnarray}}
where 
{\allowdisplaybreaks
\begin{eqnarray*}
& & \hspace{-10mm}
{\rm (I)} := \, 
-i \, \eta(\tau)^3 \, \Bigg\{ \, 
\frac{\displaystyle 
\theta_{0, m+\frac12}^{(-)}
\Big(\tau, \, \frac{(2p+1)\tau}{2m+1}+ z\Big)
}{\vartheta_{11}(\tau, \frac{z}{2}+\frac{\tau}{2}-\frac12+p\tau)} 
+ \, 
\frac{\displaystyle 
\theta_{0, m+\frac12}^{(-)}
\Big(\tau, \, \frac{(2p+1)\tau}{2m+1}- z\Big)
}{\vartheta_{11}(\tau, \frac{z}{2}-\frac{\tau}{2}+\frac12-p\tau)} 
\Bigg\}
\\[1mm]
& & \hspace{-10mm}
{\rm (II)} := \, 
\sum_{j=1}^{\infty} \sum_{r=1}^j 
\sum_{\substack{k \in \zzz_{\rm odd} \\[1mm] 1 \leq k \leq 2m}}
(-1)^{j+r} q^{(m+\frac12)j^2-\frac{1}{16m}(4m(j-r)+k)^2}
\\[2mm]
& & \hspace{-5mm}
\times \, 
\big\{e^{\frac{\pi i}{2}(4mr-k)((2p+1)\tau-1)}
+e^{-\frac{\pi i}{2}(4mr-k)((2p+1)\tau-1)}\big\}
\big[\theta_{\frac{k}{2},m}^{(-)}
-\theta_{-\frac{k}{2},m}^{(-)}](\tau, z)
\\[2mm]
& & \hspace{-7mm}
- \, \sum_{j=1}^{\infty} \sum_{r=0}^{j-1} 
\sum_{\substack{k \in \zzz_{\rm odd} \\[1mm] 1 \leq k < 2m}}
(-1)^{j+r} q^{(m+\frac12)j^2-\frac{1}{16m}(4m(j-r)-k)^2}
\\[2mm]
& & \hspace{-5mm}
\times \, 
\big\{e^{\frac{\pi i}{2}(4mr+k)((2p+1)\tau-1)}
+e^{-\frac{\pi i}{2}(4mr+k)((2p+1)\tau-1)}\big\}
\big[\theta_{\frac{k}{2},m}^{(-)}
-\theta_{-\frac{k}{2},m}^{(-)}](\tau, z)
\end{eqnarray*}}

\noindent
Using Lemma \ref{m1:lemma:2022-526a} and Note \ref{m1:note:2022-526a}, 
the LHS of \eqref{m1:eqn:2022-526a} and (I) become as follows:
\begin{subequations}
{\allowdisplaybreaks
\begin{eqnarray}
\text{LHS of \eqref{m1:eqn:2022-526a}} 
&=&
q^{- \frac{m^2(2p+1)^2}{2(2m+1)}} 
\theta_{(2p+1)m, m+\frac12}^{(-)}(\tau,0)
\label{m1:eqn:2022-526b} 
\\[1mm] 
{\rm (I)} \hspace{10mm} 
&=&
-i \, q^{\frac{m(2p+1)^2}{4(2m+1)}}
\eta(\tau)^3
\frac{\big[\theta_{p+\frac12, m+\frac12}^{(-)}
-
\theta_{-(p+\frac12), m+\frac12}^{(-)}\big](\tau, z)
}{\theta_{0, \frac12}(\tau,z)}
\label{m1:eqn:2022-526c} 
\end{eqnarray}}
\end{subequations}
We compute (II) by ptting $k'=\frac12 k$:
{\allowdisplaybreaks
\begin{eqnarray*}
& & \hspace{-10mm}
{\rm (II)} \, = \,  
\sum_{j=1}^{\infty} \sum_{r=1}^j 
\sum_{\substack{k' \in \frac12 \zzz_{\rm odd} \\[1mm] 0 < k' \leq m}}
(-1)^{j+r} q^{(m+\frac12)j^2-\frac{1}{4m}(2m(j-r)+k')^2}
\\[2mm]
& & 
\times \, \big\{
q^{\frac12(2mr-k')(2p+1)}e^{-\pi i(2mr-k')}
+
q^{-\frac12(2mr-k')(2p+1)}e^{\pi i(2mr-k')}
\big\}
\big[\theta_{k',m}^{(-)}-\theta_{-k',m}^{(-)}](\tau, z)
\nonumber
\\[2mm]
& & \hspace{-5mm}
- \,\ \sum_{j=1}^{\infty} \sum_{r=0}^{j-1} 
\sum_{\substack{k' \in \frac12 \zzz_{\rm odd} \\[1mm] 0 < k' < m}}
(-1)^{j+r} q^{(m+\frac12)j^2-\frac{1}{4m}(2m(j-r)-k')^2}
\\[2mm]
& & 
\times \, \big\{
q^{\frac12(2mr+k')(2p+1)}e^{-\pi i(2mr+k')}
+
q^{-\frac12(2mr+k')(2p+1)}e^{\pi i(2mr+k')}\big\}
\big[\theta_{k',m}^{(-)}-\theta_{-k',m}^{(-)}](\tau, z)
\\[2mm]
&=&
\sum_{j=1}^{\infty} \sum_{r=1}^j 
\sum_{\substack{k \, \in \, \frac12 \zzz_{\rm odd} \\[1mm] 0< k \leq m}}
(-1)^{j} \, e^{\pi ik} 
q^{(m+\frac12)j^2}  q^{-\frac{1}{4m}(2m(j-r)+k)^2}
\big\{q^{\frac{2p+1}{2}(2mr-k)} - q^{-\frac{2p+1}{2}(2mr-k)}\big\}
\\[-8mm]
& & \hspace{50mm}
\times \,\ 
\big[\theta_{k,m}^{(-)}-\theta_{-k,m}^{(-)}\big](\tau, z)
\\[2mm]
& & \hspace{-5mm}
+ \sum_{j=1}^{\infty} \sum_{r=0}^{j-1} 
\sum_{\substack{k \, \in \, \frac12 \zzz_{\rm odd} \\[1mm] 0< k <m}}
(-1)^{j} \, e^{\pi ik} 
q^{(m+\frac12)j^2} \, q^{-\frac{1}{4m}(2m(j-r)-k)^2}
\big\{q^{\frac{2p+1}{2}(2mr+k)} - q^{-\frac{2p+1}{2}(2mr+k)}\big\}
\\[-8mm]
& & \hspace{50mm}
\times \,\ 
\big[\theta_{k,m}^{(-)}-\theta_{-k,m}^{(-)}\big](\tau, z)
\end{eqnarray*}}
Putting $r=j-r'$, this equation is rewritten as follows:
{\allowdisplaybreaks
\begin{eqnarray*}
& & \hspace{-12mm}
{\rm (II)} \, = \,
\sum_{j=1}^{\infty} \sum_{r'=0}^{j-1} 
\sum_{\substack{k \, \in \, \frac12 \zzz_{\rm odd} \\[1mm] 0< k \leq m}}
(-1)^{j} \, e^{\pi ik} \, 
q^{(m+\frac12)j^2} \, q^{-\frac{1}{4m}(2mr'+k)^2}
\nonumber
\\[2mm]
& & \hspace{-5mm}
\times \big\{
q^{(2p+1)mj} q^{-\frac{2p+1}{2}(2mr'+k)} - 
q^{-(2p+1)mj} q^{\frac{2p+1}{2}(2mr'+k)}
\big\}
\big[\theta_{k,m}^{(-)}-\theta_{-k,m}^{(-)}\big](\tau, z)
\nonumber
\\[2mm]
& & \hspace{-7mm}
+ \,\ \sum_{j=1}^{\infty} \sum_{r'=1}^{j} 
\sum_{\substack{k \, \in \, \frac12 \zzz_{\rm odd} \\[1mm] 0< k < m}}
(-1)^{j} \, e^{\pi ik} \, 
q^{(m+\frac12)j^2} \, q^{-\frac{1}{4m}(2mr'-k)^2}
\nonumber
\\[2mm]
& & \hspace{-5mm}
\times \big\{
q^{(2p+1)mj} q^{-\frac{2p+1}{2}(2mr'-k)} -
q^{-(2p+1)mj} q^{\frac{2p+1}{2}(2mr'-k)}
\big\}
\big[\theta_{k,m}^{(-)}-\theta_{-k,m}^{(-)}\big](\tau, z)
\\[2mm]
&=&
A_1+A_2+A_3+A_4
\end{eqnarray*}}
where $A_1 \sim A_4$ are as follows:
{\allowdisplaybreaks
\begin{eqnarray*}
& & \hspace{-8mm}
A_1
:= \sum_{j=1}^{\infty} \sum_{r=0}^{j-1} 
\sum_{\substack{k \in \frac12 \zzz_{\rm odd} \\[1mm] 0< k \leq m}}
(-1)^{j} e^{\pi ik} \, 
q^{(m+\frac12)j^2+(2p+1)mj} \, 
q^{-\frac{1}{4m}(2mr+k)^2-\frac{2p+1}{2}(2mr+k)} 
\\[-8mm]
& & \hspace{50mm}
\times \,\ 
\big[\theta_{k,m}^{(-)}-\theta_{-k,m}^{(-)}\big](\tau, z)
\nonumber
\\[2mm]
& & \hspace{-8mm}
A_2 :=
- \sum_{j=1}^{\infty} \sum_{r=0}^{j-1} 
\sum_{\substack{k \in \frac12 \zzz_{\rm odd} \\[1mm] 0< k \leq m}}
(-1)^{j} e^{\pi ik} \, 
q^{(m+\frac12)j^2-(2p+1)mj} \, 
q^{-\frac{1}{4m}(2mr+k)^2+\frac{2p+1}{2}(2mr+k)}
\\[-8mm]
& & \hspace{50mm}
\times \,\ 
\big[\theta_{k,m}^{(-)}-\theta_{-k,m}^{(-)}\big](\tau, z)
\nonumber
\\[2mm]
& & \hspace{-8mm}
A_3 :=
\sum_{j=1}^{\infty} \sum_{r=1}^{j} 
\sum_{\substack{k \in \frac12 \zzz_{\rm odd} \\[1mm] 0<k<m}} 
(-1)^{j} e^{\pi ik} \, 
q^{(m+\frac12)j^2+(2p+1)mj} \, 
q^{-\frac{1}{4m}(2mr-k)^2-\frac{2p+1}{2}(2mr-k)}
\\[-8mm]
& & \hspace{50mm}
\times \,\ 
\big[\theta_{k,m}^{(-)}-\theta_{-k,m}^{(-)}\big](\tau, z)
\\[2mm]
& & \hspace{-2mm}
= \sum_{j=0}^{\infty} \sum_{r=0}^{j} 
\sum_{\substack{k \in \frac12 \zzz_{\rm odd} \\[1mm] 0<k<m}} 
(-1)^{j} e^{\pi ik} \, 
q^{(m+\frac12)j^2+(2p+1)mj} \, 
q^{-\frac{1}{4m}(2mr-k)^2-\frac{2p+1}{2}(2mr-k)}
\\[-8mm]
& & \hspace{50mm}
\times \,\ 
\big[\theta_{k,m}^{(-)}-\theta_{-k,m}^{(-)}\big](\tau, z)
\\[2mm]
& & 
-\sum_{j=0}^{\infty} 
\sum_{\substack{k \in \frac12 \zzz_{\rm odd} \\[1mm] 0<k<m}} 
\hspace{-3mm}
(-1)^{j} e^{\pi ik} 
q^{(m+\frac12)j^2+(2p+1)mj}
q^{-\frac{1}{4m}k^2+\frac{2p+1}{2}k}
\big[\theta_{k,m}^{(-)}-\theta_{-k,m}^{(-)}\big](\tau, z)
\nonumber
\\[1mm]
& & \hspace{-8mm}
A_4 :=
- \sum_{j=1}^{\infty} \sum_{r=1}^{j} 
\sum_{\substack{k \in \frac12 \zzz_{\rm odd} \\[1mm] 0< k <m}} 
(-1)^{j} e^{\pi ik} \, 
q^{(m+\frac12)j^2-(2p+1)mj} \, 
q^{-\frac{1}{4m}(2mr-k)^2+\frac{2p+1}{2}(2mr-k)}
\\[-8mm]
& & \hspace{50mm}
\times \,\ 
\big[\theta_{k,m}^{(-)}-\theta_{-k,m}^{(-)}\big](\tau, z)
\end{eqnarray*}}
We decompose $A_2$ and $A_4$ as  
$A_2=A_2'+A_2''$ and $A_4=A_4'+A_4''$, where
{\allowdisplaybreaks
\begin{eqnarray*}
& & \hspace{-8mm}
A_2' := 
- \sum_{j=1}^{\infty} \sum_{r=0}^{j-1} 
\sum_{\substack{k \in \frac12 \zzz_{\rm odd} \\[1mm] 0< k < m}}
(-1)^{j} e^{\pi ik} \, 
q^{(m+\frac12)j^2-(2p+1)mj} \, 
q^{-\frac{1}{4m}(2mr+k)^2+\frac{2p+1}{2}(2mr+k)}
\\[-8mm]
& & \hspace{50mm}
\times \,\ 
\big[\theta_{k,m}^{(-)}-\theta_{-k,m}^{(-)}\big](\tau, z)
\\[2mm]
& & \hspace{-5mm}
\underset{\substack{\\[0.5mm] \uparrow \\[1mm] 
j \rightarrow -j \\[1mm]
r \rightarrow -r
}}{=} 
- \sum_{j<0} 
\sum_{\substack{r \in \zzz \\[1mm] j < r \leq 0}} 
\sum_{\substack{k \in \frac12 \zzz_{\rm odd} \\[1mm] 0< k < m}}
(-1)^{j} e^{\pi ik} \, 
q^{(m+\frac12)j^2+(2p+1)mj} 
q^{-\frac{1}{4m}(2mr-k)^2-\frac{2p+1}{2}(2mr-k)}
\\[-8mm]
& & \hspace{50mm}
\times \, 
\big[\theta_{k,m}^{(-)}-\theta_{-k,m}^{(-)}\big](\tau, z)
\\[2mm]
&=&
- \sum_{j<0} 
\sum_{\substack{r \in \zzz \\[1mm] j< r < 0}} 
\sum_{\substack{k \in \frac12 \zzz_{\rm odd} \\[1mm] 0< k < m}}
(-1)^{j} e^{\pi ik} 
q^{(m+\frac12)j^2+(2p+1)mj} 
q^{-\frac{1}{4m}(2mr-k)^2-\frac{2p+1}{2}(2mr-k)}
\\[-8mm]
& & \hspace{50mm}
\times \, 
\big[\theta_{k,m}^{(-)}-\theta_{-k,m}^{(-)}\big](\tau, z)
\\[2mm]
& &
- \sum_{j<0} 
\sum_{\substack{k \in \frac12 \zzz_{\rm odd} \\[1mm] 0< k < m}} 
\hspace{-2mm}
(-1)^{j} e^{\pi ik} 
q^{(m+\frac12)j^2+(2p+1)mj} 
q^{-\frac{1}{4m}k^2+(p+\frac12)k}
\big[\theta_{k,m}^{(-)}-\theta_{-k,m}^{(-)}\big](\tau, z)
\\[2mm]
& & \hspace{-8mm}
A_2'' := 
- \sum_{j=1}^{\infty} \sum_{r=0}^{j-1} 
(-1)^{j} e^{\pi im} \, 
q^{(m+\frac12)j^2-(2p+1)mj} \, 
q^{-\frac{1}{4m}(2mr+m)^2+\frac{2p+1}{2}(2mr+m)}
\\[-2mm]
& & \hspace{40mm}
\times \,\ 
\big[\theta_{m,m}^{(-)}-\theta_{-m,m}^{(-)}\big](\tau, z)
\\[2mm]
& & \hspace{-2mm}
= - 2\sum_{j=1}^{\infty} \sum_{r=0}^{j-1} 
(-1)^{j} e^{\pi im} \, 
q^{(m+\frac12)j^2-(2p+1)mj} \, 
q^{-\frac{m}{4}(2r+1)(2r-4p-1)}
\theta_{m,m}^{(-)}(\tau, z)
\\[2mm]
& & \hspace{-8mm}
A_4' := 
- \sum_{j=1}^{\infty} \sum_{r=1}^{j} 
\sum_{\substack{k \in \frac12 \zzz_{\rm odd} \\[1mm] 0< k \leq m}} 
\hspace{-2mm}
(-1)^{j} e^{\pi ik} 
q^{(m+\frac12)j^2-(2p+1)mj}  
q^{-\frac{1}{4m}(2mr-k)^2+\frac{2p+1}{2}(2mr-k)}
\\[-8mm]
& & \hspace{50mm}
\times \,\ 
\big[\theta_{k,m}^{(-)}-\theta_{-k,m}^{(-)}\big](\tau, z)
\\[2mm]
& & \hspace{-5mm}
\underset{\substack{\\[0.5mm] \uparrow \\[1mm] 
j \rightarrow -j \\[1mm]
r \rightarrow -r
}}{=} 
- \sum_{j<0} 
\sum_{\substack{r \in \zzz \\[1mm] j \leq r <0}}
\sum_{\substack{k \in \frac12 \zzz_{\rm odd} \\[1mm] 0< k \leq m}} 
\hspace{-2mm}
(-1)^{j} e^{\pi ik} 
q^{(m+\frac12)j^2+(2p+1)mj} 
q^{-\frac{1}{4m}(2mr+k)^2-\frac{2p+1}{2}(2mr+k)}
\\[-8mm]
& & \hspace{50mm}
\times \,  
\big[\theta_{k,m}^{(-)}-\theta_{-k,m}^{(-)}\big](\tau, z)
\\[2mm]
& & \hspace{-8mm}
A_4'' := 
\sum_{j=1}^{\infty} \sum_{r=1}^{j} 
(-1)^{j} e^{\pi im} \, 
q^{(m+\frac12)j^2-(2p+1)mj} \, 
q^{-\frac{1}{4m}(2mr-m)^2+\frac{2p+1}{2}(2mr-m)}
\\[-3mm]
& & \hspace{50mm}
\times \,\ 
\big[\theta_{m,m}^{(-)}-\theta_{-m,m}^{(-)}\big](\tau, z)
\\[2mm]
& & \hspace{-5mm}
\underset{\substack{\\[0.5mm] \uparrow \\[1mm] r=r'+1
}}{=} \hspace{-2mm}
2\sum_{j=1}^{\infty} \sum_{r'=0}^{j-1} 
(-1)^{j} e^{\pi im} \, 
q^{(m+\frac12)j^2-(2p+1)mj} \, 
q^{-\frac{m}{4}(2r'+1)(2r'-4p-1)}
\theta_{m,m}^{(-)}(\tau, z)
\end{eqnarray*}}
Then, noticing that 
{\allowdisplaybreaks
\begin{eqnarray*}
& & \hspace{-5mm}
A_2''+A_4'' \,\ = \,\ 0 
\\[2mm]
& & \hspace{-5mm}
A_1+A_4' \, = \, 
\bigg[
\sum_{\substack{j, \, r \in \zzz \\[1mm] 0 \leq r <j}}
- 
\sum_{\substack{j, \, r \in \zzz \\[1mm] j \leq r <0}}\bigg]
\sum_{\substack{k \in \frac12 \zzz_{\rm odd} \\[1mm] 0< k \leq m}} 
\hspace{-2mm}
(-1)^{j} e^{\pi ik} q^{(m+\frac12)j^2+(2p+1)mj} 
q^{-\frac{1}{4m}(2mr+k)^2-\frac{2p+1}{2}(2mr+k)}
\\[-8mm]
& & \hspace{80mm}
\times \,  
\big[\theta_{k,m}^{(-)}-\theta_{-k,m}^{(-)}\big](\tau, z)
\\[2mm]
& & \hspace{-5mm}
A_2'+A_3 \, = \, 
\bigg[
\sum_{\substack{j, \, r \in \zzz \\[1mm] 0 \leq r \leq j}}
- 
\sum_{\substack{j, \, r \in \zzz \\[1mm] j< r < 0}} \bigg]
\sum_{\substack{k \in \frac12 \zzz_{\rm odd} \\[1mm] 0< k < m}}
\hspace{-2mm}
(-1)^{j} e^{\pi ik} 
q^{(m+\frac12)j^2+(2p+1)mj} 
q^{-\frac{1}{4m}(2mr-k)^2-\frac{2p+1}{2}(2mr-k)}
\\[-8mm]
& & \hspace{80mm}
\times \, 
\big[\theta_{k,m}^{(-)}-\theta_{-k,m}^{(-)}\big](\tau, z)
\\[3mm]
& &
- \sum_{j\in \zzz} 
\sum_{\substack{k \in \frac12 \zzz_{\rm odd} \\[1mm] 0< k < m}} 
\hspace{-2mm}
(-1)^{j} e^{\pi ik} 
q^{(m+\frac12)j^2+(2p+1)mj} 
q^{-\frac{1}{4m}k^2+(p+\frac12)k}
\big[\theta_{k,m}^{(-)}-\theta_{-k,m}^{(-)}\big](\tau, z),
\end{eqnarray*}}
we have 
{\allowdisplaybreaks
\begin{eqnarray}
& & \hspace{-5mm}
{\rm (II)} = \, 
\bigg[
\sum_{\substack{j, \, r \in \zzz \\[1mm] 0 \leq r <j}}
- 
\sum_{\substack{j, \, r \in \zzz \\[1mm] j \leq r <0}}\bigg]
\sum_{\substack{k \in \frac12 \zzz_{\rm odd} \\[1mm] 0< k \leq m}} 
\hspace{-2mm}
(-1)^{j} e^{\pi ik} q^{(m+\frac12)j^2+(2p+1)mj} 
q^{-\frac{1}{4m}(2mr+k)^2-\frac{2p+1}{2}(2mr+k)}
\nonumber
\\[-8mm]
& & \hspace{80mm}
\times \, 
\big[\theta_{k,m}^{(-)}-\theta_{-k,m}^{(-)}\big](\tau, z)
\nonumber
\\[2mm]
& &
+ \bigg[
\sum_{\substack{j, \, r \in \zzz \\[1mm] 0 \leq r \leq j}}
- 
\sum_{\substack{j, \, r \in \zzz \\[1mm] j< r < 0}} \bigg]
\sum_{\substack{k \in \frac12 \zzz_{\rm odd} \\[1mm] 0< k < m}}
\hspace{-2mm}
(-1)^{j} e^{\pi ik} 
q^{(m+\frac12)j^2+(2p+1)mj} 
q^{-\frac{1}{4m}(2mr-k)^2-\frac{2p+1}{2}(2mr-k)}
\nonumber
\\[-8mm]
& & \hspace{80mm}
\times \,
\big[\theta_{k,m}^{(-)}-\theta_{-k,m}^{(-)}\big](\tau, z)
\nonumber
\\[3mm]
& &
- \sum_{j\in \zzz} 
\sum_{\substack{k \in \frac12 \zzz_{\rm odd} \\[1mm] 0< k < m}} 
\hspace{-2mm}
(-1)^{j} e^{\pi ik} 
q^{(m+\frac12)j^2+(2p+1)mj} 
q^{-\frac{1}{4m}k^2+(p+\frac12)k}
\big[\theta_{k,m}^{(-)}-\theta_{-k,m}^{(-)}\big](\tau, z)
\label{m1:eqn:2022-526d}
\end{eqnarray}}
Now the formulas \eqref{m1:eqn:2022-526a},
\eqref{m1:eqn:2022-526b}, \eqref{m1:eqn:2022-526c} and 
\eqref{m1:eqn:2022-526d} complete the proof of 1).

\medskip

\noindent
2) \,\ Letting $(z_1,z_2)= 
(\frac{z}{2}+\frac{\tau}{2}+p\tau, 
\frac{z}{2}-\frac{\tau}{2}-p\tau)$ in \eqref{m1:eqn:2022-524a}, 
we have 
\begin{equation}
\theta_{0, m+\frac12}^{(-)}
\Big(\tau, \dfrac{-1+2m(2p+1)\tau}{2m+1} \Big) 
\Phi^{(-)[m, \frac12]}\Big(\tau, \, 
\frac{z}{2}+\frac{\tau}{2}+p\tau, \,
\frac{z}{2}-\frac{\tau}{2}-p\tau, \, 0\Big)
\, = \, 
{\rm (III)} + {\rm (IV)}
\label{m1:eqn:2022-527a}
\end{equation}
where 
{\allowdisplaybreaks
\begin{eqnarray*}
& & \hspace{-10mm}
{\rm (III)} :=  \, - \, i \, \eta(\tau)^3 \, \Bigg\{
\frac{\displaystyle 
\theta_{0, m+\frac12}^{(-)}
\Big(\tau, \, \frac{1+(2p+1)\tau}{2m+1}+ z\Big)
}{\vartheta_{11}(\tau, \frac{z}{2}+\frac{\tau}{2}+p\tau)} 
+ 
\frac{\displaystyle 
\theta_{0, m+\frac12}^{(-)}
\Big(\tau, \, \frac{1+(2p+1)\tau}{2m+1}- z\Big)
}{\vartheta_{11}(\tau, \frac{z}{2}-\frac{\tau}{2}-p\tau)} 
\Bigg\}
\\[1mm]
& & \hspace{-10mm}
{\rm (IV)} := \, 
\sum_{j=1}^{\infty} \sum_{r=1}^j 
\sum_{\substack{k \in \zzz_{\rm odd} \\[1mm] 1 \leq k \leq 2m}}
(-1)^{j+r} q^{(m+\frac12)j^2-\frac{1}{16m}(4m(j-r)+k)^2}
\\[2mm]
& & \hspace{-5mm}
\times \, 
\big\{e^{\frac{\pi i}{2}(4mr-k)(2p+1)\tau}
+e^{-\frac{\pi i}{2}(4mr-k)(2p+1)\tau}\big\}
\big[\theta_{\frac{k}{2},m}^{(-)}
-\theta_{-\frac{k}{2},m}^{(-)}](\tau, z)
\\[2mm]
& & \hspace{-7mm}
- \, \sum_{j=1}^{\infty} \sum_{r=0}^{j-1} 
\sum_{\substack{k \in \zzz_{\rm odd} \\[1mm] 1 \leq k < 2m}}
(-1)^{j+r} q^{(m+\frac12)j^2-\frac{1}{16m}(4m(j-r)-k)^2}
\\[2mm]
& & \hspace{-5mm}
\times \, 
\big\{e^{\frac{\pi i}{2}(4mr+k)(2p+1)\tau}
+e^{-\frac{\pi i}{2}(4mr+k)(2p+1)\tau}\big\}
\big[\theta_{\frac{k}{2},m}^{(-)}
-\theta_{-\frac{k}{2},m}^{(-)}](\tau, z)
\end{eqnarray*}}
Using Lemma \ref{m1:lemma:2022-526a} and Note \ref{m1:note:2022-526a}, 
the LHS of \eqref{m1:eqn:2022-527a} and (III) become as follows:
\begin{subequations}
{\allowdisplaybreaks
\begin{eqnarray}
\text{LHS of \eqref{m1:eqn:2022-527a}} 
&=& 
q^{- \frac{m^2(2p+1)^2}{2(2m+1)}} 
\theta_{(2p+1)m, m+\frac12}(\tau,0)
\label{m1:eqn:2022-527b} 
\\[1mm]
{\rm (III)} \hspace{10mm}
&=&
(-1)^p \, q^{\frac{m(2p+1)^2}{4(2m+1)}}
\eta(\tau)^3
\frac{\big[\theta_{p+\frac12, m+\frac12}
-
\theta_{-(p+\frac12), m+\frac12}\big](\tau, z)
}{\theta_{0, \frac12}^{(-)}(\tau,z)}
\label{m1:eqn:2022-527c} 
\end{eqnarray}}
\end{subequations}
We compute (IV) by ptting $k'=\frac12 k$:
{\allowdisplaybreaks
\begin{eqnarray*}
& & \hspace{-10mm}
{\rm (IV)} \, = \,  \sum_{j=1}^{\infty} \sum_{r=1}^j 
\sum_{\substack{k' \in \frac12 \zzz_{\rm odd} \\[1mm] 0 < k' \leq m}}
(-1)^{j+r} q^{(m+\frac12)j^2-\frac{1}{4m}(2m(j-r)+k')^2}
\\[2mm]
& & \hspace{3mm}
\times \, \big\{
q^{\frac12(2mr-k')(2p+1)}
+
q^{-\frac12(2mr-k')(2p+1)}
\big\}
\big[\theta_{k',m}^{(-)}-\theta_{-k',m}^{(-)}](\tau, z)
\nonumber
\\[2mm]
& & 
- \,\ \sum_{j=1}^{\infty} \sum_{r=0}^{j-1} 
\sum_{\substack{k' \in \frac12 \zzz_{\rm odd} \\[1mm] 0 < k' < m}}
(-1)^{j+r} q^{(m+\frac12)j^2-\frac{1}{4m}(2m(j-r)-k')^2}
\\[2mm]
& & \hspace{3mm}
\times \, \big\{
q^{\frac12(2mr+k')(2p+1)}
+
q^{-\frac12(2mr+k')(2p+1)}
\big\}
\big[\theta_{k',m}^{(-)}-\theta_{-k',m}^{(-)}](\tau, z)
\end{eqnarray*}}
Putting $r=j-r'$, this equation is rewritten as follows:
{\allowdisplaybreaks
\begin{eqnarray*}
& & \hspace{-8mm}
{\rm (IV)} \, = \,
\sum_{j=1}^{\infty} \sum_{r'=0}^{j-1} 
\sum_{\substack{k \, \in \, \frac12 \zzz_{\rm odd} \\[1mm] 0< k \leq m}}
(-1)^{r'} \, 
q^{(m+\frac12)j^2} \, q^{-\frac{1}{4m}(2mr'+k)^2}
\nonumber
\\[2mm]
& & \hspace{-5mm}
\times \big\{
q^{(2p+1)mj} q^{-\frac{2p+1}{2}(2mr'+k)} + 
q^{-(2p+1)mj} q^{\frac{2p+1}{2}(2mr'+k)}
\big\}
\big[\theta_{k,m}^{(-)}-\theta_{-k,m}^{(-)}\big](\tau, z)
\nonumber
\\[2mm]
& & \hspace{-7mm}
- \,\ \sum_{j=1}^{\infty} \sum_{r'=1}^{j} 
\sum_{\substack{k \, \in \, \frac12 \zzz_{\rm odd} \\[1mm] 0< k < m}}
(-1)^{r'} \,
q^{(m+\frac12)j^2} \, q^{-\frac{1}{4m}(2mr'-k)^2}
\nonumber
\\[2mm]
& & \hspace{-5mm}
\times \big\{
q^{(2p+1)mj} q^{-\frac{2p+1}{2}(2mr'-k)} +
q^{-(2p+1)mj} q^{\frac{2p+1}{2}(2mr'-k)}
\big\}
\big[\theta_{k,m}^{(-)}-\theta_{-k,m}^{(-)}\big](\tau, z)
\\[2mm]
&=&
B_1+B_2+B_3+B_4
\end{eqnarray*}}
where $B_1 \sim B_4$ are as follows:
{\allowdisplaybreaks
\begin{eqnarray*}
& & \hspace{-8mm}
B_1
:= \sum_{j=1}^{\infty} \sum_{r=0}^{j-1} 
\sum_{\substack{k \in \frac12 \zzz_{\rm odd} \\[1mm] 0< k \leq m}}
(-1)^{r} 
q^{(m+\frac12)j^2+(2p+1)mj} 
q^{-\frac{1}{4m}(2mr+k)^2-\frac{2p+1}{2}(2mr+k)} 
\big[\theta_{k,m}^{(-)}-\theta_{-k,m}^{(-)}\big](\tau, z)
\nonumber
\\[1mm]
& & \hspace{-8mm}
B_2 :=
\sum_{j=1}^{\infty} \sum_{r=0}^{j-1} 
\sum_{\substack{k \in \frac12 \zzz_{\rm odd} \\[1mm] 0< k \leq m}}
(-1)^{r} \, 
q^{(m+\frac12)j^2-(2p+1)mj} 
q^{-\frac{1}{4m}(2mr+k)^2+\frac{2p+1}{2}(2mr+k)}
\big[\theta_{k,m}^{(-)}-\theta_{-k,m}^{(-)}\big](\tau, z)
\nonumber
\\[1mm]
& & \hspace{-8mm}
B_3 :=-
\sum_{j=1}^{\infty} \sum_{r=1}^{j} 
\sum_{\substack{k \in \frac12 \zzz_{\rm odd} \\[1mm] 0<k<m}} 
\hspace{-2mm}
(-1)^{r} 
q^{(m+\frac12)j^2+(2p+1)mj} 
q^{-\frac{1}{4m}(2mr-k)^2-\frac{2p+1}{2}(2mr-k)}
\big[\theta_{k,m}^{(-)}-\theta_{-k,m}^{(-)}\big](\tau, z)
\\[1mm]
& & \hspace{-2mm}
= -\sum_{j=0}^{\infty} \sum_{r=0}^{j} 
\sum_{\substack{k \in \frac12 \zzz_{\rm odd} \\[1mm] 0<k<m}} 
\hspace{-2mm}
(-1)^{r} 
q^{(m+\frac12)j^2+(2p+1)mj} \, 
q^{-\frac{1}{4m}(2mr-k)^2-\frac{2p+1}{2}(2mr-k)}
\big[\theta_{k,m}^{(-)}-\theta_{-k,m}^{(-)}\big](\tau, z)
\\[1mm]
& & 
+\sum_{j=0}^{\infty} 
\sum_{\substack{k \in \frac12 \zzz_{\rm odd} \\[1mm] 0<k<m}} 
q^{(m+\frac12)j^2+(2p+1)mj}
q^{-\frac{1}{4m}k^2+\frac{2p+1}{2}k}
\big[\theta_{k,m}^{(-)}-\theta_{-k,m}^{(-)}\big](\tau, z)
\nonumber
\\[1mm]
& & \hspace{-8mm}
B_4 :=
- \sum_{j=1}^{\infty} \sum_{r=1}^{j} 
\sum_{\substack{k \in \frac12 \zzz_{\rm odd} \\[1mm] 0< k <m}} 
\hspace{-2mm}
(-1)^{r} 
q^{(m+\frac12)j^2-(2p+1)mj} 
q^{-\frac{1}{4m}(2mr-k)^2+\frac{2p+1}{2}(2mr-k)}
\big[\theta_{k,m}^{(-)}-\theta_{-k,m}^{(-)}\big](\tau, z)
\end{eqnarray*}}
We decompose $B_2$ and $B_4$ as $B_2=B_2'+B_2''$ and $B_4=B_4'+B_4''$, where
{\allowdisplaybreaks
\begin{eqnarray*}
& & \hspace{-8mm}
B_2' := 
\sum_{j=1}^{\infty} \sum_{r=0}^{j-1} 
\sum_{\substack{k \in \frac12 \zzz_{\rm odd} \\[1mm] 0< k < m}}
(-1)^{r} 
q^{(m+\frac12)j^2-(2p+1)mj} 
q^{-\frac{1}{4m}(2mr+k)^2+\frac{2p+1}{2}(2mr+k)}
\big[\theta_{k,m}^{(-)}-\theta_{-k,m}^{(-)}\big](\tau, z)
\\[2mm]
& & \hspace{-5mm}
\underset{\substack{\\[0.5mm] \uparrow \\[1mm] 
j \rightarrow -j \\[1mm]
r \rightarrow -r
}}{=} 
\sum_{j<0} 
\sum_{\substack{r \in \zzz \\[1mm] j < r \leq 0}} 
\sum_{\substack{k \in \frac12 \zzz_{\rm odd} \\[1mm] 0< k < m}}
\hspace{-2mm}
(-1)^{r} 
q^{(m+\frac12)j^2+(2p+1)mj} 
q^{-\frac{1}{4m}(2mr-k)^2-\frac{2p+1}{2}(2mr-k)}
\big[\theta_{k,m}^{(-)}-\theta_{-k,m}^{(-)}\big](\tau, z)
\\[2mm]
& & \hspace{-2mm}
= \sum_{j<0} 
\sum_{\substack{r \in \zzz \\[1mm] j< r < 0}} 
\sum_{\substack{k \in \frac12 \zzz_{\rm odd} \\[1mm] 0< k < m}}
\hspace{-2mm}
(-1)^{r}  
q^{(m+\frac12)j^2+(2p+1)mj} 
q^{-\frac{1}{4m}(2mr-k)^2-\frac{2p+1}{2}(2mr-k)}
\big[\theta_{k,m}^{(-)}-\theta_{-k,m}^{(-)}\big](\tau, z)
\\[2mm]
& &
+ \sum_{j<0} 
\sum_{\substack{k \in \frac12 \zzz_{\rm odd} \\[1mm] 0< k < m}} 
\hspace{-2mm}
q^{(m+\frac12)j^2+(2p+1)mj} 
q^{-\frac{1}{4m}k^2+(p+\frac12)k}
\big[\theta_{k,m}^{(-)}-\theta_{-k,m}^{(-)}\big](\tau, z)
\\[2mm]
& & \hspace{-8mm}
B_2'' := 
\sum_{j=1}^{\infty} \sum_{r=0}^{j-1} 
(-1)^{r}  \, 
q^{(m+\frac12)j^2-(2p+1)mj} \, 
q^{-\frac{1}{4m}(2mr+m)^2+\frac{2p+1}{2}(2mr+m)}
\big[\theta_{m,m}^{(-)}-\theta_{-m,m}^{(-)}\big](\tau, z)
\\[2mm]
& & \hspace{-2mm}
= 2\sum_{j=1}^{\infty} \sum_{r=0}^{j-1} 
(-1)^{r}  \, 
q^{(m+\frac12)j^2-(2p+1)mj} \, 
q^{-\frac{m}{4}(2r+1)(2r-4p-1)}
\theta_{m,m}^{(-)}(\tau, z)
\\[2mm]
& & \hspace{-8mm}
B_4' := 
- \sum_{j=1}^{\infty} \sum_{r=1}^{j} 
\sum_{\substack{k \in \frac12 \zzz_{\rm odd} \\[1mm] 0< k \leq m}} 
\hspace{-2mm}
(-1)^{r} 
q^{(m+\frac12)j^2-(2p+1)mj}  
q^{-\frac{1}{4m}(2mr-k)^2+\frac{2p+1}{2}(2mr-k)}
\big[\theta_{k,m}^{(-)}-\theta_{-k,m}^{(-)}\big](\tau, z)
\\[2mm]
& & \hspace{-5mm}
\underset{\substack{\\[0.5mm] \uparrow \\[1mm] 
j \rightarrow -j \\[1mm]
r \rightarrow -r
}}{=} \hspace{-3mm}
- \sum_{j<0} 
\sum_{\substack{r \in \zzz \\[1mm] j \leq r <0}}
\sum_{\substack{k \in \frac12 \zzz_{\rm odd} \\[1mm] 0< k \leq m}} 
\hspace{-2mm}
(-1)^{r} 
q^{(m+\frac12)j^2+(2p+1)mj} 
q^{-\frac{1}{4m}(2mr+k)^2-\frac{2p+1}{2}(2mr+k)}
\big[\theta_{k,m}^{(-)}-\theta_{-k,m}^{(-)}\big](\tau, z)
\\[2mm]
& & \hspace{-8mm}
B_4'' := 
\sum_{j=1}^{\infty} \sum_{r=1}^{j} 
(-1)^{r}  \, 
q^{(m+\frac12)j^2-(2p+1)mj} \, 
q^{-\frac{1}{4m}(2mr-m)^2+\frac{2p+1}{2}(2mr-m)}
\big[\theta_{m,m}^{(-)}-\theta_{-m,m}^{(-)}\big](\tau, z)
\\[2mm]
& & \hspace{-5mm}
\underset{\substack{\\[0.5mm] \uparrow \\[1mm] r=r'+1
}}{=} \hspace{-2mm}
-2\sum_{j=1}^{\infty} \sum_{r'=0}^{j-1} 
(-1)^{r'}  \, 
q^{(m+\frac12)j^2-(2p+1)mj} \, 
q^{-\frac{m}{4}(2r'+1)(2r'-4p-1)}
\theta_{m,m}^{(-)}(\tau, z)
\end{eqnarray*}}
Then, noticing that 
{\allowdisplaybreaks
\begin{eqnarray*}
& & \hspace{-5mm}
B_2''+B_4'' \,\ = \,\ 0 
\\[2mm]
& & \hspace{-5mm}
B_1+B_4' \, = \, 
\bigg[
\sum_{\substack{j, \, r \in \zzz \\[1mm] 0 \leq r <j}}
- 
\sum_{\substack{j, \, r \in \zzz \\[1mm] j \leq r <0}}\bigg]
\sum_{\substack{k \in \frac12 \zzz_{\rm odd} \\[1mm] 0< k \leq m}} 
\hspace{-2mm}
(-1)^{r}  q^{(m+\frac12)j^2+(2p+1)mj} 
q^{-\frac{1}{4m}(2mr+k)^2-\frac{2p+1}{2}(2mr+k)}
\\[-8mm]
& & \hspace{80mm}
\times \, 
\big[\theta_{k,m}^{(-)}-\theta_{-k,m}^{(-)}\big](\tau, z)
\\[2mm]
& & \hspace{-5mm}
B_2'+B_3 \, = \, - \, 
\bigg[
\sum_{\substack{j, \, r \in \zzz \\[1mm] 0 \leq r \leq j}}
- 
\sum_{\substack{j, \, r \in \zzz \\[1mm] j< r < 0}} \bigg]
\sum_{\substack{k \in \frac12 \zzz_{\rm odd} \\[1mm] 0< k < m}}
\hspace{-2mm}
(-1)^{r} 
q^{(m+\frac12)j^2+(2p+1)mj} 
q^{-\frac{1}{4m}(2mr-k)^2-\frac{2p+1}{2}(2mr-k)}
\\[-8mm]
& & \hspace{80mm}
\times \, 
\big[\theta_{k,m}^{(-)}-\theta_{-k,m}^{(-)}\big](\tau, z)
\\[2mm]
& &
+ \sum_{j\in \zzz} 
\sum_{\substack{k \in \frac12 \zzz_{\rm odd} \\[1mm] 0< k < m}} 
\hspace{-2mm}
q^{(m+\frac12)j^2+(2p+1)mj} 
q^{-\frac{1}{4m}k^2+(p+\frac12)k}
\big[\theta_{k,m}^{(-)}-\theta_{-k,m}^{(-)}\big](\tau, z),
\end{eqnarray*}}
we have 
{\allowdisplaybreaks
\begin{eqnarray}
& & \hspace{-10mm}
{\rm (IV)} = \, 
\bigg[
\sum_{\substack{j, \, r \in \zzz \\[1mm] 0 \leq r <j}}
- 
\sum_{\substack{j, \, r \in \zzz \\[1mm] j \leq r <0}}\bigg]
\sum_{\substack{k \in \frac12 \zzz_{\rm odd} \\[1mm] 0< k \leq m}} 
\hspace{-2mm}
(-1)^{r} q^{(m+\frac12)j^2+(2p+1)mj} 
q^{-\frac{1}{4m}(2mr+k)^2-\frac{2p+1}{2}(2mr+k)}
\nonumber
\\[-7mm]
& & \hspace{80mm}
\times \, 
\big[\theta_{k,m}^{(-)}-\theta_{-k,m}^{(-)}\big](\tau, z)
\nonumber
\\[2mm]
& &
- \bigg[
\sum_{\substack{j, \, r \in \zzz \\[1mm] 0 \leq r \leq j}}
- 
\sum_{\substack{j, \, r \in \zzz \\[1mm] j< r < 0}} \bigg]
\sum_{\substack{k \in \frac12 \zzz_{\rm odd} \\[1mm] 0< k < m}}
(-1)^{r} 
q^{(m+\frac12)j^2+(2p+1)mj} 
q^{-\frac{1}{4m}(2mr-k)^2-\frac{2p+1}{2}(2mr-k)}
\nonumber
\\[-7mm]
& & \hspace{80mm}
\times \, 
\big[\theta_{k,m}^{(-)}-\theta_{-k,m}^{(-)}\big](\tau, z)
\nonumber
\\[2mm]
& &
+ \sum_{j\in \zzz} 
\sum_{\substack{k \in \frac12 \zzz_{\rm odd} \\[1mm] 0< k < m}} 
\hspace{-2mm}
q^{(m+\frac12)j^2+(2p+1)mj} 
q^{-\frac{1}{4m}k^2+(p+\frac12)k}
\big[\theta_{k,m}^{(-)}-\theta_{-k,m}^{(-)}\big](\tau, z)
\label{m1:eqn:2022-527d}
\end{eqnarray}}
Now the formulas \eqref{m1:eqn:2022-527a},
\eqref{m1:eqn:2022-527b}, \eqref{m1:eqn:2022-527c} and 
\eqref{m1:eqn:2022-527d} complete the proof of 2).

\medskip
\noindent
3) \,\ Letting $(z_1,z_2)= 
(\frac{z}{2}-\frac12+p\tau, 
\frac{z}{2}+\frac12-p\tau)$ in \eqref{m1:eqn:2022-524a}, 
we have 
\begin{equation}
\theta_{0, m+\frac12}^{(-)}
\Big(\tau, \dfrac{-1+2m(2p\tau-1)}{2m+1} \Big) 
\Phi^{(-)[m, \frac12]}\Big(\tau, \, 
\frac{z}{2}-\frac12+p\tau, \,
\frac{z}{2}+\frac12-p\tau, \, 0\Big)
\, = \, 
{\rm (V)} + {\rm (VI)}
\label{m1:eqn:2022-527e}
\end{equation}
where 
{\allowdisplaybreaks
\begin{eqnarray*}
& & \hspace{-10mm}
{\rm (V)} := \, 
-i \, \eta(\tau)^3 \, \Bigg\{ \, 
\frac{\displaystyle 
\theta_{0, m+\frac12}^{(-)}
\Big(\tau, \, \frac{2p\tau}{2m+1}+ z\Big)
}{\vartheta_{11}(\tau, \frac{z}{2}-\frac12+p\tau)} 
+ \, 
\frac{\displaystyle 
\theta_{0, m+\frac12}^{(-)}
\Big(\tau, \, \frac{2p\tau}{2m+1}- z\Big)
}{\vartheta_{11}(\tau, \frac{z}{2}+\frac12-p\tau)} 
\Bigg\}
\\[1mm]
& & \hspace{-10mm}
{\rm (VI)} := \, 
\sum_{j=1}^{\infty} \sum_{r=1}^j 
\sum_{\substack{k \in \zzz_{\rm odd} \\[1mm] 1 \leq k \leq 2m}}
(-1)^{j+r} q^{(m+\frac12)j^2-\frac{1}{16m}(4m(j-r)+k)^2}
\\[2mm]
& & \hspace{-5mm}
\times \, 
\big\{e^{\frac{\pi i}{2}(4mr-k)(2p\tau-1)}
+e^{-\frac{\pi i}{2}(4mr-k)(2p\tau-1)}\big\}
\big[\theta_{\frac{k}{2},m}^{(-)}
-\theta_{-\frac{k}{2},m}^{(-)}](\tau, z)
\\[2mm]
& & \hspace{-7mm}
- \, \sum_{j=1}^{\infty} \sum_{r=0}^{j-1} 
\sum_{\substack{k \in \zzz_{\rm odd} \\[1mm] 1 \leq k < 2m}}
(-1)^{j+r} q^{(m+\frac12)j^2-\frac{1}{16m}(4m(j-r)-k)^2}
\\[2mm]
& & \hspace{-5mm}
\times \, 
\big\{e^{\frac{\pi i}{2}(4mr+k)(2p\tau-1)}
+e^{-\frac{\pi i}{2}(4mr+k)(2p\tau-1)}\big\}
\big[\theta_{\frac{k}{2},m}^{(-)}
-\theta_{-\frac{k}{2},m}^{(-)}](\tau, z)
\end{eqnarray*}}
Using Lemma \ref{m1:lemma:2022-526a} and Note \ref{m1:note:2022-526a}, 
the LHS of \eqref{m1:eqn:2022-527e} and (V) become as follows:
\begin{subequations}
{\allowdisplaybreaks
\begin{eqnarray}
\text{LHS of \eqref{m1:eqn:2022-527e}} 
&=& 
q^{- \frac{m^2(2p)^2}{2(2m+1)}} 
\theta_{2pm, m+\frac12}^{(-)}(\tau,0)
\label{m1:eqn:2022-527f} 
\\[1mm]
{\rm (V)} \hspace{10mm}
&=&
-i \, q^{\frac{mp^2}{2m+1}}
\eta(\tau)^3
\frac{\big[\theta_{p, m+\frac12}^{(-)}
-
\theta_{-p, m+\frac12}^{(-)}\big](\tau, z)
}{\theta_{\frac12, \frac12}(\tau,z)}
\label{m1:eqn:2022-527g} 
\end{eqnarray}}
\end{subequations}
We compute (VI) by ptting $k'=\frac12 k$:
{\allowdisplaybreaks
\begin{eqnarray*}
& & \hspace{-10mm}
{\rm (VI)} \, = \, 
\sum_{j=1}^{\infty} \sum_{r=1}^j 
\sum_{\substack{k' \in \frac12 \zzz_{\rm odd} \\[1mm] 0 < k' \leq m}}
(-1)^{j+r} q^{(m+\frac12)j^2-\frac{1}{4m}(2m(j-r)+k')^2}
\\[2mm]
& & \hspace{3mm}
\times \, \big\{
q^{(2mr-k')p} \, e^{-\pi i(2mr-k')}
+
q^{-(2mr-k')p} \, e^{\pi i(2mr-k')}\big\}
\big[\theta_{k',m}^{(-)}-\theta_{-k',m}^{(-)}](\tau, z)
\nonumber
\\[2mm]
& & 
- \,\ \sum_{j=1}^{\infty} \sum_{r=0}^{j-1} 
\sum_{\substack{k' \in \frac12 \zzz_{\rm odd} \\[1mm] 0 < k' < m}}
(-1)^{j+r} q^{(m+\frac12)j^2-\frac{1}{4m}(2m(j-r)-k')^2}
\\[2mm]
& & \hspace{3mm}
\times \, \big\{
q^{(2mr+k')p}e^{-\pi i(2mr+k')}
+
q^{-(2mr+k')2p)}e^{\pi i(2mr+k')}\big\}
\big[\theta_{k',m}^{(-)}-\theta_{-k',m}^{(-)}](\tau, z)
\\[2mm]
&=&
\sum_{j=1}^{\infty} \sum_{r=1}^j 
\sum_{\substack{k \, \in \, \frac12 \zzz_{\rm odd} \\[1mm] 0< k \leq m}}
(-1)^{j} \, e^{\pi ik} \, 
q^{(m+\frac12)j^2} \, q^{-\frac{1}{4m}(2m(j-r)+k)^2}
\big\{q^{p(2mr-k)} \, - \, q^{-p(2mr-k)}\big\}
\\[-8mm]
& & \hspace{50mm}
\times \,\ 
\big[\theta_{k,m}^{(-)}-\theta_{-k,m}^{(-)}\big](\tau, z)
\\[2mm]
& & \hspace{-7mm}
+ \,\ \sum_{j=1}^{\infty} \sum_{r=0}^{j-1} 
\sum_{\substack{k \, \in \, \frac12 \zzz_{\rm odd} \\[1mm] 0< k <m}}
(-1)^{j} \, e^{\pi ik} \, 
q^{(m+\frac12)j^2} \, q^{-\frac{1}{4m}(2m(j-r)-k)^2}
\big\{q^{p(2mr+k)} \, - \, q^{-p(2mr+k)}\big\}
\\[-8mm]
& & \hspace{50mm}
\times \,\ 
\big[\theta_{k,m}^{(-)}-\theta_{-k,m}^{(-)}\big](\tau, z)
\end{eqnarray*}}
Putting $r=j-r'$, this equation is rewritten as follows:
{\allowdisplaybreaks
\begin{eqnarray*}
& & \hspace{-13mm}
{\rm (VI)} \, = \, 
\sum_{j=1}^{\infty} \sum_{r'=0}^{j-1} 
\sum_{\substack{k \, \in \, \frac12 \zzz_{\rm odd} \\[1mm] 0< k \leq m}}
(-1)^{j} \, e^{\pi ik} \, 
q^{(m+\frac12)j^2} \, q^{-\frac{1}{4m}(2mr'+k)^2}
\nonumber
\\[1mm]
& & 
\times \big\{
q^{2pmj} q^{-p(2mr'+k)} - q^{-2pmj} q^{p(2mr'+k)}
\big\}
\big[\theta_{k,m}^{(-)}-\theta_{-k,m}^{(-)}\big](\tau, z)
\nonumber
\\[2mm]
& & \hspace{-5mm}
+ \,\ \sum_{j=1}^{\infty} \sum_{r'=1}^{j} 
\sum_{\substack{k \, \in \, \frac12 \zzz_{\rm odd} \\[1mm] 0< k < m}}
(-1)^{j} \, e^{\pi ik} \, 
q^{(m+\frac12)j^2} \, q^{-\frac{1}{4m}(2mr'-k)^2}
\nonumber
\\[1mm]
& & 
\times \big\{
q^{2pmj} q^{-p(2mr'-k)} -
q^{-2pmj} q^{p(2mr'-k)}
\big\}
\big[\theta_{k,m}^{(-)}-\theta_{-k,m}^{(-)}\big](\tau, z)
\\[2mm]
&=&
C_1+C_2+C_3+C_4
\end{eqnarray*}}
where $C_1 \sim C_4$ are as follows:
{\allowdisplaybreaks
\begin{eqnarray*}
& & \hspace{-8mm}
C_1
:= \sum_{j=1}^{\infty} \sum_{r=0}^{j-1} 
\sum_{\substack{k \in \frac12 \zzz_{\rm odd} \\[1mm] 0< k \leq m}}
(-1)^{j} e^{\pi ik} \, 
q^{(m+\frac12)j^2+2pmj} \, 
q^{-\frac{1}{4m}(2mr+k)^2-p(2mr+k)} 
\big[\theta_{k,m}^{(-)}-\theta_{-k,m}^{(-)}\big](\tau, z)
\nonumber
\\[1mm]
& & \hspace{-8mm}
C_2 :=
- \sum_{j=1}^{\infty} \sum_{r=0}^{j-1} 
\sum_{\substack{k \in \frac12 \zzz_{\rm odd} \\[1mm] 0< k \leq m}}
\hspace{-2mm}
(-1)^{j} e^{\pi ik} 
q^{(m+\frac12)j^2-2pmj} 
q^{-\frac{1}{4m}(2mr+k)^2+p(2mr+k)}
\big[\theta_{k,m}^{(-)}-\theta_{-k,m}^{(-)}\big](\tau, z)
\nonumber
\\[1mm]
& & \hspace{-8mm}
C_3 :=
\sum_{j=1}^{\infty} \sum_{r=1}^{j} 
\sum_{\substack{k \in \frac12 \zzz_{\rm odd} \\[1mm] 0<k<m}} 
(-1)^{j} e^{\pi ik} \, 
q^{(m+\frac12)j^2+2pmj} \, 
q^{-\frac{1}{4m}(2mr-k)^2-p(2mr-k)}
\big[\theta_{k,m}^{(-)}-\theta_{-k,m}^{(-)}\big](\tau, z)
\\[1mm]
& & \hspace{-2mm}
= \sum_{j=0}^{\infty} \sum_{r=0}^{j} 
\sum_{\substack{k \in \frac12 \zzz_{\rm odd} \\[1mm] 0<k<m}} 
(-1)^{j} e^{\pi ik} \, 
q^{(m+\frac12)j^2+2pmj} \, 
q^{-\frac{1}{4m}(2mr-k)^2-p(2mr-k)}
\big[\theta_{k,m}^{(-)}-\theta_{-k,m}^{(-)}\big](\tau, z)
\\[1mm]
& & 
-\sum_{j=0}^{\infty} 
\sum_{\substack{k \in \frac12 \zzz_{\rm odd} \\[1mm] 0<k<m}} 
\hspace{-3mm}
(-1)^{j} e^{\pi ik} 
q^{(m+\frac12)j^2+2pmj}
q^{-\frac{1}{4m}k^2+pk}
\big[\theta_{k,m}^{(-)}-\theta_{-k,m}^{(-)}\big](\tau, z)
\nonumber
\\[1mm]
& & \hspace{-8mm}
C_4 :=
- \sum_{j=1}^{\infty} \sum_{r=1}^{j} 
\sum_{\substack{k \in \frac12 \zzz_{\rm odd} \\[1mm] 0< k <m}} 
\hspace{-2mm}
(-1)^{j} e^{\pi ik} 
q^{(m+\frac12)j^2-2pmj} 
q^{-\frac{1}{4m}(2mr-k)^2+p(2mr-k)}
\big[\theta_{k,m}^{(-)}-\theta_{-k,m}^{(-)}\big](\tau, z)
\end{eqnarray*}}
We decompose $C_2$ and $C_4$ as $C_2=C_2'+C_2''$ and $C_4=C_4'+C_4''$, 
where
{\allowdisplaybreaks
\begin{eqnarray*}
& & \hspace{-8mm}
C_2' := 
- \sum_{j=1}^{\infty} \sum_{r=0}^{j-1} 
\sum_{\substack{k \in \frac12 \zzz_{\rm odd} \\[1mm] 0< k < m}}
\hspace{-2mm}
(-1)^{j} e^{\pi ik} 
q^{(m+\frac12)j^2-2pmj} 
q^{-\frac{1}{4m}(2mr+k)^2+p(2mr+k)}
\big[\theta_{k,m}^{(-)}-\theta_{-k,m}^{(-)}\big](\tau, z)
\\[2mm]
& & \hspace{-5mm}
\underset{\substack{\\[0.5mm] \uparrow \\[1mm] 
j \rightarrow -j \\[1mm]
r \rightarrow -r
}}{=} 
- \sum_{j<0} 
\sum_{\substack{r \in \zzz \\[1mm] j< r \leq 0}} 
\sum_{\substack{k \in \frac12 \zzz_{\rm odd} \\[1mm] 0< k < m}}
(-1)^{j} e^{\pi ik} \, 
q^{(m+\frac12)j^2+2pmj} 
q^{-\frac{1}{4m}(2mr-k)^2-p(2mr-k)}
\\[-9mm]
& & \hspace{50mm}
\times \,  
\big[\theta_{k,m}^{(-)}-\theta_{-k,m}^{(-)}\big](\tau, z)
\\[3mm]
& &\hspace{-2mm}
= \, - \sum_{j<0} 
\sum_{\substack{r \in \zzz \\[1mm] j< r < 0}} 
\sum_{\substack{k \in \frac12 \zzz_{\rm odd} \\[1mm] 0< k < m}}
(-1)^{j} e^{\pi ik} 
q^{(m+\frac12)j^2+2pmj} 
q^{-\frac{1}{4m}(2mr-k)^2-p(2mr-k)}
\\[-9mm]
& & \hspace{50mm}
\times \, 
\big[\theta_{k,m}^{(-)}-\theta_{-k,m}^{(-)}\big](\tau, z)
\\[3mm]
& & \hspace{2mm}
- \, \sum_{j<0} 
\sum_{\substack{k \in \frac12 \zzz_{\rm odd} \\[1mm] 0< k < m}} 
\hspace{-2mm}
(-1)^{j} e^{\pi ik} 
q^{(m+\frac12)j^2+2pmj} 
q^{-\frac{1}{4m}k^2+pk}
\big[\theta_{k,m}^{(-)}-\theta_{-k,m}^{(-)}\big](\tau, z)
\\[1mm]
& & \hspace{-8mm}
C_2'' := 
- \sum_{j=1}^{\infty} \sum_{r=0}^{j-1} 
(-1)^{j} e^{\pi im} 
q^{(m+\frac12)j^2-2pmj} 
q^{-\frac{1}{4m}(2mr+m)^2+p(2mr+m)}
\big[\theta_{m,m}^{(-)}-\theta_{-m,m}^{(-)}\big](\tau, z)
\\[2mm]
& & \hspace{-2mm}
= - 2\sum_{j=1}^{\infty} \sum_{r=0}^{j-1} 
(-1)^{j} e^{\pi im} \, 
q^{(m+\frac12)j^2-2pmj} \, 
q^{-\frac{m}{4}(2r+1)(2r-4p+1)}
\theta_{m,m}^{(-)}(\tau, z)
\\[2mm]
& & \hspace{-8mm}
C_4' := 
- \sum_{j=1}^{\infty} \sum_{r=1}^{j} 
\sum_{\substack{k \in \frac12 \zzz_{\rm odd} \\[1mm] 0< k \leq m}} 
\hspace{-2mm}
(-1)^{j} e^{\pi ik} 
q^{(m+\frac12)j^2-2pmj}  
q^{-\frac{1}{4m}(2mr-k)^2+p(2mr-k)}
\big[\theta_{k,m}^{(-)}-\theta_{-k,m}^{(-)}\big](\tau, z)
\\[2mm]
& & \hspace{-5mm}
\underset{\substack{\\[0.5mm] \uparrow \\[1mm] 
j \rightarrow -j \\[1mm]
r \rightarrow -r
}}{=} 
- \sum_{j<0} 
\sum_{\substack{r \in \zzz \\[1mm] j \leq r <0}}
\sum_{\substack{k \in \frac12 \zzz_{\rm odd} \\[1mm] 0< k \leq m}} 
\hspace{-2mm}
(-1)^{j} e^{\pi ik} 
q^{(m+\frac12)j^2+2pmj} 
q^{-\frac{1}{4m}(2mr+k)^2-p(2mr+k)}
\\[-9mm]
& & \hspace{50mm}
\times \, 
\big[\theta_{k,m}^{(-)}-\theta_{-k,m}^{(-)}\big](\tau, z)
\\[2mm]
& & \hspace{-8mm}
C_4'' := 
\sum_{j=1}^{\infty} \sum_{r=1}^{j} 
(-1)^{j} e^{\pi im} \, 
q^{(m+\frac12)j^2-2pmj} \, 
q^{-\frac{1}{4m}(2mr-m)^2+p(2mr-m)}
\big[\theta_{m,m}^{(-)}-\theta_{-m,m}^{(-)}\big](\tau, z)
\\[2mm]
& & \hspace{-5mm}
\underset{\substack{\\[0.5mm] \uparrow \\[1mm] r=r'+1
}}{=} \hspace{-2mm}
2\sum_{j=1}^{\infty} \sum_{r'=0}^{j-1} 
(-1)^{j} e^{\pi im} \, 
q^{(m+\frac12)j^2-2pmj} \, 
q^{-\frac{m}{4}(2r'+1)(2r'-4p+1)}
\theta_{m,m}^{(-)}(\tau, z)
\end{eqnarray*}}
Then, noticing that 
{\allowdisplaybreaks
\begin{eqnarray*}
& & \hspace{-5mm}
C_2''+C_4'' \,\ = \,\ 0 
\\[2mm]
& & \hspace{-5mm}
C_1+C_4' \, = \, 
\bigg[
\sum_{\substack{j, \, r \in \zzz \\[1mm] 0 \leq r <j}}
- 
\sum_{\substack{j, \, r \in \zzz \\[1mm] j \leq r <0}}\bigg]
\sum_{\substack{k \in \frac12 \zzz_{\rm odd} \\[1mm] 0< k \leq m}} 
\hspace{-2mm}
(-1)^{j} e^{\pi ik} q^{(m+\frac12)j^2+2pmj} 
q^{-\frac{1}{4m}(2mr+k)^2-p(2mr+k)}
\\[-8mm]
& & \hspace{90mm}
\times \,  
\big[\theta_{k,m}^{(-)}-\theta_{-k,m}^{(-)}\big](\tau, z)
\\[2mm]
& & \hspace{-5mm}
C_2'+C_3 \, = \, 
\bigg[
\sum_{\substack{j, \, r \in \zzz \\[1mm] 0 \leq r \leq j}}
- 
\sum_{\substack{j, \, r \in \zzz \\[1mm] j< r < 0}} \bigg]
\sum_{\substack{k \in \frac12 \zzz_{\rm odd} \\[1mm] 0< k < m}}
(-1)^{j} e^{\pi ik} 
q^{(m+\frac12)j^2+2pmj} 
q^{-\frac{1}{4m}(2mr-k)^2-p(2mr-k)}
\\[-8mm]
& & \hspace{90mm}
\times \, 
\big[\theta_{k,m}^{(-)}-\theta_{-k,m}^{(-)}\big](\tau, z)
\\[2mm]
& &
- \sum_{j\in \zzz} 
\sum_{\substack{k \in \frac12 \zzz_{\rm odd} \\[1mm] 0< k < m}} 
\hspace{-2mm}
(-1)^{j} e^{\pi ik} 
q^{(m+\frac12)j^2+2pmj} 
q^{-\frac{1}{4m}k^2+pk}
\big[\theta_{k,m}^{(-)}-\theta_{-k,m}^{(-)}\big](\tau, z),
\end{eqnarray*}}
we have 
{\allowdisplaybreaks
\begin{eqnarray}
& & \hspace{-10mm}
{\rm (VI)} \, = \, 
\bigg[
\sum_{\substack{j, \, r \in \zzz \\[1mm] 0 \leq r <j}}
- 
\sum_{\substack{j, \, r \in \zzz \\[1mm] j \leq r <0}}\bigg]
\sum_{\substack{k \in \frac12 \zzz_{\rm odd} \\[1mm] 0< k \leq m}} 
\hspace{-2mm}
(-1)^{j} e^{\pi ik} q^{(m+\frac12)j^2+2pmj} 
q^{-\frac{1}{4m}(2mr+k)^2-p(2mr+k)}
\nonumber
\\[-8mm]
& & \hspace{80mm}
\times \, 
\big[\theta_{k,m}^{(-)}-\theta_{-k,m}^{(-)}\big](\tau, z)
\nonumber
\\[2mm]
& &
+ \bigg[
\sum_{\substack{j, \, r \in \zzz \\[1mm] 0 \leq r \leq j}}
- 
\sum_{\substack{j, \, r \in \zzz \\[1mm] j< r < 0}} \bigg]
\sum_{\substack{k \in \frac12 \zzz_{\rm odd} \\[1mm] 0< k < m}}
(-1)^{j} e^{\pi ik} 
q^{(m+\frac12)j^2+2pmj} 
q^{-\frac{1}{4m}(2mr-k)^2-p(2mr-k)}
\nonumber
\\[-8mm]
& & \hspace{80mm}
\times \, 
\big[\theta_{k,m}^{(-)}-\theta_{-k,m}^{(-)}\big](\tau, z)
\nonumber
\\[2mm]
& &
- \sum_{j\in \zzz} 
\sum_{\substack{k \in \frac12 \zzz_{\rm odd} \\[1mm] 0< k < m}} 
\hspace{-2mm}
(-1)^{j} e^{\pi ik} 
q^{(m+\frac12)j^2+2pmj} 
q^{-\frac{1}{4m}k^2+pk}
\big[\theta_{k,m}^{(-)}-\theta_{-k,m}^{(-)}\big](\tau, z)
\label{m1:eqn:2022-527h}
\end{eqnarray}}
Now the formulas \eqref{m1:eqn:2022-527e},
\eqref{m1:eqn:2022-527f}, \eqref{m1:eqn:2022-527g} and 
\eqref{m1:eqn:2022-527h} complete the proof of 3).
\end{proof}

\medskip 

To go further, we note also the following:

\begin{lemma}
\label{m1:lemma:2022-524d}
Let $m \in \frac12 \nnn_{\rm odd}$ and $p \in \zzz_{\geq 0}$. Then 
\begin{subequations}
\begin{enumerate}
\item[{\rm 1)}] \,\ For \, $(z_1, z_2) \,\ = \,\ 
(\frac{z}{2}+\frac{\tau}{2}-\frac12, \frac{z}{2}-\frac{\tau}{2}+\frac12)$,
{\allowdisplaybreaks
\begin{eqnarray}
& & \hspace{-10mm}
\Phi^{(-)[m,\frac12]}(\tau, \, z_1+p\tau, \, z_2-p\tau, \, 0) 
\,\ = \,\ 
q^{mp(p+1)} \, \Bigg\{
\Phi^{(-)[m,\frac12]}(\tau, z_1,z_2,0)
\nonumber
\\[1mm]
& & 
- \sum_{-p< r \leq p} 
\sum_{\substack{k \in \frac12 \zzz_{\rm odd} \\[1mm] 0< k < m}} 
e^{\pi ik} \, 
q^{- \frac{(2mr+k)(2m(r-1)+k)}{4m}} \, 
\big[\theta^{(-)}_{k,m} \, - \, \theta^{(-)}_{-k,m}\big] (\tau, z)
\nonumber
\\[1mm]
& &
- \,\ q^{-mp(p+1)}
\sum_{\substack{k \in \frac12 \zzz_{\rm odd} \\[1mm] 0< k < m}} 
e^{\pi ik} \, 
q^{-\frac{1}{4m}k^2+ (2p+1)\frac{k}{2}} \, 
\big[\theta^{(-)}_{k,m} \, - \, \theta^{(-)}_{-k,m}\big] (\tau, z)
\nonumber
\\[1mm]
& &
+ \sum_{\substack{k \in \frac12 \zzz_{\rm odd} \\[1mm] 0< k < m}} 
e^{\pi ik} \, 
q^{\frac{k(2m-k)}{4m}} \, 
\big[\theta^{(-)}_{k,m} \, - \, \theta^{(-)}_{-k,m}\big] (\tau, z)
\nonumber
\\[0mm]
& & - \,\ 
e^{\pi im} \sum_{-p \leq r \leq  p} 
q^{- \frac{m}{4}(2r-1)(2r+1)} \, \theta^{(-)}_{m,m}(\tau, z)
\, + \, e^{\pi im} \, 
q^{\frac{m}{4}} \, \theta^{(-)}_{m,m}(\tau, z) \Bigg\}
\label{m1:eqn:2022-525c1}
\end{eqnarray}}
\item[{\rm 2)}] \,\ For \, $(z_1, z_2) \,\ = \,\
(\frac{z}{2}+\frac{\tau}{2}, \frac{z}{2}-\frac{\tau}{2})$,
{\allowdisplaybreaks
\begin{eqnarray}
& & \hspace{-10mm}
\Phi^{(-)[m,\frac12]}(\tau, \, z_1+p\tau, \, z_2-p\tau, \, 0) 
\,\ = \,\ 
(-1)^p \, q^{mp(p+1)} \, \Bigg\{
\Phi^{(-)[m,\frac12]}(\tau, z_1,z_2,0)
\nonumber
\\[1mm]
& & 
+ \sum_{-p< r \leq p} 
\sum_{\substack{k \in \frac12 \zzz_{\rm odd} \\[1mm] 0< k < m}} 
(-1)^r \, 
q^{- \frac{(2mr+k)(2m(r-1)+k)}{4m}} \, 
\big[\theta^{(-)}_{k,m} \, - \, \theta^{(-)}_{-k,m}\big] (\tau, z)
\nonumber
\\[1mm]
& &
+ \,\ (-1)^p \, q^{-mp(p+1)}
\sum_{\substack{k \, \in \, \frac12 \zzz_{\rm odd} \\[1mm] 0< k < m}} 
q^{-\frac{1}{4m}k^2+ (2p+1)\frac{k}{2}} \, 
\big[\theta^{(-)}_{k,m} \, - \, \theta^{(-)}_{-k,m}\big] (\tau, z)
\nonumber
\\[1mm]
& &
- \sum_{\substack{k \in \frac12 \zzz_{\rm odd} \\[1mm] 0<k < m}} 
q^{\frac{k(2m-k)}{4m}} \, 
\big[\theta^{(-)}_{k,m} \, - \, \theta^{(-)}_{-k,m}\big] (\tau, z)
\nonumber
\\[0mm]
& & + 
\sum_{-p \leq r \leq p} (-1)^r \, 
q^{- \frac{m}{4}(2r-1)(2r+1)} \, \theta^{(-)}_{m,m}(\tau, z)
\,\ - \,\ 
q^{\frac{m}{4}} \, \theta^{(-)}_{m,m}(\tau, z) \Bigg\}
\label{m1:eqn:2022-525c2}
\end{eqnarray}}
\item[{\rm 3)}] \,\ For \, $(z_1, z_2) \,\ = \,\ 
(\frac{z}{2}-\frac12, \frac{z}{2}+\frac12)$,
{\allowdisplaybreaks
\begin{eqnarray}
& & \hspace{-10mm}
\Phi^{(-)[m,\frac12]}(\tau, \, z_1+p\tau, \, z_2-p\tau, \, 0) 
\,\ = \,\ 
q^{mp(p+1)} \, \Bigg\{
\Phi^{(-)[m,\frac12]}(\tau, z_1,z_2,0)
\nonumber
\\[1mm]
& & 
- \sum_{-p< r < p} 
\sum_{\substack{k \, \in \, \frac12 \zzz_{\rm odd} \\[1mm] 0<k < m}} 
e^{\pi ik} \, 
q^{- \frac{(2mr+k)^2}{4m}} \, 
\big[\theta^{(-)}_{k,m} \, - \, \theta^{(-)}_{-k,m}\big] (\tau, z)
\nonumber
\\[1mm]
& &
- \,\ q^{-mp^2}
\sum_{\substack{k \in \frac12 \zzz_{\rm odd} \\[1mm] 0<k < m}} 
e^{\pi ik} \, 
q^{-\frac{1}{4m}k^2+ pk} \, 
\big[\theta^{(-)}_{k,m} \, - \, \theta^{(-)}_{-k,m}\big] (\tau, z)
\nonumber
\\[0mm]
& & - \,\ 
e^{\pi im} \sum_{-p \leq r< p} 
q^{- \frac{m}{4}(2r+1)^2} \, \theta^{(-)}_{m,m}(\tau, z)
\Bigg\}
\label{m1:eqn:2022-525c3}
\end{eqnarray}}
\end{enumerate}
\end{subequations}
\end{lemma}

\begin{proof} By Lemma \ref{m1:lemma:2022-524b}, we have 
\begin{subequations}
{\allowdisplaybreaks
\begin{eqnarray}
& & \hspace{-20mm}
\Phi^{(-)[m,\frac12]}(\tau, z_1+p\tau, z_2-p\tau, 0) 
\nonumber
\\[2mm]
& & \hspace{-15mm}
= \,\ 
\big\{\Phi^{(-)[m,\frac12]}(\tau, z_1, z_2, 0) - U(z_1, z_2) \big\}
\nonumber
\\[2mm]
& & \hspace{-15mm}
\times \left\{
\begin{array}{rcl}
q^{mp(p+1)} & \text{if} &
(z_1, z_2)=(\frac{z}{2}+\frac{\tau}{2}-\frac12, 
\frac{z}{2}-\frac{\tau}{2}+\frac12) \, \text{or} \, 
(\frac{z}{2}-\frac12, \frac{z}{2}+\frac12)
\\[3mm]
(-1)^p q^{mp(p+1)} & \text{if} & 
(z_1, z_2)=(\frac{z}{2}+\frac{\tau}{2}, \frac{z}{2}-\frac{\tau}{2})
\end{array} \right.
\label{m1:eqn:2022-528a1}
\end{eqnarray}}
where 
{\allowdisplaybreaks
\begin{eqnarray}
& & \hspace{-2mm}
U(\tfrac{z}{2}+\tfrac{\tau}{2}-\tfrac12, 
\tfrac{z}{2}-\tfrac{\tau}{2}+\tfrac12)
= \hspace{-2mm}
\sum_{\substack{k \in \frac12 \zzz_{\rm odd} \\[1mm] 0<k <2pm}}
\hspace{-3mm}
e^{\pi ik}q^{-\frac{k^2}{4m}-\frac{k}{2}}
\big[\theta_{k,m}^{(-)}-\theta_{-k,m}^{(-)}\big](\tau,z)
\overset{\substack{put \\[1mm]
}}{=:} {\rm (I)}
\nonumber
\\[0mm]
& & \hspace{-2mm}
U(\tfrac{z}{2}+\tfrac{\tau}{2}, 
\tfrac{z}{2}-\tfrac{\tau}{2})
= \hspace{-2mm}
\sum_{\substack{k \in \frac12 \zzz_{\rm odd} \\[1mm] 0<k <2pm}}
\hspace{-3mm}
q^{-\frac{k^2}{4m}-\frac{k}{2}}
\big[\theta_{k,m}^{(-)}-\theta_{-k,m}^{(-)}\big](\tau,z) \,\ 
\overset{\substack{put \\[1mm]
}}{=:} \,\ {\rm (II)}
\label{m1:eqn:2022-528a2}
\\[0mm]
& & \hspace{-2mm}
U(\tfrac{z}{2}-\tfrac12, 
\tfrac{z}{2}+\tfrac12)
= \hspace{-2mm}
\sum_{\substack{k \in \frac12 \zzz_{\rm odd} \\[1mm] 0<k <2pm}}
\hspace{-3mm}
e^{\pi ik}q^{-\frac{k^2}{4m}}
\big[\theta_{k,m}^{(-)}-\theta_{-k,m}^{(-)}\big](\tau,z) \,\ 
\overset{\substack{put \\[1mm]
}}{=:} \,\ {\rm (III)}
\nonumber
\end{eqnarray}}
\end{subequations}
Since
$$
\Big\{k \in \tfrac12 \zzz_{\rm odd} \,\ ; \,\ 0<k<2pm\Big\}
=
\bigcup_{r=0}^{p-1}
\Big\{k \in \tfrac12 \zzz_{\rm odd} \,\ ; \,\ 2rm<k<2(r+1)m\Big\}
$$
the formulas for (I) $\sim$ (III) are rewritten as follows
by using Lemma \ref{m1:lemma:2022-526b}:
{\allowdisplaybreaks
\begin{eqnarray*}
& & \hspace{-5mm}
{\rm (I)} = \sum_{r=0}^{p-1}
\sum_{\substack{k \in \frac12 \zzz_{\rm odd} \\[1mm] 0<k <2m}} \hspace{-2mm}
e^{\pi i(2mr+k)}
q^{-\frac{(2mr+k)^2}{4m}-\frac{2mr+k}{2}}
\big[\theta_{2mr+k,m}^{(-)}-\theta_{-(2mr+k),m}^{(-)}\big](\tau,z) 
\\[0mm]
& &
= \sum_{r=1}^p
\sum_{\substack{k \in \frac12 \zzz_{\rm odd} \\[1mm] 0<k <2m}}
e^{\pi ik}
q^{-\frac{(2m(r-1)+k)(2mr+k)}{4m}}
\big[\theta_{k,m}^{(-)}-\theta_{-k,m}^{(-)}\big](\tau,z) 
\\[1mm]
& & \hspace{-5mm}
{\rm (II)} = \sum_{r=0}^{p-1}
\sum_{\substack{k \in \frac12 \zzz_{\rm odd} \\[1mm] 0<k <2m}} \hspace{-2mm}
q^{-\frac{(2mr+k)^2}{4m}-\frac{2mr+k}{2}}
\big[\theta_{2mr+k,m}^{(-)}-\theta_{-(2mr+k),m}^{(-)}\big](\tau,z) 
\\[0mm]
& & \hspace{1mm}
= -\sum_{r=1}^p
\sum_{\substack{k \in \frac12 \zzz_{\rm odd} \\[1mm] 0<k <2m}}
(-1)^r
q^{-\frac{(2m(r-1)+k)(2mr+k)}{4m}}
\big[\theta_{k,m}^{(-)}-\theta_{-k,m}^{(-)}\big](\tau,z) 
\\[1mm]
& & \hspace{-5mm}
{\rm (III)} = \sum_{r=0}^{p-1}
\sum_{\substack{k \in \frac12 \zzz_{\rm odd} \\[1mm] 0<k <2m}} \hspace{-2mm}
e^{\pi i(2mr+k)}
q^{-\frac{(2mr+k)^2}{4m}}
\big[\theta_{2mr+k,m}^{(-)}-\theta_{-(2mr+k),m}^{(-)}\big](\tau,z) 
\end{eqnarray*}}

Due to the decomposition \,\ $
\sum\limits_{\substack{k \in \frac12 \zzz_{\rm odd} \\[1mm] 0<k <2m}}
= 
\sum\limits_{\substack{k \in \frac12 \zzz_{\rm odd} \\[1mm] 0<k <m}}
+ \sum\limits_{k=m}
+ \sum\limits_{\substack{k \in \frac12 \zzz_{\rm odd} \\[1mm] m<k <2m}}$,
each of (I) $\sim$ (III) decomposes into the sum of 3 parts:
$$
{\rm (I)}={\rm (I)}_A+{\rm (I)}_B+{\rm (I)}_C, \,\ 
{\rm (II)}={\rm (II)}_A+{\rm (II)}_B+{\rm (II)}_C, \,\ 
{\rm (III)}={\rm (III)}_A+{\rm (III)}_B+{\rm (III)}_C 
$$
where
{\allowdisplaybreaks
\begin{eqnarray*}
{\rm (I)}_A
&:=& 
\sum_{r=1}^p
\sum_{\substack{k \in \frac12 \zzz_{\rm odd} \\[1mm] 0<k <m}}
e^{\pi ik}
q^{-\frac{(2m(r-1)+k)(2mr+k)}{4m}}
\big[\theta_{k,m}^{(-)}-\theta_{-k,m}^{(-)}\big](\tau,z) 
\\[1mm]
{\rm (I)}_B
&:=& 
\sum_{r=1}^p
e^{\pi im}
q^{-\frac{(2m(r-1)+m)(2mr+m)}{4m}}
\big[\theta_{m,m}^{(-)}-\theta_{-m,m}^{(-)}\big](\tau,z) 
\\[1mm]
&=&
2 \, e^{\pi im}\sum_{r=1}^p
q^{-\frac{m}{4}(2r-1)(2r+1)} \, \theta_{m,m}^{(-)}(\tau,z) 
\\[1mm]
{\rm (I)}_C
&:=& 
\sum_{r=1}^p
\sum_{\substack{k \in \frac12 \zzz_{\rm odd} \\[1mm] m<k <2m}}
e^{\pi ik}
q^{-\frac{(2m(r-1)+k)(2mr+k)}{4m}}
\big[\theta_{k,m}^{(-)}-\theta_{-k,m}^{(-)}\big](\tau,z) 
\\[1mm]
{\rm (II)}_A
&:=& 
-\sum_{r=1}^p
\sum_{\substack{k \in \frac12 \zzz_{\rm odd} \\[1mm] 0<k <m}}
(-1)^r
q^{-\frac{(2m(r-1)+k)(2mr+k)}{4m}}
\big[\theta_{k,m}^{(-)}-\theta_{-k,m}^{(-)}\big](\tau,z) 
\\[1mm]
{\rm (II)}_B
&:=& 
-\sum_{r=1}^p
(-1)^r
q^{-\frac{(2m(r-1)+m)(2mr+m)}{4m}}
\big[\theta_{m,m}^{(-)}-\theta_{-m,m}^{(-)}\big](\tau,z) 
\\[1mm]
&=&
-2 \, \sum_{r=1}^p \, (-1)^r \, 
q^{-\frac{m}{4}(2r-1)(2r+1)} \, \theta_{m,m}^{(-)}(\tau,z)
\\[1mm]
{\rm (II)}_C
&:=& 
-\sum_{r=1}^p
\sum_{\substack{k \in \frac12 \zzz_{\rm odd} \\[1mm] m<k <2m}}
(-1)^r
q^{-\frac{(2m(r-1)+k)(2mr+k)}{4m}}
\big[\theta_{k,m}^{(-)}-\theta_{-k,m}^{(-)}\big](\tau,z) 
\\[1mm]
{\rm (III)}_A
&:=& 
\sum_{r=0}^{p-1}
\sum_{\substack{k \in \frac12 \zzz_{\rm odd} \\[1mm] 0<k <m}}
e^{\pi ik}
q^{-\frac{(2mr+k)^2}{4m}}
\big[\theta_{k,m}^{(-)}-\theta_{-k,m}^{(-)}\big](\tau,z) 
\\[1mm]
{\rm (III)}_B
&:=& 
\sum_{r=0}^{p-1}
e^{\pi im}
q^{-\frac{(2mr+m)^2}{4m}}
\big[\theta_{m,m}^{(-)}-\theta_{-m,m}^{(-)}\big](\tau,z) 
\\[1mm]
&=&
2 \, e^{\pi im} \, \sum_{r=1}^p \, 
q^{-\frac{m}{4}(2r-1)^2} \, \theta_{m,m}^{(-)}(\tau,z) 
\\[1mm]
{\rm (III)}_C
&:=& 
\sum_{r=0}^{p-1}
\sum_{\substack{k \in \frac12 \zzz_{\rm odd} \\[1mm] m<k <2m}}
e^{\pi ik}
q^{-\frac{(2mr+k)^2}{4m}}
\big[\theta_{k,m}^{(-)}-\theta_{-k,m}^{(-)}\big](\tau,z) 
\end{eqnarray*}}
Computing ${\rm (I)}_C$, ${\rm (II)}_C$ and ${\rm (III)}_C$ by putting 
$k=2m-k'$, we have 
{\allowdisplaybreaks
\begin{eqnarray*}
{\rm (I)}_C
&=& 
\sum_{r=1}^p
\sum_{\substack{k' \in \frac12 \zzz_{\rm odd} \\[1mm] 0<k' <m}}
e^{\pi ik'}
q^{-\frac{(2mr-k')(2m(r+1)-k')}{4m}}
\big[\theta_{k',m}^{(-)}-\theta_{-k',m}^{(-)}\big](\tau,z) 
\\[0mm]
& \hspace{-3mm}
\underset{\substack{\\[0.5mm] \uparrow \\[1mm] r=-r'
}}{=} \hspace{-3mm} &
\sum_{\substack{r' \in \zzz \\[1mm] -p \leq r' <0}}
\sum_{\substack{k' \in \frac12 \zzz_{\rm odd} \\[1mm] 0<k' <m}}
e^{\pi ik'}
q^{-\frac{(2mr'+k')(2m(r'-1)+k')}{4m}}
\big[\theta_{k',m}^{(-)}-\theta_{-k',m}^{(-)}\big](\tau,z) 
\\[1mm]
{\rm (II)}_C
&=& 
-\sum_{r=1}^p
\sum_{\substack{k' \in \frac12 \zzz_{\rm odd} \\[1mm] 0<k' <m}}
(-1)^r
q^{-\frac{(2mr-k')(2m(r+1)-k')}{4m}}
\big[\theta_{k',m}^{(-)}-\theta_{-k',m}^{(-)}\big](\tau,z)
\\[0mm]
& \hspace{-3mm}
\underset{\substack{\\[0.5mm] \uparrow \\[1mm] r=-r'
}}{=} \hspace{-3mm} & - 
\sum_{\substack{r' \in \zzz \\[1mm] -p \leq r' <0}}
\sum_{\substack{k' \in \frac12 \zzz_{\rm odd} \\[1mm] 0<k' <m}}
(-1)^{r'}
q^{-\frac{(2mr'+k')(2m(r'-1)-k')}{4m}}
\big[\theta_{k',m}^{(-)}-\theta_{-k',m}^{(-)}\big](\tau,z)
\\[1mm]
{\rm (III)}_C 
&=&
\sum_{r=0}^{p-1}
\sum_{\substack{k' \in \frac12 \zzz_{\rm odd} \\[1mm] 0<k' <m}}
e^{\pi ik'}
q^{-\frac{(2m(r+1)-k')^2}{4m}}
\big[\theta_{k',m}^{(-)}-\theta_{-k',m}^{(-)}\big](\tau,z) 
\\[0mm]
& \hspace{-3mm}
\underset{\substack{\\[0.5mm] \uparrow \\[1mm] r+1=-r'
}}{=} \hspace{-3mm} &
\sum_{\substack{r' \in \zzz \\[1mm] -p \leq r' <0}}
\sum_{\substack{k' \in \frac12 \zzz_{\rm odd} \\[1mm] 0<k' <m}}
e^{\pi ik'}
q^{-\frac{(2mr'+k')^2}{4m}}
\big[\theta_{k',m}^{(-)}-\theta_{-k',m}^{(-)}\big](\tau,z) 
\end{eqnarray*}}
Then we have
{\allowdisplaybreaks
\begin{eqnarray*}
& & \hspace{-7mm}
{\rm (I)} \,\ = \,\ 
\big\{{\rm (I)}_A+{\rm (I)}_C\big\}+{\rm (I)}_B
\\[1mm]
& & \hspace{-4mm}
=\bigg[\sum_{r=1}^p+ \hspace{-2mm}
\sum_{\substack{r \in \zzz \\[1mm] -p \leq r <0}}\bigg]
\sum_{\substack{k \in \frac12 \zzz_{\rm odd} \\[1mm] 0<k <m}}
\hspace{-2mm}
e^{\pi ik}
q^{-\frac{(2m(r-1)+k)(2mr+k)}{4m}}
\big[\theta_{k,m}^{(-)}-\theta_{-k,m}^{(-)}\big](\tau,z) 
+{\rm (I)}_B
\\[1mm]
& & \hspace{-4mm}
= 
\sum_{\substack{r \in \zzz \\[1mm] -p \leq r \leq p}}
\sum_{\substack{k \in \frac12 \zzz_{\rm odd} \\[1mm] 0<k <m}}
\hspace{-2mm}
e^{\pi ik}
q^{-\frac{(2m(r-1)+k)(2mr+k)}{4m}}
\big[\theta_{k,m}^{(-)}-\theta_{-k,m}^{(-)}\big](\tau,z) 
\\[1mm]
& &
- \sum_{\substack{k \in \frac12 \zzz_{\rm odd} \\[1mm] 0<k <m}}
\hspace{-2mm}
e^{\pi ik}
q^{\frac{k(2m-k)}{4m}}
\big[\theta_{k,m}^{(-)}-\theta_{-k,m}^{(-)}\big](\tau,z) 
\\[0mm]
& &
+ \,\ 2 \, e^{\pi im}\sum_{r=1}^p
q^{-\frac{m}{4}(2r-1)(2r+1)} \, \theta_{m,m}^{(-)}(\tau,z) \, , 
\hspace{10mm} \text{proving 1).}
\\[1mm]
& & \hspace{-7mm}
{\rm (II)} \,\ = \,\ 
\big\{{\rm (II)}_A+{\rm (II)}_C\big\}+{\rm (II)}_B
\\[1mm]
& & \hspace{-4mm}
=
-\bigg[\sum_{r=1}^p+
\sum_{\substack{r \in \zzz \\[1mm] -p \leq r <0}}\bigg]
\sum_{\substack{k \in \frac12 \zzz_{\rm odd} \\[1mm] 0<k <m}}
(-1)^r
q^{-\frac{(2m(r-1)+k)(2mr+k)}{4m}}
\big[\theta_{k,m}^{(-)}-\theta_{-k,m}^{(-)}\big](\tau,z) 
+{\rm (II)}_B
\\[1mm]
& & \hspace{-4mm}
= \, - \sum_{\substack{r \in \zzz \\[1mm] -p \leq r \leq p}}
\sum_{\substack{k \in \frac12 \zzz_{\rm odd} \\[1mm] 0<k <m}}
(-1)^r
q^{-\frac{(2m(r-1)+k)(2mr+k)}{4m}}
\big[\theta_{k,m}^{(-)}-\theta_{-k,m}^{(-)}\big](\tau,z) 
\\[1mm]
& &
+ \sum_{\substack{k \in \frac12 \zzz_{\rm odd} \\[1mm] 0<k <m}}
q^{\frac{k(2m-k)}{4m}}
\big[\theta_{k,m}^{(-)}-\theta_{-k,m}^{(-)}\big](\tau,z) 
\\[1mm]
& &
- \, 2 \, \sum_{r=1}^p \, (-1)^r \, 
q^{-\frac{m}{4}(2r-1)(2r+1)} \, \theta_{m,m}^{(-)}(\tau,z) \, ,
\hspace{10mm} \text{proving 2).}
\\[1mm]
& & \hspace{-7mm}
{\rm (III)} \,\ = \,\ 
\big\{{\rm (III)}_A+{\rm (III)}_C\big\}+{\rm (III)}_B
\\[1mm]
& & \hspace{-4mm}
=
\bigg[\sum_{r=0}^{p-1}
+ 
\sum_{\substack{r \in \zzz \\[1mm] -p \leq r <0}}\bigg]
\sum_{\substack{k \in \frac12 \zzz_{\rm odd} \\[1mm] 0<k <m}}
e^{\pi ik}
q^{-\frac{(2mr+k)^2}{4m}}
\big[\theta_{k,m}^{(-)}-\theta_{-k,m}^{(-)}\big](\tau,z) 
+{\rm (III)}_B
\\[1mm]
& & \hspace{-4mm}
= \, 
\sum_{\substack{r \in \zzz \\[1mm] -p \leq r <p}}
\sum_{\substack{k \in \frac12 \zzz_{\rm odd} \\[1mm] 0<k <m}}
e^{\pi ik}
q^{-\frac{(2mr+k)^2}{4m}}
\big[\theta_{k,m}^{(-)}-\theta_{-k,m}^{(-)}\big](\tau,z) 
\\[1mm]
& &
+ \,\ 
2 \, e^{\pi im} \, \sum_{r=1}^p \, 
q^{-\frac{m}{4}(2r-1)^2} \, \theta_{m,m}^{(-)}(\tau,z) \, ,
\hspace{10mm} \text{proving 3).} 
\end{eqnarray*}}
Thus the proof of Lemma \ref{m1:lemma:2022-524d} is completed.
\end{proof}

\medskip

Then, by Lemma \ref{m1:lemma:2022-524c} and Lemma 
\ref{m1:lemma:2022-524d}, we obtain the following:

\begin{prop}
\label{m1:prop:2022-525a}
Let $m \in \frac12 \nnn_{\rm odd}$ and $p \in \zzz_{\geq 0}$. Then
\begin{enumerate}
\item[{\rm 1)}] \,\ For \,\ $(z_1,z_2) \,\ = \,\ (
\frac{z}{2}+\frac{\tau}{2}-\frac12, \, 
\frac{z}{2}-\frac{\tau}{2}+\frac12)$,
{\allowdisplaybreaks
\begin{eqnarray}
& & \hspace{-3mm}
q^{-\frac{m}{4}} \, 
\theta_{(2p+1)m, m+\frac12}^{(-)}(\tau,0) 
\Bigg\{
\Phi^{(-)[m,\frac12]}(\tau, z_1,z_2,0) 
\nonumber
\\[1mm]
& &+ 
\sum_{\substack{k \in \frac12 \zzz_{\rm odd} \\[1mm]0< k < m}} 
\hspace{-2mm}
e^{\pi ik} \, 
q^{\frac{k(2m-k)}{4m}} \, 
\big[\theta^{(-)}_{k,m} \, - \, \theta^{(-)}_{-k,m}\big] (\tau, z)
\,\ + \,\ e^{\pi im} \, 
q^{\frac{m}{4}} \, \theta^{(-)}_{m,m}(\tau, z)
\Bigg\}
\nonumber
\\[1mm]
& & \hspace{-7mm}
= \,\ 
- \,\ i \, \eta(\tau)^3 \cdot 
\frac{\theta_{p+\frac12,m+\frac12}^{(-)}(\tau,z)
\, - \, 
\theta_{-(p+\frac12),m+\frac12}^{(-)}(\tau,z)
}{\theta_{0, \frac12}(\tau,z)}
\nonumber
\\[1mm]
& & \hspace{-5mm}
+ 
\bigg[
\sum_{\substack{r \, \in \zzz \\[1mm] 0 \leq r < j}} 
- 
\sum_{\substack{r \, \in \zzz \\[1mm] j \leq r <0}} 
\bigg] \hspace{-6mm} 
\sum_{\hspace{7mm}
\substack{k \in \frac12 \zzz_{\rm odd} \\[1mm] 0< k \leq m}} 
\hspace{-6mm}
(-1)^{j} e^{\pi ik} 
q^{(m+\frac12)(j+\frac{m(2p+1)}{2m+1})^2 
- m (r+p+\frac{m+k}{2m})^2}
\big[\theta_{k,m}^{(-)}-\theta_{-k,m}^{(-)}\big](\tau, z)
\nonumber
\\[1mm]
& & \hspace{-5mm}
+ 
\bigg[
\sum_{\substack{r \, \in \zzz \\[1mm] 0 \leq r \leq j}} 
- 
\sum_{\substack{r \, \in \zzz \\[1mm] j<r <0}} 
\bigg] \hspace{-6mm}
\sum_{\hspace{7mm} 
\substack{k \in \frac12 \zzz_{\rm odd} \\[1mm] 0< k <m}}
\hspace{-6mm}
(-1)^{j} e^{\pi ik} 
q^{(m+\frac12)(j+\frac{m(2p+1)}{2m+1})^2 
- m (r+p+\frac{m-k}{2m})^2} 
\big[\theta_{k,m}^{(-)}-\theta_{-k,m}^{(-)}\big](\tau, z)
\nonumber
\\[1mm]
& & \hspace{-5mm}
+ \,\ 
\theta_{(2p+1)m, m+\frac12}^{(-)}(\tau,0) 
\sum_{\substack{r \, \in \zzz \\[1mm] -p< r \, \leq p}} \,\ 
\sum_{\substack{k \in \frac12 \zzz_{\rm odd} \\[1mm] 0< k < m}} 
e^{\pi ik} \, 
q^{- \, m \, (r+\frac{k-m}{2m})^2} \, 
\big[\theta^{(-)}_{k,m} \, - \, \theta^{(-)}_{-k,m}\big] (\tau, z)
\nonumber
\\[1mm]
& & \hspace{-5mm}
+ \,\ e^{\pi im} \, \theta_{(2p+1)m, m+\frac12}^{(-)}(\tau,0) 
\sum_{\substack{r \, \in \zzz \\[1mm] -p \leq r \leq p}} 
q^{-mr^2} \, \theta^{(-)}_{m,m}(\tau, z)
\label{m1:eqn:2022-525a1}
\end{eqnarray}}
\item[{\rm 2)}] \,\ For \,\ $(z_1, z_2) \,\ = \,\ (
\frac{z}{2}+\frac{\tau}{2}, \frac{z}{2}-\frac{\tau}{2})$,
{\allowdisplaybreaks
\begin{eqnarray}
& & \hspace{-10mm}
q^{-\frac{m}{4}} \, 
\theta_{(2p+1)m, m+\frac12}(\tau,0) 
\Bigg\{
\Phi^{(-)[m,\frac12]}(\tau, z_1,z_2,0) 
\nonumber
\\[1mm]
& &
-\sum_{\substack{k \in \frac12 \zzz_{\rm odd} \\[1mm] 0< k < m}} 
\hspace{-2mm}
q^{\frac{k(2m-k)}{4m}} \, 
\big[\theta^{(-)}_{k,m} \, - \, \theta^{(-)}_{-k,m}\big] (\tau, z)
\,\ - \,\ 
q^{\frac{m}{4}} \, \theta^{(-)}_{m,m}(\tau, z)
\Bigg\}
\nonumber
\\[1mm]
& & \hspace{-5mm}
= \,\ 
\eta(\tau)^3 \cdot 
\frac{\theta_{p+\frac12,m+\frac12}(\tau,z)
\, - \, 
\theta_{-(p+\frac12),m+\frac12}(\tau,z)
}{\theta_{0, \frac12}^{(-)}(\tau,z)}
\nonumber
\\[1mm]
& & \hspace{-3mm}
+ \,\ (-1)^p \, \bigg[
\sum_{\substack{r \, \in \zzz \\[1mm] 0 \leq r <j}} 
- 
\sum_{\substack{r \, \in \zzz \\[1mm] j \leq r <0}} 
\bigg] \, 
\sum_{\substack{k \in \frac12 \zzz_{\rm odd} \\[1mm]
0< k \leq m}} 
(-1)^r \, 
q^{(m+\frac12)(j+\frac{m(2p+1)}{2m+1})^2 
\, - \, m \, (r+p+\frac{m+k}{2m})^2}
\nonumber
\\[-8mm]
& & \hspace{80mm}
\times \, 
\big[\theta_{k,m}^{(-)}-\theta_{-k,m}^{(-)}\big](\tau, z)
\nonumber
\\[2mm]
& & \hspace{-3mm}
- \,\ (-1)^p \, 
\bigg[
\sum_{\substack{r \, \in \zzz \\[1mm] 0 \leq r \leq j}} 
- 
\sum_{\substack{r \, \in \zzz \\[1mm] j< r<0}} 
\bigg] \, 
\sum_{\substack{k \in \frac12 \zzz_{\rm odd} \\[1mm] 0<k <m}}
(-1)^r \,
q^{(m+\frac12)(j+\frac{m(2p+1)}{2m+1})^2 
\, - \, m \, (r+p+\frac{m-k}{2m})^2} 
\nonumber
\\[-8mm]
& & \hspace{80mm}
\times \, 
\big[\theta_{k,m}^{(-)}-\theta_{-k,m}^{(-)}\big](\tau, z)
\nonumber
\\[2mm]
& & \hspace{-3mm}
- \,\ 
\theta_{(2p+1)m, m+\frac12}(\tau,0) 
\sum_{\substack{r \, \in \zzz \\[1mm] -p< r \leq p}} \,\ 
\sum_{\substack{k \in \frac12 \zzz_{\rm odd} \\[1mm]
0< k < m}} 
(-1)^r \, 
q^{- \, m \, (r+\frac{k-m}{2m})^2} \, 
\big[\theta^{(-)}_{k,m} \, - \, \theta^{(-)}_{-k,m}\big] (\tau, z)
\nonumber
\\[1mm]
& & \hspace{-3mm}
- \,\ \theta_{(2p+1)m, m+\frac12}(\tau,0) 
\sum_{\substack{r \, \in \zzz \\[1mm] -p \leq r \leq p}} 
(-1)^r \, 
q^{-mr^2} \, \theta^{(-)}_{m,m}(\tau, z)
\label{m1:eqn:2022-525a2}
\end{eqnarray}}
\item[{\rm 3)}] \,\ For \,\ $(z_1,z_2) \,\ = \,\ (
\frac{z}{2}-\frac12, \frac{z}{2}+\frac12)$,
{\allowdisplaybreaks
\begin{eqnarray}
& & \hspace{-3mm}
\theta_{2pm, m+\frac12}^{(-)}(\tau,0) \, 
\Phi^{(-)[m,\frac12]}(\tau, z_1,z_2,0)
\nonumber
\\[1mm]
& & \hspace{-7mm}
= \,\ 
- \,\ i \, \eta(\tau)^3 \,\ 
\frac{\theta_{p,m+\frac12}^{(-)}(\tau,z)
\, - \, 
\theta_{-p,m+\frac12}^{(-)}(\tau,z)
}{\theta_{\frac12, \frac12}(\tau,z)}
\nonumber
\\[1mm]
& & \hspace{-5mm}
+ \bigg[
\sum_{\substack{r \, \in \zzz \\[1mm] 0 \leq r <j}} 
- 
\sum_{\substack{r \, \in \zzz \\[1mm] j \leq r <0}} 
\bigg] \hspace{-6mm}
\sum_{\hspace{7mm}
\substack{k \in \frac12 \zzz_{\rm odd} \\[1mm] 0< k \leq m}} 
\hspace{-6mm}
(-1)^{j} e^{\pi ik} \, 
q^{(m+\frac12)(j+\frac{2mp}{2m+1})^2 \, - \, m (r+p+\frac{k}{2m})^2}
\big[\theta_{k,m}^{(-)}-\theta_{-k,m}^{(-)}\big](\tau, z)
\nonumber
\\[2mm]
& & \hspace{-5mm}
+ \bigg[
\sum_{\substack{r \, \in \zzz \\[1mm] 0 \leq r \leq j}} 
- 
\sum_{\substack{r \, \in \zzz \\[1mm] j < r <0}} 
\bigg] \hspace{-6mm}
\sum_{\hspace{7mm}
\substack{k \in \frac12 \zzz_{\rm odd} \\[1mm] 0< k < m}}
\hspace{-6mm}
(-1)^{j} e^{\pi ik} \, 
q^{(m+\frac12)(j+\frac{2mp}{2m+1})^2 \, - \, m (r+p-\frac{k}{2m})^2}
\big[\theta_{k,m}^{(-)}-\theta_{-k,m}^{(-)}\big](\tau, z)
\nonumber
\\[2mm]
& & \hspace{-5mm}
+ \,\ \theta_{2pm, m+\frac12}^{(-)}(\tau,0) 
\sum_{\substack{r \, \in \zzz \\[1mm] -p < r <p}} \,\ 
\sum_{\substack{k \in \frac12 \zzz_{\rm odd} \\[1mm] 0<k < m}} 
e^{\pi ik} \, 
q^{- \frac{(2mr+k)^2}{4m}} \, 
\big[\theta^{(-)}_{k,m} \, - \, \theta^{(-)}_{-k,m}\big] (\tau, z)
\nonumber
\\[1mm]
& & \hspace{-5mm}
+ \,\ 
2 \, e^{\pi im} \, \theta_{2pm, m+\frac12}^{(-)}(\tau,0) 
\sum_{r=0}^{p-1} 
q^{- \frac{m}{4}(2r+1)^2} \, \theta^{(-)}_{m,m}(\tau, z)
\label{m1:eqn:2022-525a3}
\end{eqnarray}}
\end{enumerate}
\end{prop}

\section{Calculation of $R_{j,m}^{(\pm)}(\tau, a\tau+b)$}

In order to compute $R_{j,m}^{(\pm)}(\tau, a\tau+b)$, we consider functions 
$P_{j,m}^{(\pm)}(\tau, w)$ and $Q_{j,m}^{(\pm)}(\tau, w)$ defined by 
\begin{subequations}
{\allowdisplaybreaks
\begin{eqnarray}
P_{j,m}^{(\pm)}(\tau, w) &:=& \sum_{k \in \zzz}(\pm1)^k
E\bigg(\Big(j+2mk-2m \frac{{\rm Im}(w)}{{\rm Im}(\tau)}\Big)
\sqrt{\frac{{\rm Im}(\tau)}{m}}\bigg) 
e^{-\frac{\pi i}{2m}(j+2mk)^2\tau+2\pi i(j+2mk)w}
\nonumber
\\[-4mm]
& &
\label{m1:eqn:2022-522b1}
\\[0mm]
Q_{j,m}^{(\pm)}(\tau, w) &:=& \hspace{-5mm}
\sum_{n \equiv j \, {\rm mod} \, 2m} \hspace{-3mm}
(\pm1)^{\frac{n-j}{2m}} {\rm sgn}\Big(n-\frac12-j+2m\Big)
e^{-\frac{\pi in^2}{2m}\tau+2\pi inw}
\nonumber
\\[1mm]
&=&
\Big[\sum_{k \in \zzz_{\geq 0}}-\sum_{k \in \zzz_{<0}}\Big]
(\pm1)^k
e^{-\frac{\pi i}{2m}(j+2mk)^2\tau+2\pi i(j+2mk)w}
\label{m1:eqn:2022-522b2}
\end{eqnarray}}
so that
\begin{equation}
R_{j,m}^{(\pm)}(\tau, w) \, = \, 
P_{j,m}^{(\pm)}(\tau, w) \, + \, Q_{j,m}^{(\pm)}(\tau, w)
\label{m1:eqn:2022-522b3}
\end{equation}
\end{subequations}
Then it is easy to see the following:
\begin{subequations}
{\allowdisplaybreaks
\begin{eqnarray}
P^{(\pm)}_{j+2m,m}(\tau, w) &=& \pm \, P^{(\pm)}_{j,m}(\tau,w)
\label{m1:eqn:2022-523e1}
\\[1mm]
Q^{(\pm)}_{j+2m,m}(\tau, w) &=& \pm \, Q^{(\pm)}_{j,m}(\tau,w)
\, \mp \, 2 \, e^{-\frac{\pi ij^2}{2m}\tau+2\pi ijw}
\label{m1:eqn:2022-523e2}
\end{eqnarray}}
\end{subequations}

\vspace{-3mm}

\noindent
It is also easy to see the following formulas hold for $a, b \in \rrr$ :
\begin{subequations}
{\allowdisplaybreaks
\begin{eqnarray}
& & \hspace{-5mm}
P^{(\pm)}_{j,m}(\tau, a\tau+b) \,\ = \,\ 
e^{2\pi ijb}\sum\limits_{k \in \zzz} \, 
(\pm1)^k e^{4\pi imbk} 
\nonumber
\\[-1mm]
& & \hspace{20mm}
\times \,\ 
E \bigg(\big(j+2m(k-a)\big)
\sqrt{\frac{{\rm Im}(\tau)}{m}}\bigg) 
e^{-\frac{1}{4m}(j+2mk)(j-2m(2a-k))}
\label{m1:eqn:2022-523f1}
\\[1mm]
& & \hspace{-5mm}
Q^{(\pm)}_{j,m}(\tau, a\tau+b) \,\ = \,\ 
e^{2\pi ijb}\sum\limits_{k \in \zzz}
(\pm1)^k e^{4\pi imbk} 
{\rm sgn}\Big(2m(k+1)-\frac12\Big) \, 
e^{-\frac{1}{4m}(j+2mk)(j-2m(2a-k))}
\nonumber
\\[-4mm]
& &
\label{m1:eqn:2022-523f2}
\end{eqnarray}}
\end{subequations}
We claim the following:

\begin{lemma}
\label{m1:lemma:2022-523b}
Let $m \in \frac12 \nnn$,  $j \in \frac12 \zzz$,  
$a \in \frac12 \zzz_{\geq 0}$ and $b \in \rrr$
such that $4mb \in \zzz$. Then
\begin{enumerate}
\item[{\rm 1)}]
\begin{enumerate}
\item[{\rm (i)}] \, $P^{(+)}_{j,m}(\tau, a\tau+b)
\, + \, e^{4\pi ijb}e^{8\pi imab} P^{(+)}_{2m-j,m}(\tau, a\tau+b)
\,\ = \,\ 0$
\item[{\rm (ii)}] \, $P^{(-)}_{j,m}(\tau, a\tau+b)
\, - \, (-1)^{2a}e^{4\pi ijb}e^{8\pi imab} P^{(-)}_{2m-j,m}(\tau, a\tau+b)
\,\ = \,\ 0$
\end{enumerate}
\item[{\rm 2)}]
\begin{enumerate}
\item[{\rm (i)}]  \, $Q^{(+)}_{j,m}(\tau, a\tau+b)
\, + \, e^{4\pi ijb}e^{8\pi imab} Q^{(+)}_{2m-j,m}(\tau, a\tau+b)$
$$
= \,\ 2 \, e^{2\pi ijb}e^{8\pi imab}
\sum_{\substack{k \in \zzz \\[1mm] 1 \leq k \leq 2a}}
e^{4\pi imbk}
q^{-\frac{1}{4m}(j+2m(2a-k))(j-2mk)}
$$
\item[{\rm (ii)}]  \, $Q^{(-)}_{j,m}(\tau, a\tau+b)
\, - \, (-1)^{2a}e^{4\pi ijb}e^{8\pi imab} Q^{(-)}_{2m-j,m}(\tau, a\tau+b)$
$$
= \,\ 2 \, (-1)^{2a}e^{2\pi ijb}e^{8\pi imab}
\sum_{\substack{k \in \zzz \\[1mm]
1 \leq k \leq 2a}}
(-1)^k e^{4\pi imbk}
q^{-\frac{1}{4m}(j+2m(2a-k))(j-2mk)}
$$
\end{enumerate}
\end{enumerate}
\end{lemma}

\begin{proof} 1) Letting $j \rightarrow -j$ and 
$k=2a-k'$ in \eqref{m1:eqn:2022-523f1}, we have 
{\allowdisplaybreaks
\begin{eqnarray}
& & \hspace{-3mm}
P_{-j,m}^{(\pm)}(\tau, a\tau+b) \,\ = \,\ 
e^{-2\pi ijb} \, e^{8\pi imab}
\nonumber
\\[1mm]
& &
\times 
\sum_{k' \in \zzz} (\pm 1)^{2a-k'} 
e^{-4\pi imbk'} 
E\bigg(\big(-j+2m(a-k') \big) 
\sqrt{\frac{{\rm Im}(\tau)}{m}} \bigg)
q^{- \, \frac{1}{4m} (-j+2m(2a-k')) (-j-2mk')}
\nonumber
\\[1mm]
&=&
- \, (\pm 1)^{2a} \, e^{-2\pi ijb} \, e^{8\pi imab}
\nonumber
\\[0mm]
& &
\times \, 
\underbrace{\sum_{k' \in \zzz} (\pm 1)^{k'} e^{4\pi imbk'}
E\bigg(\big(j+2m(k'-a) \big) \,
\sqrt{\frac{{\rm Im}(\tau)}{m}} \bigg) \, 
q^{- \, \frac{1}{4m} \, (j-2m(2a-k')) (j+2mk')}}_{
\substack{|| \\[0mm] {\displaystyle 
e^{-2\pi ijb} \, P_{j,m}^{(\pm)}(\tau, a\tau+b) 
}}}
\nonumber
\\[0mm]
&=&
- \, (\pm 1)^{2a} \, e^{-4\pi ijb} \, e^{8\pi imab} \, 
P_{j,m}^{(\pm)}(\tau, a\tau+b) 
\label{m1:eqn:2022-915a}
\end{eqnarray}}
Then we have
{\allowdisplaybreaks
\begin{eqnarray*}
P_{j,m}^{(\pm)}(\tau, a\tau+b) 
&=&
- \, (\pm 1)^{2a} \, e^{4\pi ijb} \, 
\underbrace{e^{-8\pi imab}}_{\substack{|| \\[0mm] 
{\displaystyle e^{8\pi imab}
}}} 
P_{-j,m}^{(\pm)}(\tau, a\tau+b)
\\[1mm]
&=&
\mp \, (\pm1)^{2a} \, e^{4\pi ijb} \, e^{8\pi imab} \, 
P_{2m-j,m}^{(\pm)}(\tau, a\tau+b)
\end{eqnarray*}}
by \eqref{m1:eqn:2022-523e1}, proving 1).

\medskip

\noindent
2) Decomposing the sum $\sum_{k \in \zzz}$ into two parts
in the RHS of \eqref{m1:eqn:2022-523f2}, we have
$$
Q_{j,m}^{(\pm)}(\tau, a\tau+b) \,\ = \,\ 
{\rm (I)}_j^{(\pm)} \, + \, {\rm (II)}_j^{(\pm)}
$$
where
{\allowdisplaybreaks
\begin{eqnarray*}
{\rm (I)}_j^{(\pm)} &:=&
e^{2\pi ijb} 
\sum_{k \, \in \zzz_{\geq 0}} \, (\pm 1)^k \, e^{4\pi imbk} \, 
q^{- \, \frac{1}{4m} \, (j+2mk) \big(j-2m(2a-k)\big)}
\\[0mm]
{\rm (II)}_j^{(\pm)} &:=&- \, 
e^{2\pi ijb} 
\sum_{k \, \in \zzz_{< 0}} \, (\pm 1)^k \, e^{4\pi imbk} \, 
q^{- \, \frac{1}{4m} \, (j+2mk) \big(j-2m(2a-k)\big)}
\end{eqnarray*}}
First we compute ${\rm (II)}_j^{(\pm)}$ by putting $k=2a-k'$:
\begin{subequations}
{\allowdisplaybreaks
\begin{eqnarray}
{\rm (II)}_j^{(\pm)} &=&
- \, e^{2\pi ijb} 
\sum_{\substack{k' \in \zzz \\[1mm] k' >2a}} 
(\pm 1)^{k'+2a}  
\underbrace{e^{4\pi imb(2a-k')}}_{\substack{|| \\[0mm] 
{\displaystyle 
e^{8\pi imb2a}e^{4\pi imbk'}
}}}  
q^{- \, \frac{1}{4m} \, (j+2m(2a-k')) (j-2mk')}
\nonumber
\\[0mm]
&=&
- \, (\pm 1)^{2a} \, e^{2\pi ijb} \, e^{8\pi imab}
\underbrace{\sum_{k' \in \zzz_{\geq 0}} 
(\pm 1)^{k'} \, e^{4\pi imbk'} \, 
q^{- \, \frac{1}{4m} \, (-j-2m(2a-k')) (-j+2mk')}}_{
\substack{|| \\[0mm] {\displaystyle 
e^{2\pi ijb} \, \times \, {\rm (I)}_{-j}^{(\pm)}
}}}
\nonumber
\\[0mm]
& &
+ \,\ (\pm 1)^{2a} \, e^{2\pi ijb} \, e^{8\pi imab}
\sum_{\substack{k' \in \zzz \\[1mm] 0 \leq k' \leq 2a}} 
(\pm 1)^{k'} 
e^{4\pi imbk'} \, q^{- \, \frac{1}{4m} \, (-j-2m(2a-k')) (-j+2mk')}
\nonumber
\\[1mm]
&=&
- \, (\pm 1)^{2a} \, e^{4\pi ijb} e^{8\pi imab} \times {\rm (I)}_{-j}^{(\pm)}
\nonumber
\\[1mm]
& &+ \,\ 
(\pm 1)^{2a} \, e^{2\pi ijb} e^{8\pi imab} \hspace{-3mm}
\sum_{\substack{k \, \in \zzz \\[1mm] 0 \, \leq \, k \, \leq \, 2a}} 
\hspace{-3mm} 
(\pm 1)^{k}e^{4\pi imbk} q^{- \frac{1}{4m} (j+2m(2a-k)) (j-2mk)}
\label{m1:eqn:2022-602a1}
\\[3mm]
{\rm (II)}_{-j}^{(\pm)} \hspace{-2mm} &=& \text{
letting $j\rightarrow -j$ in \eqref{m1:eqn:2022-602a1}}
\nonumber
\\[1mm]
&=& 
- \, (\pm 1)^{2a} \, e^{-4\pi ijb} e^{8\pi imab} \times {\rm (I)}_{j}^{(\pm)}
\nonumber
\\[1mm]
& &
+ \,\ (\pm 1)^{2a} \, e^{-2\pi ijb} e^{8\pi imab} \hspace{-3mm}
\sum_{\substack{k \, \in \zzz \\[1mm] 0 \, \leq \, k \, \leq \, 2a}} 
\hspace{-3mm} (\pm 1)^{k}
e^{4\pi imbk} q^{- \frac{1}{4m} (-j+2m(2a-k)) (-j-2mk)}
\nonumber
\\[3mm]
&=&
- \, (\pm 1)^{2a} \, e^{-4\pi ijb} e^{8\pi imab} \times {\rm (I)}_{j}^{(\pm)}
\nonumber
\\[1mm]
& &
+ \,\ (\pm 1)^{2a} \, e^{-2\pi ijb} e^{8\pi imab} \hspace{-3mm} 
\sum_{\substack{k \, \in \zzz \\[1mm] 0 \, \leq \, k \, \leq \, 2a}} 
\hspace{-3mm} (\pm 1)^{k}
e^{4\pi imbk} q^{- \frac{1}{4m} (j-2m(2a-k)) (j+2mk)}
\label{m1:eqn:2022-602a2}
\end{eqnarray}}
We note that the formula \eqref{m1:eqn:2022-602a1} is rewritten 
as follows:
{\allowdisplaybreaks
\begin{eqnarray}
{\rm (I)}_{-j}^{(\pm)} &=& 
- \, (\pm 1)^{2a} \, e^{-4\pi ijb} 
\underbrace{e^{-8\pi imab}}_{\substack{|| \\[0mm] 
{\displaystyle e^{8\pi imab}
}}} 
\times {\rm (II)}_j^{(\pm)}
\nonumber
\\[0mm]
& &
+ \,\ e^{-2\pi ijb} \hspace{-3mm}
\sum_{\substack{k \, \in \zzz \\[1mm] 0 \, \leq \, k \, \leq \, 2a}} 
\hspace{-3mm} (\pm 1)^{k}
e^{4\pi imbk} q^{- \frac{1}{4m} (j+2m(2a-k)) (j-2mk)}
\label{m1:eqn:2022-602a3}
\end{eqnarray}}
\end{subequations}
Then, by \eqref{m1:eqn:2022-602a2} and \eqref{m1:eqn:2022-602a3},
we have
\begin{subequations}
{\allowdisplaybreaks
\begin{eqnarray}
& & \hspace{-10mm}
Q_{-j,m}^{(\pm)}(\tau, a\tau+b) \, = \, 
{\rm (I)}_{-j}^{(\pm)} + {\rm (II)}_{-j}^{(\pm)}
\nonumber
\\[2mm]
&=& 
- \, (\pm 1)^{2a} \, e^{-4\pi ijb} 
\underbrace{e^{8\pi imab}}_{\substack{|| \\[0mm] {\displaystyle 
e^{-8\pi imab}
}}} \,\ \times \,\ \big\{ \hspace{-2mm}
\underbrace{{\rm (I)}_{j}^{(\pm)} + {\rm (II)}_{j}^{(\pm)}
}_{\substack{|| \\[-1mm] {\displaystyle \hspace{5mm}
Q_{j,m}^{(\pm)}(\tau, a\tau+b)
}}} \hspace{-2mm}
\big\}
\nonumber
\\[1mm]
& &
+ \,\ 
e^{-2\pi ijb} \hspace{-2mm}
\sum_{\substack{k \, \in \zzz \\[1mm] 0 \, \leq \, k \, \leq \, 2a}} 
\hspace{-2mm} (\pm 1)^{k}
e^{4\pi imbk} q^{- \frac{1}{4m} (j+2m(2a-k)) (j-2mk)}
\nonumber
\\[1mm]
& &
+ \,\ (\pm 1)^{2a} \, e^{-2\pi ijb} e^{8\pi imab} \hspace{-2mm}
\sum_{0 \, \leq \, k \, \leq \, 2a} \hspace{-2mm} (\pm 1)^{k}
e^{4\pi imbk} q^{- \frac{1}{4m} (j-2m(2a-k)) (j+2mk)}
\label{m1:eqn:2022-602b1}
\end{eqnarray}}
namely,
{\allowdisplaybreaks
\begin{eqnarray}
& & \hspace{-10mm}
Q_{j,m}^{(\pm)}(\tau, a\tau+b)
\,\ = \,\ 
- \, (\pm 1)^{2a} \, e^{4\pi ijb} \, e^{8\pi imab} \, 
Q_{-j,m}^{(\pm)}(\tau, a\tau+b)
\nonumber
\\[3mm]
& &
+ \,\ (\pm 1)^{2a} \, 
e^{2\pi ijb} \, e^{8\pi imab}\hspace{-3mm}
\sum_{\substack{k , \in \zzz \\[1mm] 0 \, \leq \, k \, \leq \, 2a}} 
\hspace{-3mm} (\pm 1)^{k}
e^{4\pi imbk} q^{- \frac{1}{4m} (j+2m(2a-k)) (j-2mk)}
\nonumber
\\[1mm]
& &
+ \,\ e^{2\pi ijb}  \hspace{-3mm}
\sum_{\substack{k \, \in \zzz \\[1mm] 0 \, \leq \, k \, \leq \, 2a}} 
\hspace{-3mm} (\pm 1)^{k}
e^{4\pi imbk} q^{- \frac{1}{4m} (j-2m(2a-k)) (j+2mk)}
\label{m1:eqn:2022-602b2}
\\[0mm]
&=&
- \, (\pm 1)^{2a} \, e^{4\pi ijb} \, e^{8\pi imab} \, 
Q_{-j,m}^{(\pm)}(\tau, a\tau+b)
\nonumber
\\[1mm]
& &
+ \,\ (\pm 1)^{2a} \, 
e^{2\pi ijb} \, e^{8\pi imab}\hspace{-3mm}
\sum_{\substack{k \, \in \zzz \\[1mm] 1 \, \leq \, k \, \leq \, 2a}} 
\hspace{-3mm} (\pm 1)^{k}
e^{4\pi imbk} q^{- \frac{1}{4m} (j+2m(2a-k)) (j-2mk)}
\nonumber
\\[1mm]
& &
+ \,\ e^{2\pi ijb}  \hspace{-3mm}
\sum_{\substack{k \, \in \zzz \\[1mm] 0 \, \leq \, k \, < \, 2a}} 
\hspace{-3mm} (\pm 1)^{k}
e^{4\pi imbk} q^{- \frac{1}{4m} (j-2m(2a-k)) (j+2mk)}
\nonumber
\\[0mm]
& &+ \,\ 2 \, (\pm 1)^{2a} \, 
e^{2\pi ijb} \, e^{8\pi imab} \, q^{-\frac{1}{4m}j(j+4ma)}
\label{m1:eqn:2022-602b3}
\end{eqnarray}}
\end{subequations}
Note that, by \eqref{m1:eqn:2022-523e2}, we have 
{\allowdisplaybreaks
\begin{eqnarray*}
Q_{-j,m}^{(\pm)}(\tau,a\tau+b) &=& \pm \, 
Q_{2m-j,m}^{(\pm)}(\tau,a\tau+b)
\,\ + \,\ 2 \, e^{-\frac{\pi i}{2m}j^2\tau-2\pi ij(a\tau+b)}
\\[3mm]
&=& 
\pm \, Q_{2m-j,m}^{(\pm)}(\tau,a\tau+b) 
\,\ + \,\ 
2 \, e^{-2\pi ijb} q^{-\frac{1}{4m}j(j+4ma)}
\end{eqnarray*}}
Substituting this equation into \eqref{m1:eqn:2022-602b3}, we have 
{\allowdisplaybreaks
\begin{eqnarray*}
& & \hspace{-10mm}
Q_{j,m}^{(\pm)}(\tau,a\tau+b)
\,\ = \,\ 
\mp \, (\pm 1)^{2a} \,  e^{4\pi ijb} \, e^{8\pi imab} \, 
Q_{2m-j,m}^{(\pm)}(\tau,a\tau+b) 
\\[1mm]
& &
+ \,\ \underbrace{(\pm 1)^{2a} \, e^{2\pi ijb} \, e^{8\pi imab} 
\sum_{\substack{k \, \in \zzz \\[1mm] 1 \, \leq \, k \, \leq \, 2a}} 
\hspace{-3mm} (\pm 1)^k \, 
e^{4\pi imbk} q^{- \frac{1}{4m} (j+2m(2a-k)) (j-2mk)}}_{\rm (III)}
\\[1mm]
& &
+ \,\ \underbrace{e^{2\pi ijb} 
\sum_{\substack{k \, \in \zzz \\[1mm] 0 \, \leq \, k \, \leq \, 2a-1}} 
\hspace{-3mm} (\pm 1)^k \, 
e^{4\pi imbk} q^{- \frac{1}{4m} (j-2m(2a-k)) (j+2mk)}}_{\rm (IV)}
\end{eqnarray*}}
We compute (IV) by puttong $k=2a-k'$ as follows:
{\allowdisplaybreaks
\begin{eqnarray*}
{\rm (IV)} &=& e^{2\pi ijb} 
\sum_{\substack{k' \in \zzz \\[1mm] 1 \, \leq \, k' \, \leq \, 2a}} 
\hspace{-3mm} (\pm 1)^{2a-k'} \, 
e^{4\pi imb(2a-k')} q^{- \frac{1}{4m} (j+2m(2a-k')) (j-2mk')}
\\[3mm]
&=&
(\pm 1)^{2a} \, e^{2\pi ijb} \, e^{8\pi imab}
\sum_{\substack{k' \in \zzz \\[1mm] 1 \, \leq \, k' \, \leq \, 2a}} 
\hspace{-3mm} (\pm 1)^{k'} \, 
\underbrace{e^{-4\pi imbk'}}_{\substack{|| \\[0mm] 
{\displaystyle e^{4\pi imbk'}
}}} q^{- \frac{1}{4m} (j+2m(2a-k')) (j-2mk')}
\\[-5mm]
&=& {\rm (III)}
\end{eqnarray*}}
Then we have 
{\allowdisplaybreaks
\begin{eqnarray*}
& & \hspace{-10mm}
Q_{j,m}^{(\pm)}(\tau,a\tau+b) \,\ = \,\ \mp \, 
(\pm 1)^{2a} \,  e^{4\pi ijb} \, e^{8\pi imab} \, 
Q_{2m-j,m}^{(\pm)}(\tau,a\tau+b) 
\\[1mm]
& &
+ \,\ 2 \, (\pm 1)^{2a} \, e^{2\pi ijb} \, e^{8\pi imab} 
\sum_{\substack{k \, \in \zzz \\[1mm] 1 \, \leq \, k \, \leq \, 2a}} 
\hspace{-3mm} (\pm 1)^k \, 
e^{4\pi imbk} q^{- \frac{1}{4m} (j+2m(2a-k)) (j-2mk)}
\end{eqnarray*}}
Thus the proof of 2) is completed.
\end{proof}

\medskip

From this lemma, we obtain the following: 

\vspace{0mm}

\begin{prop}
\label{m1:prop:2022-524b}
Let $m \in \frac12 \nnn$, $s, b \in \frac12 \zzz$, 
$a \in \frac12 \zzz_{\geq 0}$ and $j \in s+\zzz$. 
\begin{enumerate}
\item[{\rm 1)}] \quad If \quad 
$e^{4\pi isb} e^{8\pi imab} \, = \, -1$, \quad then
{\allowdisplaybreaks
\begin{eqnarray*}
& & \hspace{-15mm}
R^{(+)}_{j,m}(\tau, a\tau+b)-R^{(+)}_{2m-j,m}(\tau, a\tau+b)
\\[2mm]
&=&
-2 \, e^{2\pi ijb}e^{4\pi isb}
\sum_{\substack{k \in \zzz \\[1mm] 1 \leq k \leq 2a}}
e^{4\pi imbk}
q^{-\frac{1}{4m}(j+2m(2a-k))(j-2mk)}
\end{eqnarray*}}
\item[{\rm 2)}] \quad If 

\vspace{-10.8mm}

\begin{equation} \hspace{-80mm}
(-1)^{2a} e^{4\pi isb} e^{8\pi imab} \, = \, -1
\label{m1:eqn:2022-524d1}
\end{equation}
then

\vspace{-10mm}

{\allowdisplaybreaks
\begin{eqnarray}
& & \hspace{-15mm}
R^{(-)}_{j,m}(\tau, a\tau+b)+R^{(-)}_{2m-j,m}(\tau, a\tau+b)
\nonumber
\\[2mm]
&=&
-2 \, e^{2\pi ijb}e^{4\pi isb}
\sum_{\substack{k \in \zzz \\[1mm] 1 \leq k \leq 2a}}
(-1)^k
e^{4\pi imbk}
q^{-\frac{1}{4m}(j+2m(2a-k))(j-2mk)}
\label{m1:eqn:2022-524d2}
\end{eqnarray}}
\end{enumerate}
\end{prop}

\begin{proof} Under this setting, we have \,\ 
$e^{4\pi isb}=e^{4\pi ijb} \in \{\pm 1\}$.
So the condition 
$$
(\pm 1)^{2a} e^{4\pi isb} e^{8\pi imab} \, = \, -1
$$
means that 
$$
(\pm 1)^{2a} \, e^{8\pi imab} \,\ = \,\ - \, e^{4\pi isb}
$$
Then we have 
\begin{equation}
(\pm 1)^{2a}e^{2\pi ijb}e^{8\pi imab} \, = \, 
- \, e^{2\pi ijb} \, e^{4\pi isb}
\label{m1:eqn:2022-528c}
\end{equation}
and the formula \eqref{m1:eqn:2022-524d2} follows from 
\eqref{m1:eqn:2022-522b3} and Lemma \ref{m1:lemma:2022-523b} 
and \eqref{m1:eqn:2022-528c}.
\end{proof}

\medskip

We note that this Proposition \ref{m1:prop:2022-524b} gives the following:

\begin{note}
\label{m1:note:2022-524a}
In the case when we consider $R^{(-)}_{j,m}$ for
$m \in \frac12 \nnn_{\rm odd}$ and $s=\frac12$,
\begin{enumerate}
\item[{\rm 1)}] the condition \eqref{m1:eqn:2022-524d1} is satisfied 
if $(a,b)$ is $(\frac12, -\frac12)$ or $(\frac12, 0)$ 
or $(0, -\frac12)$.
\item[{\rm 2)}] $R^{(-)}_{j,m}(\tau, a\tau+b)+
R^{(-)}_{2m-j,m}(\tau, a\tau+b)$ 
is as follows for $j \in \frac12+\zzz$:
\begin{enumerate}
\item[{\rm (i)}] if \, $(a,b)=(\frac12, -\frac12)$, \,\ 
$R^{(-)}_{j,m}(\tau, a\tau+b)+R^{(-)}_{2m-j,m}(\tau, a\tau+b)
\, = \, 
-2 \, e^{\pi ij}q^{-\frac{1}{4m}j(j-2m)}$
\item[{\rm (ii)}] if \, $(a,b)=(\frac12, 0)$, \hspace{4mm} 
$R^{(-)}_{j,m}(\tau, a\tau+b)+R^{(-)}_{2m-j,m}(\tau, a\tau+b)
\, = \, 2 \, q^{-\frac{1}{4m}j(j-2m)}$
\item[{\rm (iii)}] if \, $(a,b)=(0, -\frac12)$, \,\ 
$R^{(-)}_{j,m}(\tau, a\tau+b)+R^{(-)}_{2m-j,m}(\tau, a\tau+b)
\, = \, 0$
\end{enumerate}
\end{enumerate}
\end{note}

\section{Modified function $\widetilde{\Phi}^{(-)[m,\frac12]}$}

In this section we compute $\Phi^{(-)[m, \frac12]}_{\rm add}
(\tau, \frac{z}{2}+a\tau+b, \frac{z}{2}-a\tau-b ,0)$
in the case when $m \in \frac12 \nnn_{\rm odd}$ and 
$(a,b)= (\frac12, -\frac12)$, $(\frac12,0)$ or $(0, -\frac12)$.

\begin{prop}
\label{m1:prop:2022-524c}
Let $m \in \frac12 \nnn_{\rm odd}$. Then
\begin{enumerate}
\item[{\rm 1)}] $\Phi^{(-)[m, \frac12]}_{\rm add}(\tau, \, 
\frac{z}{2}+\frac{\tau}{2}-\frac12, \, 
\frac{z}{2}-\frac{\tau}{2}+\frac12, \, 0)$
$$
= \,\ 
\sum_{\substack{k \in \frac12 \zzz_{\rm odd} \\[1mm] 0 < k<m}}
e^{\pi ik}
q^{-\frac{1}{4m}k(k-2m)}
\big[\theta_{k,m}^{(-)}-\theta_{-k,m}^{(-)}\big](\tau, z)
+
e^{\pi im} q^{\frac{m}{4}} \theta_{m,m}^{(-)}(\tau,z)
$$
\item[{\rm 2)}] $\Phi^{(-)[m, \frac12]}_{\rm add}(\tau, \, 
\frac{z}{2}+\frac{\tau}{2}, \, 
\frac{z}{2}-\frac{\tau}{2}, \, 0)
= \, - \hspace{-2mm}
\sum\limits_{\substack{k \in \frac12 \zzz_{\rm odd} \\[1mm] 0 < k<m}}
\hspace{-2mm}
q^{-\frac{1}{4m}k(k-2m)}
\big[\theta_{k,m}^{(-)}-\theta_{-k,m}^{(-)}\big](\tau, z)
- 
q^{\frac{m}{4}} \theta_{m,m}^{(-)}(\tau,z)$
\item[{\rm 3)}] $\Phi^{(-)[m, \frac12]}_{\rm add}(\tau, \, 
\frac{z}{2}-\frac12, \, 
\frac{z}{2}+\frac12, \, 0) \,\ = \,\ 0$
\end{enumerate}
\end{prop}

\begin{proof} When $z_1-z_2=2(a\tau+b)$ and $z_1+z_2=z$, 
the formula \eqref{m1:eqn:2022-512b1} gives
\begin{subequations}
{\allowdisplaybreaks
\begin{eqnarray}
& & \hspace{-15mm}
\Phi_{\rm add}^{(-)[m,\frac12]}(\tau, z_1,z_2,0) 
\, = \, 
-\frac12 \sum_{\substack{k \in \frac12 \zzz_{\rm odd} \\[1mm]
0 < k <2m}}
R^{(-)}_{k,m}(\tau, a\tau+b)
\big[\theta_{k,m}^{(-)}-\theta_{-k,m}^{(-)}\big](\tau,z)
\label{m1:eqn:2022-528d1}
\\[1mm]
& \hspace{-5mm}
\underset{\substack{\\[0.5mm] \uparrow \\[1mm] k \rightarrow 2m-k
}}{=} \hspace{-5mm}
&
-\frac12 \sum_{\substack{k \in \frac12 \zzz_{\rm odd} \\[1mm]
0 < k <2m}}
R^{(-)}_{2m-k,m}(\tau, a\tau+b)
\big[\theta_{k,m}^{(-)}-\theta_{-k,m}^{(-)}\big](\tau,z)
\label{m1:eqn:2022-528d2}
\\[0mm]
&=& \,\ 
\frac12 \, \times \, \big\{
\eqref{m1:eqn:2022-528d1}+\eqref{m1:eqn:2022-528d2}\big\}
\nonumber
\\[2mm]
&=&
-\frac14 \sum_{\substack{k \in \frac12 \zzz_{\rm odd} \\[1mm]
0 < k <2m}}
\big\{R^{(-)}_{k,m}(\tau, a\tau+b)+
R^{(-)}_{2m-k,m}(\tau, a\tau+b)\big\}
\big[\theta_{k,m}^{(-)}-\theta_{-k,m}^{(-)}\big](\tau,z)
\label{m1:eqn:2022-528d3}
\end{eqnarray}}
\end{subequations}
We compute this equation \eqref{m1:eqn:2022-528d3} in each case
by using Note \ref{m1:note:2022-524a} as follows:

\medskip

\noindent
\underline{In the case} \, $(z_1, z_2)= (\frac{z}{2}+\frac{\tau}{2}-\frac12, \, 
\frac{z}{2}-\frac{\tau}{2}+\frac12)$ \,\ i.e, \,\ 
$(a,b) =(\frac12, -\frac12)$,
{\allowdisplaybreaks
\begin{eqnarray*}
& & \hspace{-8mm}
\Phi_{\rm add}^{(\pm)[m,\frac12]}(\tau, z_1,z_2,0) 
=
- \frac14\sum_{\substack{k \in \frac12 \zzz_{\rm odd} \\[1mm] 0 < k <2m}}
(-2) e^{\pi ik}q^{\frac{k(2m-k)}{4m}}
\big[\theta_{k,m}^{(-)}-\theta_{-k,m}^{(-)}\big](\tau,z)
\\[1mm]
&=&
\frac12 \sum_{\substack{k \in \frac12 \zzz_{\rm odd} \\[1mm] 0 < k <2m}}
e^{\pi ik}q^{\frac{k(2m-k)}{4m}}
\big[\theta_{k,m}^{(-)}-\theta_{-k,m}^{(-)}\big](\tau,z)
\\[1mm]
&=&
\frac12 \bigg[
\sum_{\substack{k \in \frac12 \zzz_{\rm odd} \\[1mm] 0 < k <m}}
+\sum_{k=m}
+
\sum_{\substack{k \in \frac12 \zzz_{\rm odd} \\[1mm] m < k <2m}}
\bigg]
e^{\pi ik}q^{\frac{k(2m-k)}{4m}}
\big[\theta_{k,m}^{(-)}-\theta_{-k,m}^{(-)}\big](\tau,z)
\\[1mm]
&=&
\sum_{\substack{k \in \frac12 \zzz_{\rm odd} \\[1mm] 0 < k <m}}
e^{\pi ik}q^{\frac{k(2m-k)}{4m}}
\big[\theta_{k,m}^{(-)}-\theta_{-k,m}^{(-)}\big](\tau,z)
+ 
e^{\pi im} q^{\frac{m}{4}}\theta_{m,m}^{(-)}(\tau, z)
\end{eqnarray*}}

\noindent
\underline{In the case} \, $(z_1, z_2)= (\frac{z}{2}+\frac{\tau}{2}, \, 
\frac{z}{2}-\frac{\tau}{2})$ \,\ i.e, \,\ 
$(a,b) =(\frac12, 0)$,
{\allowdisplaybreaks
\begin{eqnarray*}
& & \hspace{-8mm}
\Phi_{\rm add}^{(\pm)[m,\frac12]}(\tau, z_1,z_2,0) 
=
- \frac14
\sum_{\substack{k \in \frac12 \zzz_{\rm odd} \\[1mm] 0 < k <2m}}
2 \,  q^{\frac{k(2m-k)}{4m}}
\big[\theta_{k,m}^{(-)}-\theta_{-k,m}^{(-)}\big](\tau,z)
\\[1mm]
&=&-
\frac12 
\sum_{\substack{k \in \frac12 \zzz_{\rm odd} \\[1mm] 0 < k <2m}}
q^{\frac{k(2m-k)}{4m}}
\big[\theta_{k,m}^{(-)}-\theta_{-k,m}^{(-)}\big](\tau,z)
\\[1mm]
&=&
\frac12 \bigg[
\sum_{\substack{k \in \frac12 \zzz_{\rm odd} \\[1mm] 0 < k <m}}
+\sum_{k=m}
+
\sum_{\substack{k \in \frac12 \zzz_{\rm odd} \\[1mm] m < k <2m}}
\bigg]
e^{\pi ik}q^{\frac{k(2m-k)}{4m}}
\big[\theta_{k,m}^{(-)}-\theta_{-k,m}^{(-)}\big](\tau,z)
\\[1mm]
&=&-
\sum_{\substack{k \in \frac12 \zzz_{\rm odd} \\[1mm]
0 < k <m}}
q^{\frac{k(2m-k)}{4m}}
\big[\theta_{k,m}^{(-)}-\theta_{-k,m}^{(-)}\big](\tau,z)
- 
q^{\frac{m}{4}}\theta_{m,m}^{(-)}(\tau, z)
\end{eqnarray*}}

\noindent
\underline{In the case} \, $(z_1, z_2)= (\frac{z}{2}-\frac12, \, 
\frac{z}{2}+\frac12)$ \,\ i.e, \,\ 
$(a,b) =(0, -\frac12)$,
{\allowdisplaybreaks
\begin{eqnarray*}
& & \hspace{-8mm}
\Phi_{\rm add}^{(\pm)[m,\frac12]}(\tau, z_1,z_2,0) 
\\[1mm]
&=&
-\frac14 
\sum_{\substack{k \in \frac12 \zzz_{\rm odd} \\[1mm] 0 < k <2m}}
\big\{
\underbrace{
R^{(-)}_{k,m}(\tau, -\tfrac12)+R^{(-)}_{2m-k,m}(\tau, -\tfrac12)}_{0}
\big\}
\big[\theta_{k,m}^{(-)}-\theta_{-k,m}^{(-)}\big](\tau,z)
\, = \, 0
\end{eqnarray*}}
Thus the proof of Proposition \ref{m1:prop:2022-524c} is completed.
\end{proof}

\medskip

Then, by Proposition \ref{m1:prop:2022-525a} and Proposition 
\ref{m1:prop:2022-524c}, we obtain the expression for the modified 
function $\widetilde{\Phi}^{(-)[m,\frac12]}$ as follows:

\begin{thm}
\label{m1:thm:2022-524a}
Let $m \in \frac12 \nnn_{\rm odd}$ and $p \in \zzz_{\geq 0}$. Then
\begin{enumerate}
\item[{\rm 1)}] $q^{-\frac{m}{4}} \, 
\theta_{(2p+1)m, m+\frac12}^{(-)}(\tau,0) \, 
\widetilde{\Phi}^{(-)[m,\frac12]}(\tau, 
\frac{z}{2}+\frac{\tau}{2}-\frac12 ,
\frac{z}{2}-\frac{\tau}{2}+\frac12,0 ) $
{\allowdisplaybreaks
\begin{eqnarray*}
& & \hspace{-7mm}
= \,\ 
- \,\ i \, \eta(\tau)^3 \cdot 
\frac{\theta_{p+\frac12,m+\frac12}^{(-)}(\tau,z)
\, - \, 
\theta_{-(p+\frac12),m+\frac12}^{(-)}(\tau,z)
}{\theta_{0, \frac12}(\tau,z)}
\nonumber
\\[1mm]
& & \hspace{-5mm}
+ 
\bigg[
\sum_{\substack{r \, \in \zzz \\[1mm] 0 \leq r <j}} 
- 
\sum_{\substack{r \, \in \zzz \\[1mm] j \leq r<0}} 
\bigg] 
\sum_{\substack{k \in \frac12 \zzz_{\rm odd} \\[1mm] 0< k \leq m}} 
\hspace{-2mm}
(-1)^{j} e^{\pi ik} 
q^{(m+\frac12)(j+\frac{m(2p+1)}{2m+1})^2 
- m (r+p+\frac{m+k}{2m})^2}
\big[\theta_{k,m}^{(-)}-\theta_{-k,m}^{(-)}\big](\tau, z)
\nonumber
\\[1mm]
& & \hspace{-5mm}
+ 
\bigg[
\sum_{\substack{r \, \in \zzz \\[1mm] 0 \leq r \leq j}} 
- 
\sum_{\substack{r \, \in \zzz \\[1mm] j< r <0}} 
\bigg] 
\sum_{\substack{k \in \frac12 \zzz_{\rm odd} \\[1mm] 0<k<m}}
\hspace{-2mm}
(-1)^{j} e^{\pi ik} 
q^{(m+\frac12)(j+\frac{m(2p+1)}{2m+1})^2 
- m(r+p+\frac{m-k}{2m})^2} 
\big[\theta_{k,m}^{(-)}-\theta_{-k,m}^{(-)}\big](\tau, z)
\nonumber
\\[1mm]
& & \hspace{-5mm}
+ \,\ 
\theta_{(2p+1)m, m+\frac12}^{(-)}(\tau,0) 
\sum_{\substack{r \, \in \zzz \\[1mm] -p< r \leq p}} \,\ 
\sum_{\substack{k \in \frac12 \zzz_{\rm odd} \\[1mm] 0<k <m}} 
e^{\pi ik} \, 
q^{- \, m \, (r+\frac{k-m}{2m})^2} \, 
\big[\theta^{(-)}_{k,m} \, - \, \theta^{(-)}_{-k,m}\big] (\tau, z)
\nonumber
\\[1mm]
& & \hspace{-5mm}
+ \,\ e^{\pi im} \, \theta_{(2p+1)m, m+\frac12}^{(-)}(\tau,0) 
\sum_{\substack{r \, \in \zzz \\[1mm] -p \leq r \leq p}} 
q^{-mr^2} \, \theta^{(-)}_{m,m}(\tau, z)
\end{eqnarray*}}

\item[{\rm 2)}] \,\ $q^{-\frac{m}{4}} \, 
\theta_{(2p+1)m, m+\frac12}(\tau,0) \, 
\widetilde{\Phi}^{(-)[m,\frac12]}(\tau, 
\frac{z}{2}+\frac{\tau}{2},
\frac{z}{2}-\frac{\tau}{2}, 0) $
{\allowdisplaybreaks
\begin{eqnarray*}
& & \hspace{-7mm}
= \,\ 
\eta(\tau)^3 \cdot 
\frac{\theta_{p+\frac12,m+\frac12}(\tau,z)
\, - \, 
\theta_{-(p+\frac12),m+\frac12}(\tau,z)
}{\theta_{0, \frac12}^{(-)}(\tau,z)}
\nonumber
\\[1mm]
& & \hspace{-5mm}
+ (-1)^p \bigg[
\sum_{\substack{r \, \in \zzz \\[1mm] 0 \leq r <j}} 
- 
\sum_{\substack{r \, \in \zzz \\[1mm]  j \leq r <0}} 
\bigg] \hspace{-6mm}
\sum_{\hspace{7mm}
\substack{k \in \frac12 \zzz_{\rm odd} \\[1mm] 0<k \leq m}} 
\hspace{-6mm}
(-1)^r 
q^{(m+\frac12)(j+\frac{m(2p+1)}{2m+1})^2 
- m(r+p+\frac{m+k}{2m})^2}
\big[\theta_{k,m}^{(-)}-\theta_{-k,m}^{(-)}\big](\tau, z)
\nonumber
\\[1mm]
& & \hspace{-5mm}
- (-1)^p 
\bigg[
\sum_{\substack{r \, \in \zzz \\[1mm] 0 \leq r \leq j}} 
- 
\sum_{\substack{r \, \in \zzz \\[1mm] j < r<0 }} 
\bigg] \hspace{-6mm} 
\sum_{\hspace{7mm}
\substack{k \in \frac12 \zzz_{\rm odd} \\[1mm] 0<k <m}}
\hspace{-6mm}
(-1)^r 
q^{(m+\frac12)(j+\frac{m(2p+1)}{2m+1})^2 
- m (r+p+\frac{m-k}{2m})^2} 
\big[\theta_{k,m}^{(-)}-\theta_{-k,m}^{(-)}\big](\tau, z)
\nonumber
\\[1mm]
& & \hspace{-5mm}
- \,\ 
\theta_{(2p+1)m, m+\frac12}(\tau,0) 
\sum_{\substack{r \, \in \zzz \\[1mm] -p < r \leq p}} \,\ 
\sum_{\substack{k \in \frac12 \zzz_{\rm odd} \\[1mm] 0<k < m}} 
(-1)^r \, 
q^{- \, m \, (r+\frac{k-m}{2m})^2} \, 
\big[\theta^{(-)}_{k,m} \, - \, \theta^{(-)}_{-k,m}\big] (\tau, z)
\nonumber
\\[1mm]
& & \hspace{-5mm}
- \,\ \theta_{(2p+1)m, m+\frac12}(\tau,0) 
\sum_{\substack{r \, \in \zzz \\[1mm] -p \leq r \leq p}} 
(-1)^r \, 
q^{-mr^2} \, \theta^{(-)}_{m,m}(\tau, z)
\end{eqnarray*}}

\item[{\rm 3)}] \,\ $\theta_{2pm, m+\frac12}^{(-)}(\tau,0) \, 
\Phi^{(-)[m,\frac12]}(\tau, \frac{z}{2}-\frac12, 
\frac{z}{2}+\frac12, 0)$
{\allowdisplaybreaks
\begin{eqnarray*}
& & \hspace{-7mm}
= \,\ 
- \,\ i \, \eta(\tau)^3 \,\ 
\frac{\theta_{p,m+\frac12}^{(-)}(\tau,z)
\, - \, 
\theta_{-p,m+\frac12}^{(-)}(\tau,z)
}{\theta_{\frac12, \frac12}(\tau,z)}
\nonumber
\\[1mm]
& & \hspace{-5mm}
+ \bigg[
\sum_{\substack{r \, \in \zzz \\[1mm] 0 \leq r <j}} 
- 
\sum_{\substack{r \, \in \zzz \\[1mm] j \leq r <0}} 
\bigg] 
\sum_{\substack{k \in \frac12 \zzz_{\rm odd} \\[1mm] 0<k \leq m}} 
\hspace{-2mm}
(-1)^{j} e^{\pi ik} 
q^{(m+\frac12)(j+\frac{2mp}{2m+1})^2 - m (r+p+\frac{k}{2m})^2}
\big[\theta_{k,m}^{(-)}-\theta_{-k,m}^{(-)}\big](\tau, z)
\nonumber
\\[1mm]
& & \hspace{-5mm}
+ \bigg[
\sum_{\substack{r \, \in \zzz \\[1mm] 0 \leq r \leq j}} 
- 
\sum_{\substack{r \, \in \zzz \\[1mm] j< r <0}} 
\bigg] 
\sum_{\substack{k \in \frac12 \zzz_{\rm odd} \\[1mm] 0<k < m}}
\hspace{-2mm}
(-1)^{j} e^{\pi ik} 
q^{(m+\frac12)(j+\frac{2mp}{2m+1})^2 - m (r+p-\frac{k}{2m})^2}
\big[\theta_{k,m}^{(-)}-\theta_{-k,m}^{(-)}\big](\tau, z)
\nonumber
\\[1mm]
& & \hspace{-5mm}
+ \,\ \theta_{2pm, m+\frac12}^{(-)}(\tau,0) 
\sum_{\substack{r \, \in \zzz \\[1mm] -p < r<p}} \,\ 
\sum_{\substack{k \in \frac12 \zzz_{\rm odd} \\[1mm] 0<k < m}} 
e^{\pi ik} \, 
q^{- \frac{(2mr+k)^2}{4m}} \, 
\big[\theta^{(-)}_{k,m} \, - \, \theta^{(-)}_{-k,m}\big] (\tau, z)
\nonumber
\\[1mm]
& & \hspace{-5mm}
+ \,\ 
2 \, e^{\pi im} \, \theta_{2pm, m+\frac12}^{(-)}(\tau,0) 
\sum_{r=0}^{p-1} 
q^{- \frac{m}{4}(2r+1)^2} \, \theta^{(-)}_{m,m}(\tau, z)
\end{eqnarray*}}
\end{enumerate}
\end{thm}

\section{Functions $\widetilde{\psi}^{(i)[m]}(\tau,z)$}

For $m \in \frac12 \nnn_{\rm odd}$ and $i \in \{1,2,3\}$, we 
consider functions $\widetilde{\psi}^{(i)[m]}(\tau,z)$ defined by 
{\allowdisplaybreaks
\begin{eqnarray*}
\widetilde{\psi}^{(1)[m]}(\tau,z) 
&:=& \hspace{3mm}
\widetilde{\Phi}^{(-)[m,\frac12]}
\Big(\tau, \,\ \frac{z}{2}+\frac{\tau}{2}-\frac12, \,\ 
\frac{z}{2}-\frac{\tau}{2}+\frac12, \,\ 0\Big)
\\[0mm]
&=&
- \, \widetilde{\Phi}^{(-)[m,\frac12]}
\Big(\tau, \,\ \frac{z}{2}+\frac{\tau}{2}+\frac12, \,\ 
\frac{z}{2}-\frac{\tau}{2}-\frac12, \,\ 0\Big)
\\[1mm]
\widetilde{\psi}^{(2)[m]}(\tau,z) 
&:=& \hspace{3mm}
\widetilde{\Phi}^{(-)[m,\frac12]}
\Big(\tau, \,\ \frac{z}{2}+\frac{\tau}{2}, \,\ 
\frac{z}{2}-\frac{\tau}{2}, \,\ 0\Big)
\\[1mm]
\widetilde{\psi}^{(3)[m]}(\tau,z) 
&:=& \hspace{3mm}
\widetilde{\Phi}^{(-)[m,\frac12]}
\Big(\tau, \,\ \frac{z}{2}-\frac12, \,\ \frac{z}{2}+\frac12, \,\ 0\Big)
\\[0mm]
&=&
- \, \widetilde{\Phi}^{(-)[m,\frac12]}
\Big(\tau, \,\ \frac{z}{2}+\frac12, \,\ \frac{z}{2}-\frac12, \,\ 0\Big)
\end{eqnarray*}}
These functions $\widetilde{\psi}^{(i)[m]}(\tau,z)$ satisfy the 
follwing modular transformation properties:

\vspace{0mm}

\begin{lemma} 
\label{m1:lemma:2022-522c}
Let $m \in \frac12 \nnn_{\rm odd}$, then
\begin{enumerate}
\item[{\rm 1)}]
\begin{enumerate}
\item[{\rm (i)}] $\widetilde{\psi}^{(1)[m]} \Big(-\dfrac{1}{\tau}, \dfrac{z}{\tau}\Big) 
\, = \, 
- \, \tau \, e^{\pi im} \, e^{\frac{\pi im}{2\tau}z^2} \, 
e^{-\frac{\pi im}{2\tau}} \, q^{-\frac{m}{4}} \, 
\widetilde{\psi}^{(1)[m]}(\tau, z)$
\item[{\rm (ii)}] $\widetilde{\psi}^{(2)[m]} \Big(-\dfrac{1}{\tau}, \dfrac{z}{\tau}\Big) 
\, = \hspace{4mm}
\tau \, e^{\frac{\pi im}{2\tau}z^2} \, 
e^{-\frac{\pi im}{2\tau}} \, \widetilde{\psi}^{(3)[m]}(\tau, z)$
\item[{\rm (iii)}] $\widetilde{\psi}^{(3)[m]} \Big(-\dfrac{1}{\tau}, \dfrac{z}{\tau}\Big) 
\, = \, 
- \, \tau \, e^{\frac{\pi im}{2\tau}z^2} \, 
q^{-\frac{m}{4}} \, \widetilde{\psi}^{(2)[m]}(\tau, z)$
\end{enumerate}
\item[{\rm 2)}]
\begin{enumerate}
\item[{\rm (i)}] $\widetilde{\psi}^{(1)[m]}(\tau+1,z) 
\, = \hspace{4mm} 
\widetilde{\psi}^{(2)[m]}(\tau,z)$
\item[{\rm (ii)}] $\widetilde{\psi}^{(2)[m]}(\tau+1,z) 
\, = \, 
- \, \widetilde{\psi}^{(1)[m]}(\tau,z)$
\item[{\rm (iii)}] $\widetilde{\psi}^{(3)[m]}(\tau+1,z) 
\, = \hspace{4mm} 
\widetilde{\psi}^{(3)[m]}(\tau,z)$
\end{enumerate}
\end{enumerate}
\end{lemma}

\begin{proof} These are obtained easily from Lemma \ref{m1:lemma:2022-512a}
as follows.

\medskip

\noindent
1) (i) \, $\widetilde{\psi}^{(1)[m]}\Big(-\dfrac{1}{\tau}, \dfrac{z}{\tau}\Big) 
=-
\widetilde{\Phi}^{(-)[m, \frac12]}\Big(-\dfrac{1}{\tau},
\dfrac{z}{2\tau}-\dfrac{1}{2\tau}+\dfrac12, 
\dfrac{z}{2\tau}+\dfrac{1}{2\tau}-\dfrac12, 0\Big)$
{\allowdisplaybreaks
\begin{eqnarray*}
&=& -
\widetilde{\Phi}^{(-)[m, \frac12]}\Big(-\dfrac{1}{\tau}, \, 
\frac{\frac{z}{2}-\frac12+\frac{\tau}{2}}{\tau}, \, 
\frac{\frac{z}{2}+\frac12-\frac{\tau}{2}}{\tau}, \, 0\Big)
\\[1mm]
&=&- \, 
\tau \, e^{\frac{2\pi im}{\tau}
(\frac{z}{2}-\frac12+\frac{\tau}{2})
(\frac{z}{2}+\frac12-\frac{\tau}{2})}
\widetilde{\Phi}^{(-)[m, \frac12]}\Big(\tau, 
\frac{z}{2}-\frac12+\frac{\tau}{2},
\frac{z}{2}+\frac12-\frac{\tau}{2}, 0 \Big)
\\[1mm]
&=&
- \, \tau \, e^{\pi im} e^{\frac{\pi im}{2\tau}z^2}
e^{-\frac{\pi im}{2\tau}}
q^{-\frac{m}{4}}
\widetilde{\psi}^{(1)[m]}(\tau,z) \, ,  
\hspace{10mm} \text{proving (i)}.
\end{eqnarray*}}

(ii) \, $\widetilde{\psi}^{(2)[m]}\Big(-\dfrac{1}{\tau}, \dfrac{z}{\tau}\Big) 
=
\widetilde{\Phi}^{(-)[m, \frac12]}\Big(-\dfrac{1}{\tau},
\dfrac{z}{2\tau}-\dfrac{1}{2\tau}, 
\dfrac{z}{2\tau}+\dfrac{1}{2\tau}, 0\Big)$
{\allowdisplaybreaks
\begin{eqnarray*}
&=& 
\widetilde{\Phi}^{(-)[m, \frac12]}\Big(-\dfrac{1}{\tau}, \, 
\frac{\frac{z}{2}-\frac12}{\tau}, \, 
\frac{\frac{z}{2}+\frac12}{\tau}, \, 0\Big)
\\[1mm]
&=&
\tau \, e^{\frac{2\pi im}{\tau}
(\frac{z}{2}-\frac12)
(\frac{z}{2}+\frac12)}
\widetilde{\Phi}^{(-)[m, \frac12]}\Big(\tau, 
\frac{z}{2}-\frac12,
\frac{z}{2}+\frac12, 0 \Big)
\\[1mm]
&=&
\tau \, e^{\frac{\pi im}{2\tau}z^2}
e^{-\frac{\pi im}{2\tau}}
\widetilde{\psi}^{(3)[m]}(\tau,z) \, ,  
\hspace{10mm} \text{proving (ii)}.
\end{eqnarray*}}

(iii) \, $\widetilde{\psi}^{(3)[m]}\Big(-\dfrac{1}{\tau}, \dfrac{z}{\tau}\Big) 
=-
\widetilde{\Phi}^{(-)[m, \frac12]}\Big(-\dfrac{1}{\tau},
\dfrac{z}{2\tau}+\dfrac12, 
\dfrac{z}{2\tau}-\dfrac12, 0\Big)$
{\allowdisplaybreaks
\begin{eqnarray*}
&=& -
\widetilde{\Phi}^{(-)[m, \frac12]}\Big(-\dfrac{1}{\tau}, \, 
\frac{\frac{z}{2}+\frac{\tau}{2}}{\tau}, \, 
\frac{\frac{z}{2}-\frac{\tau}{2}}{\tau}, \, 0\Big)
\\[1mm]
&=&- \, 
\tau \, e^{\frac{2\pi im}{\tau}
(\frac{z}{2}+\frac{\tau}{2})
(\frac{z}{2}-\frac{\tau}{2})}
\widetilde{\Phi}^{(-)[m, \frac12]}\Big(\tau, 
\frac{z}{2}+\frac{\tau}{2},
\frac{z}{2}-\frac{\tau}{2}, 0 \Big)
\\[1mm]
&=&
- \, \tau \,  e^{\frac{\pi im}{2\tau}z^2}
q^{-\frac{m}{4}}
\widetilde{\psi}^{(2)[m]}(\tau,z) \, ,  
\hspace{10mm} \text{proving (iii)}.
\end{eqnarray*}}

\noindent
2) (i) \, $\widetilde{\psi}^{(1)[m]}(\tau+1,z) \, = \, 
\widetilde{\Phi}^{(-)[m,\frac12]}\Big(\tau+1, \, 
\dfrac{z}{2}+\dfrac{\tau+1}{2}-\dfrac12, \, 
\dfrac{z}{2}-\dfrac{\tau+1}{2}+\dfrac12, \, 0\Big)$
{\allowdisplaybreaks
\begin{eqnarray*}
&=&
\widetilde{\Phi}^{(-)[m,\frac12]}\Big(\tau+1, \, 
\frac{z}{2}+\frac{\tau}{2}, \, 
\frac{z}{2}-\frac{\tau}{2}, \, 0\Big)
\\[1mm]
&=&
\widetilde{\Phi}^{(-)[m,\frac12]}\Big(\tau, \, 
\frac{z}{2}+\frac{\tau}{2}, \, 
\frac{z}{2}-\frac{\tau}{2}, \, 0\Big)
\, = \, \widetilde{\psi}^{(2)[m]}(\tau,z) \, ,  
\hspace{10mm} \text{proving (i)}.
\end{eqnarray*}}

(ii) \, $\widetilde{\psi}^{(2)[m]}(\tau+1,z) \, = \, 
\widetilde{\Phi}^{(-)[m,\frac12]}\Big(\tau+1, \, 
\dfrac{z}{2}+\dfrac{\tau+1}{2}, \, 
\dfrac{z}{2}-\dfrac{\tau+1}{2}, \, 0\Big)$
{\allowdisplaybreaks
\begin{eqnarray*}
&=&
\widetilde{\Phi}^{(-)[m,\frac12]}\Big(\tau+1, \, 
\frac{z}{2}+\frac{\tau}{2}+\frac12, \, 
\frac{z}{2}-\frac{\tau}{2}-\frac12, \, 0\Big)
\\[1mm]
&=&
\widetilde{\Phi}^{(-)[m,\frac12]}\Big(\tau, \, 
\frac{z}{2}+\frac{\tau}{2}+\frac12, \, 
\frac{z}{2}-\frac{\tau}{2}-\frac12, \, 0\Big)
\, = \, - \widetilde{\psi}^{(1)[m]}(\tau,z) \, ,  
\hspace{3mm} \text{proving (ii)}.
\end{eqnarray*}}

(iii) \, $\widetilde{\psi}^{(3)[m]}(\tau+1,z) \, = \, 
\widetilde{\Phi}^{(-)[m,\frac12]}\Big(\tau+1, \, 
\dfrac{z}{2}-\dfrac12, \, 
\dfrac{z}{2}+\dfrac12, \, 0\Big)$
$$
= \, 
\widetilde{\Phi}^{(-)[m,\frac12]}\Big(\tau, \, 
\frac{z}{2}-\frac12, \, 
\frac{z}{2}+\frac12, \, 0\Big)
\, = \, \widetilde{\psi}^{(3)[m]}(\tau,z) \, ,  
\hspace{10mm} \text{proving (iii)}.
$$
Thus the proof of Lemma \ref{m1:lemma:2022-522c} is completed.
\end{proof}

\section{Functions $\Xi^{(i)[m,p]}(\tau,z)$ 
and $\Upsilon^{(i)[m,p]}(\tau,z)$}

For $m \in \frac12 \nnn_{\rm odd}$ and $p \in \zzz$ such that 
$0 \leq p \leq 2m$ and $i \in \{1,2,3\}$, we define functions 
$\Xi^{(i)[m,p]}(\tau,z)$ 
and $\Upsilon^{(i)[m,p]}(\tau,z)$ as follows:
{\allowdisplaybreaks
\begin{eqnarray*}
\Xi^{(1)[m,p]}(\tau,z) &:=&
q^{-\frac{m}{4}} \, 
\theta_{(2p+1)m,m+\frac12}^{(-)}(\tau,0) \cdot 
\widetilde{\psi}^{(1)[m]}(\tau,z)
\\[1mm]
\Xi^{(2)[m,p]}(\tau,z) &:=&
q^{-\frac{m}{4}} \, 
\theta_{(2p+1)m,m+\frac12}(\tau,0) \cdot 
\widetilde{\psi}^{(2)[m]}(\tau,z)
\\[1mm]
\Xi^{(3)[m,p]}(\tau,z) &:=&
\theta_{2pm,m+\frac12}^{(-)}(\tau,0) \cdot 
\widetilde{\psi}^{(3)[m]}(\tau,z)
\end{eqnarray*}}
and
{\allowdisplaybreaks
\begin{eqnarray*}
\Upsilon^{(1)[m,p]}(\tau,z) &:=&
- \, i \,\ \eta(\tau)^3 \, \cdot 
\dfrac{\theta_{p+\frac12,m+\frac12}^{(-)}(\tau,z)
\, - \, 
\theta_{-(p+\frac12),m+\frac12}^{(-)}(\tau,z)
}{\theta_{0, \frac12}(\tau,z)}
\\[1mm]
\Upsilon^{(2)[m,p]}(\tau,z) &:=&
\eta(\tau)^3 \, \cdot 
\dfrac{\theta_{p+\frac12,m+\frac12}(\tau,z)
\, - \, 
\theta_{-(p+\frac12),m+\frac12}(\tau,z)
}{\theta_{0, \frac12}^{(-)}(\tau,z)}
\\[1mm]
\Upsilon^{(3)[m,p]}(\tau,z) &:=&
- \, i \,\ 
\eta(\tau)^3 \, \cdot 
\dfrac{\theta_{p,m+\frac12}^{(-)}(\tau,z)
\, - \, 
\theta_{-p,m+\frac12}^{(-)}(\tau,z)
}{\theta_{\frac12, \frac12}(\tau,z)}
\end{eqnarray*}}

To compute the $S$-transformation of $\Xi^{(i)[m,p]}(\tau,z)$, we use 
the following formulas which are obtained easily from 
Lemma \ref{m1:lemma:2022-526c}.

\begin{note}
\label{m1:note:2022-528a}
Let $m \in \frac12 \nnn_{\rm odd}$ and $p \in \zzz_{\geq 0}$. Then
\begin{enumerate}
\item[{\rm 1)}] $\theta_{(2p+1)m,m+\frac12}\Big(-\dfrac{1}{\tau},0\Big)
\, = \, 
\dfrac{(-i\tau)^{\frac12}}{\sqrt{2m+1}}
\sum\limits_{\ell =0}^{2m}
e^{-\frac{4\pi im^2}{2m+1}(2p+1)\ell}
\theta_{2\ell m,m+\frac12}^{(-)}(\tau,0)$
\item[{\rm 2)}] $\theta_{(2p+1)m,m+\frac12}^{(-)}
\Big(-\dfrac{1}{\tau},0\Big) 
\, = \, 
\dfrac{(-i\tau)^{\frac12}}{\sqrt{2m+1}}
\sum\limits_{\ell =0}^{2m}
e^{-\frac{2\pi im^2}{2m+1}(2p+1)(2\ell+1)}
\theta_{(2\ell+1) m,m+\frac12}^{(-)}(\tau,0)$
\item[{\rm 3)}] $\theta_{2pm,m+\frac12}^{(-)}
\Big(-\dfrac{1}{\tau},0\Big) \hspace{6.5mm}
= \, 
\dfrac{(-i\tau)^{\frac12}}{\sqrt{2m+1}}
\sum\limits_{\ell =0}^{2m}
e^{-\frac{4\pi im^2}{2m+1}p(2\ell+1)}
\theta_{(2\ell+1) m,m+\frac12}(\tau,0)$
\end{enumerate}
\end{note}

\medskip

The $S$-transformation properties of these functions
$\Xi^{(i)[m,p]}$ and $\Upsilon^{(i)[m,p]}$
are given by the following formulas:

\begin{lemma} 
\label{m1:lemma:2022-522d}
Let $m \in \frac12 \nnn_{\rm odd}$ and $p \in \zzz_{\geq 0}$. Then
\begin{enumerate}
\item[{\rm 1)}]
\begin{enumerate}
\item[{\rm (i)}] $\Xi^{(1)[m,p]}\Big(-\dfrac{1}{\tau}, \dfrac{z}{\tau}\Big)
= 
\dfrac{(-i\tau)^{\frac32}}{\sqrt{2m+1}}  
e^{\frac{\pi im}{2\tau}z^2} 
\sum\limits_{\ell=0}^{2m} 
e^{-\frac{\pi i}{2(2m+1)}(2p+1)(2\ell+1)} 
\Xi^{(1)[m, \ell]}(\tau, z)$
\item[{\rm (ii)}] $\Upsilon^{(1)[m,p]}
\Big(-\dfrac{1}{\tau}, \dfrac{z}{\tau}\Big)
= 
\dfrac{(-i\tau)^{\frac32}}{\sqrt{2m+1}}  
e^{\frac{\pi im}{2\tau}z^2}  
\sum\limits_{\ell=0}^{2m}  
e^{-\frac{\pi i}{2(2m+1)}(2p+1)(2\ell+1)} 
\Upsilon^{(1)[m, \ell]}(\tau, z)$
\end{enumerate}
\item[{\rm 2)}]
\begin{enumerate}
\item[{\rm (i)}] $\Xi^{(2)[m,p]}\Big(-\dfrac{1}{\tau}, \dfrac{z}{\tau}\Big)
= \, i \, 
\dfrac{(-i\tau)^{\frac32}}{\sqrt{2m+1}} 
e^{\frac{\pi im}{2\tau}z^2} 
\sum\limits_{\ell=0}^{2m} e^{-\frac{\pi i}{2m+1}(2p+1)\ell} 
\Xi^{(3)[m, \ell]}(\tau, z)$
\item[{\rm (ii)}] $\Upsilon^{(2)[m,p]}
\Big(-\dfrac{1}{\tau}, \dfrac{z}{\tau}\Big) 
= \, i \, 
\dfrac{(-i\tau)^{\frac32}}{\sqrt{2m+1}} 
e^{\frac{\pi im}{2\tau}z^2} 
\sum\limits_{\ell=0}^{2m} 
e^{-\frac{\pi i}{2m+1}(2p+1)\ell} 
\Upsilon^{(3)[m, \ell]}(\tau, z)$
\end{enumerate}
\item[{\rm 3)}]
\begin{enumerate}
\item[{\rm (i)}] $\Xi^{(3)[m,p]}\Big(-\dfrac{1}{\tau}, \dfrac{z}{\tau}\Big)
= \, - \, i \, 
\dfrac{(-i\tau)^{\frac32}}{\sqrt{2m+1}} 
e^{\frac{\pi im}{2\tau}z^2} 
\sum\limits_{\ell=0}^{2m} 
e^{-\frac{\pi i}{2m+1}p(2\ell+1)} 
\Xi^{(2)[m, \ell]}(\tau, z)$
\item[{\rm (ii)}] $\Upsilon^{(3)[m,p]}
\Big(-\dfrac{1}{\tau}, \dfrac{z}{\tau}\Big) 
= \, - \, i \,
\dfrac{(-i\tau)^{\frac32}}{\sqrt{2m+1}}  
e^{\frac{\pi im}{2\tau}z^2} 
\sum\limits_{\ell=0}^{2m} 
e^{-\frac{\pi i}{2m+1}p(2\ell+1)} 
\Upsilon^{(2)[m, \ell]}(\tau, z)$
\end{enumerate}
\end{enumerate}
\end{lemma}

\begin{proof} 1) (i) \, $\Xi^{(1)[m,p]}
\Big(-\dfrac{1}{\tau}, \dfrac{z}{\tau}\Big)
=
e^{-2\pi i \frac{m}{4}(-\frac{1}{\tau})} \, 
\theta_{(2p+1)m,m+\frac12}^{(-)}\Big(-\dfrac{1}{\tau}, 0\Big) \, 
\widetilde{\psi}^{(1)[m, \frac12]}\Big(-\dfrac{1}{\tau}, \dfrac{z}{\tau}\Big)$
{\allowdisplaybreaks
\begin{eqnarray*}
&=&
e^{\frac{\pi im}{2\tau}} \cdot 
\frac{(-i\tau)^{\frac12}}{\sqrt{2m+1}}
\sum_{\ell=0}^{2m}
e^{-\frac{2\pi im^2}{2m+1}(2p+1)(2\ell+1)} \, 
\theta_{(2\ell+1)m,m+\frac12}^{(-)}(\tau,0)
\\[0mm]
& & \hspace{50mm}
\times \, \Big\{
- \tau \, e^{\pi im} \, e^{\frac{\pi im}{2\tau}z^2} \, 
e^{-\frac{\pi im}{2\tau}} \, q^{-\frac{m}{4}} \, 
\widetilde{\psi}^{(1)[m,\frac12]}(\tau, z)
\Big\}
\\[1mm]
&=&
- \, i \, \frac{(-i\tau)^{\frac32}}{\sqrt{2m+1}} \, 
e^{\frac{\pi im}{2\tau}z^2}
\sum_{\ell=0}^{2m}
\underbrace{e^{\pi im} \, e^{-\frac{2\pi im^2}{2m+1}(2p+1)(2\ell+1)}}_{
\substack{|| \\[0mm] {\displaystyle 
i \, e^{-\frac{\pi i}{2(2m+1)}(2p+1)(2\ell+1)}
}}} 
\underbrace{\theta_{(2\ell+1)m,m+\frac12}^{(-)}(\tau,0) \, 
q^{-\frac{m}{4}} \, \widetilde{\psi}^{(1)[m,\frac12]}(\tau, z)}_{
\substack{|| \\[0mm] {\displaystyle \Xi^{(1)[m, \ell]}(\tau, z)
}}}
\\[1mm]
&=&
\frac{(-i\tau)^{\frac32}}{\sqrt{2m+1}} \, 
e^{\frac{\pi im}{2\tau}z^2} \, 
\sum_{\ell=0}^{2m} e^{-\frac{\pi i}{2(2m+1)}(2p+1)(2\ell+1)} \, 
\Xi^{(1)[m, \ell]}(\tau, z)
\end{eqnarray*}}

(ii) $\Upsilon^{(1)[m,p]}\Big(-\dfrac{1}{\tau}, \dfrac{z}{\tau}\Big)
=
- \, i \, 
\eta (-\frac{1}{\tau})^3 \cdot 
\dfrac{
\theta_{p+\frac12,m+\frac12}^{(-)}(-\frac{1}{\tau}, \frac{z}{\tau}) 
\, - \, 
\theta_{-(p+\frac12),m+\frac12}^{(-)}(-\frac{1}{\tau}, \frac{z}{\tau})
}{
\theta_{0, \frac12}(-\frac{1}{\tau}, \, \frac{z}{\tau})}$
{\allowdisplaybreaks
\begin{eqnarray*}
&=&
- \, i \, 
(-i\tau)^{\frac32}\eta(\tau)^3 \cdot
\dfrac{(-i\tau)^{\frac12}}{\sqrt{2(m+\frac12)}} \,\ 
e^{\frac{\pi i(m+\frac12)}{2\tau}z^2}
\\[1mm]
& &
\times \, 
\sum_{\ell=0}^{2m}
e^{-\frac{\pi i}{m+\frac12}(p+\frac12)(\ell+\frac12)} \, 
\frac{
\big[\theta_{\ell+\frac12,m+\frac12}^{(-)}
- \theta_{-(\ell+\frac12),m+\frac12}^{(-)}\big](\tau, z)
}{
(-i\tau)^{\frac12} \, e^{\frac{\pi iz^2}{4\tau}} \, \theta_{0,\frac12}(\tau,z)}
\\[1mm]
&=& 
- \, i \,
\frac{(-i\tau)^{\frac32}}{\sqrt{2m+1}} \, 
e^{\frac{\pi im}{2\tau}z^2} 
\sum_{\ell=0}^{2m} e^{-\frac{\pi i}{2(2m+1)}(2p+1)(2\ell+1)} 
\underbrace{\eta(\tau)^3  
\frac{\big[\theta_{\ell+\frac12,m+\frac12}^{(-)}
- \theta_{-(\ell+\frac12),m+\frac12}^{(-)}\big](\tau, z)
}{\theta_{0,\frac12}(\tau,z)}}_{
\substack{|| \\[0mm] {\displaystyle 
i \, \Upsilon^{(1)[m,\ell]}(\tau, z)
}}}
\end{eqnarray*}}

\noindent
2) (i) \, $\Xi^{(2)[m,p]}\Big(-\dfrac{1}{\tau}, \dfrac{z}{\tau}\Big)
=
e^{-2\pi i \frac{m}{4}(-\frac{1}{\tau})} \, 
\theta_{(2p+1)m,m+\frac12}\Big(-\dfrac{1}{\tau}, 0\Big) \, 
\widetilde{\psi}^{(2)[m, \frac12]}\Big(-\dfrac{1}{\tau}, \dfrac{z}{\tau}\Big)$
{\allowdisplaybreaks
\begin{eqnarray*}
&=&
e^{\frac{\pi im}{2\tau}} \cdot 
\frac{(-i\tau)^{\frac12}}{\sqrt{2m+1}}
\sum_{\ell=0}^{2m}
e^{-\frac{2\pi im^2}{2m+1}(2p+1) \cdot 2\ell} 
\theta_{2\ell m,m+\frac12}^{(-)}(\tau,0)
\Big\{
\tau \, e^{\frac{\pi im}{2\tau}z^2}  
e^{-\frac{\pi im}{2\tau}} \widetilde{\psi}^{(3)[m,\frac12]}(\tau, z)
\Big\}
\\[1mm]
&=&
i \,\ \frac{(-i\tau)^{\frac32}}{\sqrt{2m+1}} \,\ 
e^{\frac{\pi im}{2\tau}z^2}
\sum_{\ell=0}^{2m} \hspace{5mm}
\underbrace{e^{-\frac{2\pi im^2}{2m+1}(2p+1) \cdot 2\ell}}_{
\substack{|| \\[0mm] {\displaystyle e^{-\frac{\pi i}{2m+1}(2p+1)\ell}
}}} \, 
\underbrace{\theta_{2\ell m,m+\frac12}^{(-)}(\tau,0) \, 
\widetilde{\psi}^{(3)[m,\frac12]}(\tau, z)}_{
\substack{|| \\[0mm] {\displaystyle \Xi^{(3)[m, \ell]}(\tau, z)
}}}
\\[0mm]
&=& i \, 
\frac{(-i\tau)^{\frac32}}{\sqrt{2m+1}} \, 
e^{\frac{\pi im}{2\tau}z^2} \, 
\sum_{\ell=0}^{2m} e^{-\frac{\pi i}{2m+1}(2p+1)\ell} \, 
\Xi^{(3)[m, \ell]}(\tau, z)
\end{eqnarray*}}

(ii) $\Upsilon^{(2)[m,p]}\Big(-\dfrac{1}{\tau}, \dfrac{z}{\tau}\Big)
=
\eta (-\tfrac{1}{\tau})^3
\dfrac{ 
\theta_{p+\frac12,m+\frac12}(-\frac{1}{\tau}, \frac{z}{\tau}) 
\, - \, 
\theta_{-(p+\frac12),m+\frac12}(-\frac{1}{\tau}, \frac{z}{\tau})
}{
\theta_{0, \frac12}^{(-)}(-\frac{1}{\tau}, \, \frac{z}{\tau})}$
{\allowdisplaybreaks
\begin{eqnarray*} 
&=&
(-i\tau)^{\frac32}\eta(\tau)^3 \cdot 
\dfrac{(-i\tau)^{\frac12}}{\sqrt{2(m+\frac12)}} \, 
e^{\frac{\pi i(m+\frac12)}{2\tau}z^2}
\sum_{\ell=0}^{2m}
e^{-\frac{\pi i}{m+\frac12}(p+\frac12)\ell} 
\dfrac{ 
\big[\theta_{\ell,m+\frac12}^{(-)}
- \theta_{-\ell,m+\frac12}^{(-)}\big](\tau, z)
}{
(-i\tau)^{\frac12} \, e^{\frac{\pi iz^2}{4\tau}} \, 
\theta_{\frac12,\frac12}(\tau,z)}
\\[1mm]
&=&
\frac{(-i\tau)^{\frac32}}{\sqrt{2m+1}} \, 
e^{\frac{\pi im}{2\tau}z^2} \, 
\sum_{\ell=0}^{2m} e^{-\frac{\pi i}{2(2m+1)}(2p+1) \cdot 2\ell} \, 
\underbrace{\eta(\tau)^3 \, 
\frac{\big[\theta_{\ell,m+\frac12}^{(-)}- \theta_{-\ell,m+\frac12}^{(-)}\big]
(\tau, z)
}{\theta_{\frac12,\frac12}(\tau,z)}}_{\substack{|| \\[0mm] 
{\displaystyle i \, \Upsilon^{(3)[m,\ell]}(\tau,z)
}}}
\end{eqnarray*}}

\noindent
3) (i) \, $\Xi^{(3)[m,p]}
\Big(-\dfrac{1}{\tau}, \dfrac{z}{\tau}\Big)
=
\theta_{2pm,m+\frac12}^{(-)}\Big(-\dfrac{1}{\tau}, 0\Big)
\widetilde{\psi}^{(3)[m, \frac12]}\Big(-\dfrac{1}{\tau}, \dfrac{z}{\tau}\Big)$
{\allowdisplaybreaks
\begin{eqnarray*}
&=& 
\frac{(-i\tau)^{\frac12}}{\sqrt{2m+1}}
\sum_{\ell=0}^{2m}
e^{-\frac{2\pi im^2}{2m+1}\cdot 2p(2\ell+1)} \, 
\theta_{(2\ell+1)m,m+\frac12}(\tau,0)
\big\{
- \tau \, e^{\frac{\pi im}{2\tau}z^2} \, 
q^{-\frac{m}{4}} \, \widetilde{\psi}^{(2)[m,\frac12]}(\tau, z)\big\}
\\[1mm]
&=& 
-i \, 
\frac{(-i\tau)^{\frac32}}{\sqrt{2m+1}} \, 
e^{\frac{\pi im}{2\tau}z^2}
\sum_{\ell=0}^{2m} 
\underbrace{e^{-\frac{2\pi im^2}{2m+1} \cdot 2p(2\ell+1)}}_{
\substack{|| \\[0mm] {\displaystyle 
e^{-\frac{\pi i}{2m+1}p(2\ell+1)}
}}} \, 
\underbrace{\theta_{(2\ell+1)m,m+\frac12}(\tau,0) \, 
q^{-\frac{m}{4}} \, \widetilde{\psi}^{(2)[m,\frac12]}(\tau, z)}_{
\substack{|| \\[0mm] {\displaystyle \Xi^{(2)[m, \ell]}(\tau, z)
}}}
\\[0mm]
&=&- \, i \,\ 
\frac{(-i\tau)^{\frac32}}{\sqrt{2m+1}} \, 
e^{\frac{\pi im}{2\tau}z^2} \, 
\sum_{\ell=0}^{2m} e^{-\frac{\pi i}{2m+1}p(2\ell+1)} \, 
\Xi^{(2)[m, \ell]}(\tau, z)
\end{eqnarray*}}

(ii) \, $\Upsilon^{(3)[m,p]}\Big(-\dfrac{1}{\tau}, \dfrac{z}{\tau}\Big)
- \, i \, 
\eta (-\tfrac{1}{\tau})^3 \cdot 
\dfrac{
\theta_{p+\frac12,m+\frac12}^{(-)}(-\frac{1}{\tau}, \frac{z}{\tau}) 
\, - \, 
\theta_{-(p+\frac12),m+\frac12}^{(-)}(-\frac{1}{\tau}, \frac{z}{\tau})
}{
\theta_{\frac12, \frac12}(-\frac{1}{\tau}, \, \frac{z}{\tau})}$
{\allowdisplaybreaks
\begin{eqnarray*} 
&=&- \, i 
(-i\tau)^{\frac32}\eta(\tau)^3
\dfrac{(-i\tau)^{\frac12}}{\sqrt{2(m+\frac12)}} \, 
e^{\frac{\pi i(m+\frac12)}{2\tau}z^2}
\sum_{\ell=0}^{2m}
e^{-\frac{\pi i}{m+\frac12}p(\ell+\frac12)} 
\frac{\big[\theta_{\ell+\frac12,m+\frac12}
- \theta_{-(\ell+\frac12),m+\frac12}\big](\tau, z)
}{
(-i\tau)^{\frac12} \, e^{\frac{\pi iz^2}{4\tau}} \, 
\theta_{0,\frac12}^{(-)}(\tau,z)}
\\[1mm]
&=&- \, i \, 
\frac{(-i\tau)^{\frac32}}{\sqrt{2m+1}} \, 
e^{\frac{\pi im}{2\tau}z^2} \, 
\sum_{\ell=0}^{2m} e^{-\frac{\pi i}{2m+1}p(2\ell+1)} 
\underbrace{\eta(\tau)^3 \, 
\frac{\big[\theta_{\ell+\frac12,m+\frac12}
- \theta_{-(\ell+\frac12),m+\frac12}\big](\tau, z)
}{\theta_{0,\frac12}^{(-)}(\tau,z)}}_{\substack{|| \\[0mm] 
{\displaystyle \Upsilon^{(2)[m,\ell]}(\tau,z)
}}}
\end{eqnarray*}}
Thus the proof of Lemma \ref{m1:lemma:2022-522d} is completed.
\end{proof}

\medskip

The $T$-transformation properties of $\Xi^{(i)[m,p]}$ 
and $\Upsilon^{(i)[m,p]}$ are given by the following formulas:

\begin{lemma} 
\label{m1:lemma:2022-522a}
Let $m \in \frac12 \nnn_{\rm odd}$ and $p \in \zzz_{\geq 0}$. Then
\begin{enumerate}
\item[{\rm 1)}]
\begin{enumerate}
\item[{\rm (i)}] \quad $\Xi^{(1)[m,p]}(\tau+1, z) 
\,\ = \,\ 
e^{\frac{\pi i(2p+1)^2}{4(2m+1)}-\frac{\pi i}{4}} \, 
\Xi^{(2)[m, p}(\tau,z)$
\item[{\rm (ii)}] \quad $\Xi^{(2)[m,p]}(\tau+1, z) 
\,\ = \,\ 
- \, e^{\frac{\pi i(2p+1)^2}{4(2m+1)}-\frac{\pi i}{4}} \, 
\Xi^{(1)[m, p]}(\tau,z)$
\item[{\rm (iii)}] \quad $\Xi^{(3)[m,p]}(\tau+1, z) 
\,\ = \,\ 
e^{\frac{\pi ip^2}{2m+1}} \, \Xi^{(3)[m,p]}(\tau, z) $
\end{enumerate}
\item[{\rm 2)}]
\begin{enumerate}
\item[{\rm (i)}]  \quad $\Upsilon^{(1)[m,p]}(\tau+1, z)  
\,\ = \hspace{5mm} 
e^{\frac{\pi i(2p+1)^2}{4(2m+1)}-\frac{\pi i}{4}} \, 
\Upsilon^{(2)[m,p]}(\tau, z)$
\item[{\rm (ii)}]  \quad $\Upsilon^{(2)[m,p]}(\tau+1, z)  
\,\ = \,\ - \, 
e^{\frac{\pi i(2p+1)^2}{4(2m+1)}-\frac{\pi i}{4}} \, 
\Upsilon^{(1)[m,p]}(\tau, z)$

\item[{\rm (iii)}]  \quad $\Upsilon^{(3)[m,p]}(\tau+1, z)  
\,\ = \hspace{5mm}
e^{\frac{\pi ip^2}{2m+1}} \, \Upsilon^{(3)[m,p]}(\tau, z)$
\end{enumerate}
\end{enumerate}
\end{lemma}

\begin{proof} 1) (i) \, $\Xi^{(1)[m,p]}(\tau+1, z)
\, = \, e^{-2\pi i \frac{m}{4}(\tau+1)} 
\theta_{(2p+1)m, m+\frac12}^{(-)}(\tau+1,0)  
\widetilde{\psi}^{(1)[m]}(\tau+1,z)$
{\allowdisplaybreaks
\begin{eqnarray*}
&=&
q^{-\frac{m}{4}} e^{-\frac{\pi im}{2}} \cdot 
e^{\frac{\pi i}{2m+1}(2p+1)^2m^2} 
\theta_{(2p+1)m, m+\frac12}(\tau,0) \cdot
\widetilde{\psi}^{(2)[m]}(\tau,z)
\\[1mm]
&=& 
e^{\frac{\pi i(2p+1)^2}{4(2m+1)}-\frac{\pi i}{4}} \,\ 
\underbrace{q^{-\frac{m}{4}} \, 
\theta_{(2p+1)m, m+\frac12}(\tau,0) \, 
\widetilde{\psi}^{(2)[m]}(\tau,z)}_{\substack{|| \\[0mm] 
{\displaystyle \Xi^{(2)[m, p]}(\tau,z)
}}}
\end{eqnarray*}}

(ii) \, $\Xi^{(2)[m,p]}(\tau+1, z)
\, = \, e^{-2\pi i \frac{m}{4}(\tau+1)} \, 
\theta_{(2p+1)m, m+\frac12}(\tau+1,0) \, 
\widetilde{\psi}^{(2)[m]}(\tau+1,z)$
{\allowdisplaybreaks
\begin{eqnarray*}
&=&
q^{-\frac{m}{4}} e^{-\frac{\pi im}{2}} \cdot 
e^{\frac{\pi i}{2m+1}(2p+1)^2m^2} 
\theta_{(2p+1)m, m+\frac12}^{(-)}(\tau,0) \cdot 
(-1) \widetilde{\psi}^{(1)[m]}(\tau,z)
\\[1mm]
&=& - \, 
e^{\frac{\pi i(2p+1)^2}{4(2m+1)}-\frac{\pi i}{4}} \, 
\underbrace{q^{-\frac{m}{4}} \, 
\theta_{(2p+1)m, m+\frac12}^{(-)}(\tau,0) \, 
\widetilde{\psi}^{(1)[m]}(\tau,z)}_{\substack{|| \\[0mm] 
{\displaystyle \Xi^{(1)[m, p]}(\tau,z)
}}}
\end{eqnarray*}}

(iii) \, $\Xi^{(3)[m,p]}(\tau+1, z)  \, = \,\ 
\theta_{2pm, m+\frac12}^{(-)}(\tau+1,0) \, 
\widetilde{\psi}^{(3)[m]}(\tau+1,z)$
$$
= \,\ 
e^{\frac{\pi i}{2m+1}(2pm)^2} 
\underbrace{\theta_{2pm, m+\frac12}^{(-)}(\tau,0) \cdot 
\widetilde{\psi}^{(3)[m]}(\tau,z)}_{\substack{|| \\[0mm] 
{\displaystyle \Xi^{(3)[m, p]}(\tau,z)
}}}
\, = \,\ 
e^{\frac{\pi ip^2}{2m+1}} \, \Xi^{(3)[m, p]}(\tau,z)
$$

\noindent
2) (i) $\Upsilon^{(1)[m,p]}(\tau+1, z) = \, 
- \, i \, \eta(\tau+1)^3 \, 
\dfrac{
[\theta_{p+\frac12, m+\frac12}^{(-)}
-
\theta_{-(p+\frac12), m+\frac12}^{(-)}](\tau+1,z)
}{
\theta_{0, \frac12}(\tau+1,z)} $
{\allowdisplaybreaks
\begin{eqnarray*}
&=&
- \, i \, 
[e^{\frac{\pi i}{12}} \eta(\tau)]^3 \, 
\frac{
e^{\frac{\pi i}{2m+1}(p+\frac12)^2} \, 
[\theta_{p+\frac12, m+\frac12}
-
\theta_{-(p+\frac12), m+\frac12}](\tau,z)
}{
\theta_{0, \frac12}^{(-)}(\tau,z)} 
\\[1mm]
&=& 
e^{\frac{\pi i(2p+1)^2}{4(2m+1)}-\frac{\pi i}{4}} \,\ 
\underbrace{\eta(\tau)^3 \, 
\frac{[\theta_{p+\frac12, m+\frac12}
-
\theta_{-(p+\frac12), m+\frac12}](\tau,z)}{
\theta_{0, \frac12}^{(-)}(\tau,z)}}_{\substack{|| \\[0mm] 
{\displaystyle \Upsilon^{(2)[m,p]}(\tau, z)
}}}
\end{eqnarray*}}

(ii) $\Upsilon^{(2)[m,p]}(\tau+1, z) = \, 
\eta(\tau+1)^3 \, 
\dfrac{[\theta_{p+\frac12, m+\frac12}
-
\theta_{-(p+\frac12), m+\frac12}](\tau+1,z)
}{
\theta_{0, \frac12}^{(-)}(\tau+1,z)} $
{\allowdisplaybreaks
\begin{eqnarray*}
&=&
[e^{\frac{\pi i}{12}} \eta(\tau)]^3
\frac{
e^{\frac{\pi i}{2m+1}(p+\frac12)^2} \, 
[\theta_{p+\frac12, m+\frac12}^{(-)}
-
\theta_{-(p+\frac12), m+\frac12}^{(-)}](\tau,z)
}{
\theta_{0, \frac12}(\tau,z)}
\\[1mm]
&=& 
e^{\frac{\pi i}{4}} \, 
e^{\frac{\pi i}{2m+1}(p+\frac12)^2} \,\ 
\underbrace{\eta(\tau)^3 \,\ 
\frac{[\theta_{p+\frac12, m+\frac12}^{(-)}
-
\theta_{-(p+\frac12), m+\frac12}^{(-)}](\tau,z)}{
\theta_{0, \frac12}(\tau,z)}}_{\substack{|| \\[0mm] 
{\displaystyle i \, \Upsilon^{(1)[m,p]}(\tau, z)
}}}
\\[0mm]
&=&
- \, e^{-\frac{\pi i}{4}} \, 
e^{\frac{\pi i}{2m+1}(p+\frac12)^2} \, \Upsilon^{(1)[m,p]}(\tau, z)
\end{eqnarray*}}

(iii) $\Upsilon^{(3)[m,p]}(\tau+1, z) \, = \, 
- \, i \,\ 
\eta(\tau+1)^3 \,\ \dfrac{ 
[\theta_{p, m+\frac12}^{(-)}
-
\theta_{-p, m+\frac12}^{(-)}](\tau+1,z)
}{
\theta_{\frac12, \frac12}(\tau+1,z)}$
{\allowdisplaybreaks
\begin{eqnarray*}
&=&
- \, i \,\ 
e^{\frac{\pi i}{12}} \eta(\tau)]^3 \,\ 
\dfrac{ 
e^{\frac{\pi i}{2m+1}p^2} \, 
[\theta_{p, m+\frac12}^{(-)}
-
\theta_{-p, m+\frac12}^{(-)}](\tau,z)
}{ 
e^{\frac{\pi i}{4}} \, \theta_{\frac12, \frac12}^{(-)}(\tau,z)}
\\[1mm]
&=& 
- \, i \,\ e^{\frac{\pi i}{2m+1}p^2}
\underbrace{\eta(\tau)^3 \, 
\frac{[\theta_{p, m+\frac12}^{(-)} 
-\theta_{-p, m+\frac12}^{(-)}](\tau,z)}{
\theta_{\frac12, \frac12}(\tau,z)}}_{\substack{|| \\[0mm] 
{\displaystyle i \, \Upsilon^{(3)[m,p]}(\tau, z)
}}}
\,\ = \,\ 
e^{\frac{\pi ip^2}{2m+1}} \, \Upsilon^{(3)[m,p]}(\tau, z)
\end{eqnarray*}}
Thus the proof of Lemma \ref{m1:lemma:2022-522a} is completed.
\end{proof}

\section{Indefinite modular forms $g^{(i)[m,p]}_k(\tau)$}

For $m \in \frac12 \nnn_{\rm odd}$ and $p \in \zzz$ such that 
$0 \leq p \leq 2m$ and $i \in \{1,2,3\}$, we put 
\begin{subequations}
\begin{equation}
G^{(i)[m,p]}(\tau,z) \,\ := \,\ 
\Xi^{(i)[m,p]}(\tau,z) -\Upsilon^{(i)[m,p]}(\tau,z)
\label{m1:eqn:2022-522a1}
\end{equation}
Then, by Theorem \ref{m1:thm:2022-524a}, $G^{(i)[m,p]}(\tau,z)$ 
can be written in the following form:
{\allowdisplaybreaks
\begin{eqnarray}
& & \hspace{-20mm}
G^{(i)[m,p]}(\tau,z) \,\ = \,\  
\sum_{\substack{k \in \frac12 \zzz_{\rm odd} \\[1mm]
0 < k \leq m}}g^{(i)[m,p]}_k(\tau)
\big[\theta_{k,m}^{(-)}-\theta_{-k,m}^{(-)}\big](\tau,z)
\nonumber
\\[1mm]
& & \hspace{-8mm}
= \hspace{-2mm}
\sum_{\substack{k \in \frac12 \zzz_{\rm odd} \\[1mm] 0 < k < m}}
\hspace{-2mm}
g^{(i)[m,p]}_k(\tau)
\big[\theta_{k,m}^{(-)}-\theta_{-k,m}^{(-)}\big](\tau,z)
+
2 \, g^{(i)[m,p]}_m(\tau) \theta_{m,m}^{(-)}(\tau,z)
\label{m1:eqn:2022-522a2}
\end{eqnarray}}
\end{subequations}

\medskip

The modular transformation properties of $G^{(i)[m,p]}(\tau,z)$
are obtained immediately from \eqref{m1:eqn:2022-522a1} and 
Lemma \ref{m1:lemma:2022-522d} and Lemma
\ref{m1:lemma:2022-522a} as follows:

\begin{prop} 
\label{m1:prop:2022-522a}
Let $m \in \frac12 \nnn_{\rm odd}$ and $p \in \zzz_{\geq 0}$ such 
that $0 \leq p \leq 2m$. Then 
\begin{enumerate}
\item[{\rm 1)}]
\begin{subequations}
\begin{enumerate}
\item[{\rm (i)}] $G^{(1)[m,p]}\Big(-\dfrac{1}{\tau}, \dfrac{z}{\tau}\Big)
= 
\dfrac{(-i\tau)^{\frac32}}{\sqrt{2m+1}} 
e^{\frac{\pi im}{2\tau}z^2} 
\sum\limits_{p'=0}^{2m} e^{-\frac{\pi i}{2(2m+1)}(2p+1)(2p'+1)} 
G^{(1)[m,p']}(\tau,z)$

\vspace{-10mm}

\begin{equation}
\label{m1:eqn:2022-530a1}
\end{equation}

\vspace{1mm}

\item[{\rm (ii)}] $G^{(2)[m,p]}\Big(-\dfrac{1}{\tau}, \dfrac{z}{\tau}\Big)
= 
\dfrac{i \, (-i\tau)^{\frac32}}{\sqrt{2m+1}} 
e^{\frac{\pi im}{2\tau}z^2} 
\sum\limits_{p'=0}^{2m} e^{-\frac{\pi i}{2m+1}(2p+1)p'} 
G^{(3)[m,p']}(\tau,z)$

\vspace{-10mm}

\begin{equation}
\label{m1:eqn:2022-530a2}
\end{equation}

\vspace{1mm}

\item[{\rm (iii)}] $G^{(3)[m,p]}\Big(-\dfrac{1}{\tau}, \dfrac{z}{\tau}\Big)
= 
\dfrac{- i \, (-i\tau)^{\frac32}}{\sqrt{2m+1}} 
e^{\frac{\pi im}{2\tau}z^2} 
\sum\limits_{p'=0}^{2m} e^{-\frac{\pi i}{2m+1}p(2p'+1)} 
G^{(2)[m,p']}(\tau,z)$

\vspace{-10mm}

\begin{equation}
\label{m1:eqn:2022-530a3}
\end{equation}
\end{enumerate}
\end{subequations}

\vspace{1mm}

\item[{\rm 2)}]
\begin{enumerate}
\item[{\rm (i)}] $G^{(1)[m,p]}(\tau+1, z) 
\,\ = \,\ 
e^{\frac{\pi i(2p+1)^2}{4(2m+1)}-\frac{\pi i}{4}} \, 
G^{(2)[m, p]}(\tau,z)$

\item[{\rm (ii)}] $G^{(2)[m,p]}(\tau+1, z) 
\,\ = \,\ 
- \, e^{\frac{\pi i(2p+1)^2}{4(2m+1)}-\frac{\pi i}{4}} \, 
G^{(1)[m, p]}(\tau,z)$

\item[{\rm (iii)}] $G^{(3)[m,p]}(\tau+1, z) 
\,\ = \,\ 
e^{\frac{\pi ip^2}{2m+1}} \, G^{(3)[m,p]}(\tau, z) $
\end{enumerate}
\end{enumerate}
\end{prop}

\medskip

Also one can see easily, from \eqref{m1:eqn:2022-522a2} and 
Theorem \ref{m1:thm:2022-524a}, that $g^{(i)[m,p]}_k(\tau)$'s
are written explicitly as follows:
{\allowdisplaybreaks
\begin{eqnarray*}
& & \hspace{-10mm}
\underset{\substack{\\[1mm] (0 \, < \, k \, < \, m)
}}{g^{(1)[m,p]}_k (\tau)} = \,\ 
\bigg[
\sum_{\substack{r \, \in \zzz \\[1mm] 0 \leq r < j}} 
- 
\sum_{\substack{r \, \in \zzz \\[1mm] j \leq r \, < 0}} 
\bigg]
(-1)^{j} e^{\pi ik} \, 
q^{(m+\frac12)(j+\frac{m(2p+1)}{2m+1})^2 
\, - \, m \, (r+p+\frac{m+k}{2m})^2}
\\[1mm]
& & 
+ \,\ \bigg[
\sum_{\substack{r \, \in \zzz \\[1mm] 0 \leq r \leq j}} 
- 
\sum_{\substack{r \, \in \zzz \\[1mm] j<r<0}} 
\bigg] 
(-1)^{j} e^{\pi ik} \,
q^{(m+\frac12)(j+\frac{m(2p+1)}{2m+1})^2 
\, - \, m \, (r+p+\frac{m-k}{2m})^2} 
\\[3mm]
& & 
+ \,\ 
\theta_{(2p+1)m, m+\frac12}^{(-)}(\tau,0) 
\sum_{-p< r \leq p} 
e^{\pi ik} \, 
q^{- \, m \, (r+\frac{k-m}{2m})^2} 
\\[1mm]
& & \hspace{-10mm}
g^{(1)[m,p]}_m (\tau) \, = \, \bigg[
\sum_{\substack{r \, \in \zzz \\[1mm] 0< r \leq j}} 
- 
\sum_{\substack{r \, \in \zzz \\[1mm] j< r \leq 0}} 
\bigg]
(-1)^{j} e^{\pi im} \, 
q^{(m+\frac12)(j+\frac{m(2p+1)}{2m+1})^2- \, m \, (r+p)^2}
\\[3mm]
& & 
+ \,\ \frac12 \, e^{\pi im} \, \theta_{(2p+1)m, m+\frac12}^{(-)}(\tau,0) 
\sum_{\substack{r \, \in \zzz \\[1mm] -p \leq r \leq p}} 
q^{-mr^2} 
\\[2mm]
& & \hspace{-10mm}
\underset{\substack{\\[1mm] (0 \, < \, k \, < \, m)
}}{g^{(2)[m,p]}_k (\tau)} = \, 
(-1)^p \, \bigg[
\sum_{\substack{r \, \in \zzz \\[1mm] 0 \leq r < j}} 
- 
\sum_{\substack{r \, \in \zzz \\[1mm] j \leq r < 0}} 
\bigg] \, 
(-1)^r \, 
q^{(m+\frac12)(j+\frac{m(2p+1)}{2m+1})^2 
\, - \, m \, (r+p+\frac{m+k}{2m})^2}
\nonumber
\\[1mm]
& & \hspace{-3mm}
- \,\ (-1)^p \, \bigg[
\sum_{\substack{r \, \in \zzz \\[1mm] 0 \leq r \leq j}} 
- 
\sum_{\substack{r \, \in \zzz \\[1mm] j < r  <0}} 
\bigg] \, 
(-1)^r \,
q^{(m+\frac12)(j+\frac{m(2p+1)}{2m+1})^2 
\, - \, m \, (r+p+\frac{m-k}{2m})^2} 
\nonumber
\\[1mm]
& & \hspace{-3mm}
- \,\ 
\theta_{(2p+1)m, m+\frac12}(\tau,0) 
\sum_{\substack{r \, \in \zzz \\[1mm] -p < r \leq p}} 
(-1)^r \, 
q^{- \, m \, (r+\frac{k-m}{2m})^2} \, 
\\[1mm]
& & \hspace{-10mm}
g^{(2)[m,p]}_m (\tau) := \,\ 
- \, (-1)^p \, \bigg[
\sum_{\substack{r \, \in \zzz \\[1mm] 0< r \leq j}} 
- 
\sum_{\substack{r \, \in \zzz \\[1mm] j< r \leq 0}} 
\bigg]
(-1)^{r} \, 
q^{(m+\frac12)(j+\frac{m(2p+1)}{2m+1})^2- \, m \, (r+p)^2}
\\[1mm]
& & \hspace{-3mm}
- \,\ \frac12 \,\ \theta_{(2p+1)m, m+\frac12}(\tau,0) 
\sum_{\substack{r \, \in \zzz \\[1mm] -p \leq r \leq p}} 
(-1)^r \, q^{-mr^2}
\\[2mm]
& & \hspace{-10mm}
\underset{\substack{\\[1mm] (0 \, < \, k \, < \, m)
}}{g^{(3)[m,p]}_k (\tau)} := \,\ 
\bigg[
\sum_{\substack{r \, \in \zzz \\[1mm] 0 \leq r <j}} 
- 
\sum_{\substack{r \, \in \zzz \\[1mm] j \leq r <0}} 
\bigg] 
(-1)^{j} e^{\pi ik} \, 
q^{(m+\frac12)(j+\frac{2mp}{2m+1})^2 \, - \, m (r+p+\frac{k}{2m})^2}
\\[1mm]
& & \hspace{-3mm}
+ \,\ \bigg[
\sum_{\substack{r \, \in \zzz \\[1mm] 0 \leq r \leq j}} 
- 
\sum_{\substack{r \, \in \zzz \\[1mm] j< r<0}} 
\bigg] 
(-1)^{j} e^{\pi ik} \, 
q^{(m+\frac12)(j+\frac{2mp}{2m+1})^2 \, - \, m (r+p-\frac{k}{2m})^2}
\\[3mm]
& & \hspace{-3mm}
+ \,\ \theta_{2pm, m+\frac12}^{(-)}(\tau,0) 
\sum_{\substack{r \, \in \zzz \\[1mm] -p< r <p}} 
e^{\pi ik} \, 
q^{- \frac{(2mr+k)^2}{4m}} \, 
\\[1mm]
& & \hspace{-10mm}
g^{(3)[m,p]}_m (\tau) := \,\ 
e^{\pi im} \, \bigg[
\sum_{\substack{r \, \in \zzz \\[1mm] 0< r \leq j}} 
- 
\sum_{\substack{r \, \in \zzz \\[1mm] j< r \leq 0}} 
\bigg] 
(-1)^{j} \, 
q^{(m+\frac12)(j+\frac{2mp}{2m+1})^2 \, - \, m(r-\frac12+p)^2}
\\[1mm] 
& & + \,\ 
e^{\pi im} \, \theta_{2pm, m+\frac12}^{(-)}(\tau,0) 
\sum_{r=1}^{p} q^{- \frac{m}{4}(2r-1)^2}
\end{eqnarray*}}

Modular transformation properties of these functions 
can be obtained from Proposition \ref{m1:prop:2022-522a}.
First we compute $S$-transformation as follows:

\begin{lemma} 
\label{m1:lemma:2022-522b}
Let $m, \ell \in \frac12 \nnn_{\rm odd}$ 
and $p \in \zzz$ such that $\frac12 \leq \ell < m$ and
$0 \leq p \leq 2m$. Then
\begin{enumerate}
\item[{\rm 1)}] 
\begin{subequations}
\begin{enumerate}
\item[{\rm (i)}] $
\sum\limits_{\substack{k \in \frac12 \zzz_{\rm odd} \\[1mm] 0< k \leq m}} 
\sin \dfrac{\pi k\ell}{m} \cdot 
g^{(1)[m,p]}_k\Big(-\dfrac{1}{\tau}\Big)
= 
\dfrac{\tau \, \sqrt{2m}}{2\sqrt{2m+1}} \, 
\sum\limits_{p'=0}^{2m} e^{-\frac{\pi i}{2(2m+1)}(2p+1)(2p'+1)} \, 
g^{(1)[m,p']}_\ell(\tau)$

\vspace{-8mm}

\begin{equation}
\label{m1:eqn:2022-530c1}
\end{equation}

\item[{\rm (ii)}] $
\sum\limits_{\substack{k \in \frac12 \zzz_{\rm odd} \\[1mm] 0< k \leq m}} 
\sin \dfrac{\pi k\ell}{m} \cdot 
g^{(2)[m,p]}_k\Big(-\dfrac{1}{\tau}\Big)
= 
\dfrac{i \, \tau \, \sqrt{2m}}{2\sqrt{2m+1}} \, 
\sum\limits_{p'=0}^{2m} e^{-\frac{\pi i}{2m+1}(2p+1)p'} \, 
g^{(3)[m,p']}_\ell(\tau)$

\vspace{-14mm}

\begin{equation}
\label{m1:eqn:2022-530c2}
\end{equation}

\vspace{4mm}

\item[{\rm (iii)}] $
\sum\limits_{\substack{k \in \frac12 \zzz_{\rm odd} \\[1mm] 0< k \leq m}} 
\sin \dfrac{\pi k\ell}{m} \cdot 
g^{(3)[m,p]}_k\Big(-\dfrac{1}{\tau}\Big)
= 
\dfrac{- i \, \tau \, \sqrt{2m}}{2\sqrt{2m+1}} \, 
\sum\limits_{p'=0}^{2m} e^{-\frac{\pi i}{2m+1}p(2p'+1)} \, 
g^{(2)[m,p']}_\ell(\tau)$

\vspace{-14mm}

\begin{equation}
\label{m1:eqn:2022-530c3}
\end{equation}
\end{enumerate}
\end{subequations}

\vspace{4mm}

\item[{\rm 2)}]
\begin{subequations}
\begin{enumerate}
\item[{\rm (i)}] \,\ $\dfrac12 \, 
\sum\limits_{\substack{k \in \frac12 \zzz_{\rm odd} \\[1mm] 0< k \leq m}} 
\sin \pi k \cdot g^{(1)[m,p]}_k\Big(-\dfrac{1}{\tau}\Big)
= 
\dfrac{\tau \, \sqrt{2m}}{2 \sqrt{2m+1}} \, 
\sum\limits_{p'=0}^{2m} e^{-\frac{\pi i}{2(2m+1)}(2p+1)(2p'+1)} 
g^{(1)[m,p']}_m(\tau)$

\vspace{-5mm}

\begin{equation} 
\label{m1:eqn:2022-530d1}
\end{equation}

\item[{\rm (ii)}] \,\ $\dfrac12 \, 
\sum\limits_{\substack{k \in \frac12 \zzz_{\rm odd} \\[1mm] 0< k \leq m}} 
\sin \pi k \cdot g^{(2)[m,p]}_k\Big(-\dfrac{1}{\tau}\Big)
= 
\dfrac{i \, \tau \, \sqrt{2m}}{2 \sqrt{2m+1}} \, 
\sum\limits_{p'=0}^{2m} e^{-\frac{\pi i}{2m+1}(2p+1)p'} 
g^{(3)[m,p']}_m(\tau)$

\vspace{-14mm}

\begin{equation} 
\label{m1:eqn:2022-530d2}
\end{equation}

\vspace{4mm}

\item[{\rm (iii)}] \,\ $\dfrac12 \, 
\sum\limits_{\substack{k \in \frac12 \zzz_{\rm odd} \\[1mm] 0< k \leq m}} 
\sin \pi k \cdot g^{(3)[m,p]}_k\Big(-\dfrac{1}{\tau}\Big)
= 
\dfrac{-i \, \tau \, \sqrt{2m}}{2 \sqrt{2m+1}} \, 
\sum\limits_{p'=0}^{2m} e^{-\frac{\pi i}{2m+1}p(2p'+1)} 
g^{(2)[m,p']}_m(\tau)$

\vspace{-14mm}

\begin{equation}  
\label{m1:eqn:2022-530d3}
\end{equation}
\end{enumerate}
\end{subequations}
\end{enumerate}
\end{lemma}

\vspace{5mm}

\begin{proof} 1) Substituting \eqref{m1:eqn:2022-522a2} into 
\eqref{m1:eqn:2022-530a1}, one has 
\begin{subequations}
{\allowdisplaybreaks
\begin{eqnarray}
& & \hspace{-12mm}
\text{LHS of \eqref{m1:eqn:2022-530a1}} \, = \,
G^{(1)[m,p]}\Big(-\frac{1}{\tau}, \frac{z}{\tau}\Big)
\, = 
\sum_{\substack{k \in \frac12 \zzz_{\rm odd} \\[1mm] 0< k \leq m}}
\hspace{-2mm}
g^{(1)[m,p]}_k\Big(-\frac{1}{\tau}\Big) \, 
\big[\theta_{k,m}^{(-)}- \theta_{-k,m}^{(-)}\big]
\Big(-\frac{1}{\tau}, \frac{z}{\tau}\Big)
\nonumber
\\[1mm]
&=&
\hspace{-2mm}
\sum_{\substack{k \in \frac12 \zzz_{\rm odd} \\[1mm] 0<k \leq m}} 
\hspace{-2mm}
g^{(1)[m,p]}_k\Big(-\frac{1}{\tau}\Big) \, 
\frac{- \, 2 \, i}{\sqrt{2m}} \, (-i\tau)^{\frac12} \, 
e^{\frac{\pi im}{2\tau}z^2} 
\nonumber
\\[1mm]
& &
\times \, \Bigg\{
\sum_{\substack{\\[0.5mm] \ell \in \frac12 \zzz_{\rm odd} \\[1mm] 0 < \ell < m}} 
\hspace{-3mm}
\sin \frac{\pi k\ell}{m} \cdot 
\big[\theta_{\ell,m}^{(-)}-\theta_{-\ell,m}^{(-)}\big](\tau, z)
+
\sin \pi k \cdot \theta_{m,m}^{(-)}(\tau, z)\Bigg\}
\label{m1:eqn:2022-530b1}
\\[1mm]
& & \hspace{-12mm}
\text{RHS of \eqref{m1:eqn:2022-530a1}} =
\frac{(-i\tau)^{\frac32}}{\sqrt{2m+1}} \, 
e^{\frac{\pi im}{2\tau}z^2} \, 
\sum_{p'=0}^{2m} e^{-\frac{\pi i}{2(2m+1)}(2p+1)(2p'+1)} \, 
G^{(1)[m,p']}(\tau,z)
\nonumber
\\[1mm]
&=&
\frac{(-i\tau)^{\frac32}}{\sqrt{2m+1}} 
e^{\frac{\pi im}{2\tau}z^2} 
\sum_{p'=0}^{2m} e^{-\frac{\pi i}{2(2m+1)}(2p+1)(2p'+1)} 
\hspace{-3mm}
\sum_{\substack{\ell \in \frac12 \zzz_{\rm odd} \\[1mm] 0<\ell \leq m}} 
\hspace{-3mm}
g^{(1)[m,p']}_\ell(\tau) 
\big[\theta_{\ell,m}^{(-)}- \theta_{-\ell,m}^{(-)}\big](\tau,z)
\label{m1:eqn:2022-530b2}
\end{eqnarray}}
\end{subequations}
Then, comparing the terms of $\theta_{\ell,m}^{(-)}(\tau,z)$ in 
\eqref{m1:eqn:2022-530b1} and \eqref{m1:eqn:2022-530b2}, 
we obtain the formulas \eqref{m1:eqn:2022-530c1} and 
\eqref{m1:eqn:2022-530d1}. 
The others are obtained by similar calculation.
\end{proof}

\medskip

In order to deduce explicit formulas for $g^{(i)[m,p]}_j(-\frac{1}{\tau})$
from the above Lemma \ref{m1:lemma:2022-522b}, 
we use the following simple formula:

\begin{note}
\label{m1:note:2022-524b}
Let $m \in \frac12 \nnn_{\rm odd}$ and 
$j, k \in \frac12 \zzz_{\rm odd}$ such that $0 < j, \, k  \leq m$. Then
\begin{subequations}
\begin{equation}
\sum_{\substack{\ell \in \frac12 \zzz_{\rm odd} \\[1mm] 0< \ell<m}}
\sin \frac{\pi j\ell}{m}\sin \frac{\pi k\ell}{m}
+ \frac12 \sin \pi j \cdot \sin \pi k 
\,\ = \,\ \frac{m}{2} \, c_j \, \delta_{j, k}
\label{m1:eqn:2022-524e}
\end{equation}
where
\begin{equation}
c_j \,\ := \,\ \left\{
\begin{array}{cccl}
1 & & \, \text{if} \,  & j <m \\[1mm]
2 & & \text{if} & j=m
\end{array}\right.
\label{m1:eqn:2022-713a}
\end{equation}
\end{subequations}
\end{note}

\medskip

Summing up the above, we arrive at the following theorem:

\begin{thm} 
\label{m1:thm:2022-524b}
Let $m, j \in \frac12 \nnn_{\rm odd}$ and $p \in \zzz$ such that
$0 \leq p \leq 2m$ and $0<j \leq m$. Then $S$- and $T$-transformation 
of $g^{(i)[m,p]}_j(\tau)$ are given by the following formulas:
\begin{enumerate}
\item[{\rm 1)}] $S$-transformation:
\begin{enumerate}
\item[{\rm (i)}] $g^{(1)[m,p]}_j\Big(-\dfrac{1}{\tau}\Big)
= 
\dfrac{\tau}{c_j\sqrt{m(m+\frac12)}} \, 
\sum\limits_{\substack{\ell \in \frac12 \zzz_{\rm odd} \\[1mm] 0< \ell \leq m}} 
\sum\limits_{p'=0}^{2m} e^{-\frac{\pi i}{2(2m+1)}(2p+1)(2p'+1)} \, 
\sin \dfrac{\pi j\ell}{m}
\, g^{(1)[m,p']}_{\ell}(\tau)$

\item[{\rm (ii)}] $g^{(2)[m,p]}_j\Big(-\dfrac{1}{\tau}\Big) 
= 
\dfrac{i \, \tau}{c_j\sqrt{m(m+\frac12)}} \, 
\sum\limits_{\substack{\ell \in \frac12 \zzz_{\rm odd} \\[1mm] 0<\ell \leq m}} 
\sum\limits_{p'=0}^{2m} e^{-\frac{\pi i}{2m+1}(2p+1)p'} \, 
\sin \dfrac{\pi j\ell}{m}
\, g^{(3)[m,p']}_{\ell}(\tau)$

\item[{\rm (iii)}] $g^{(3)[m,p]}_j\Big(-\dfrac{1}{\tau}\Big) 
=  
\dfrac{- i \, \tau}{c_j\sqrt{m(m+\frac12)}} \, 
\sum\limits_{\substack{\ell \in \frac12 \zzz_{\rm odd} \\[1mm] 0<\ell \leq m}} 
\sum\limits_{p'=0}^{2m} e^{-\frac{\pi i}{2m+1}p(2p'+1)} \, 
\sin \dfrac{\pi j\ell}{m}
\, g^{(2)[m,p']}_{\ell}(\tau)$
\end{enumerate}
where $c_j$ is defined by \eqref{m1:eqn:2022-713a}.

\item[{\rm 2)}] $T$-transformation:
\begin{enumerate}
\item[{\rm (i)}] \quad $g^{(1)[m,p]}_j(\tau+1) 
\,\ = \,\ 
e^{\frac{\pi i(2p+1)^2}{4(2m+1)}-\frac{\pi ij^2}{2m} 
-\frac{\pi i}{4}} \, 
g^{(2)[m, p]}_j(\tau)$

\item[{\rm (ii)}] \quad $g^{(2)[m,p]}_j(\tau+1) 
\,\ = \,\ 
- \, e^{\frac{\pi i(2p+1)^2}{4(2m+1)}-\frac{\pi ij^2}{2m}
-\frac{\pi i}{4}} \, 
g^{(1)[m, p]}_j(\tau)$

\item[{\rm (iii)}] \quad $g^{(3)[m,p]}_j(\tau+1) 
\,\ = \,\ 
e^{\frac{\pi ip^2}{2m+1}-\frac{\pi ij^2}{2m}} \, 
g^{(3)[m,p]}_j(\tau) $
\end{enumerate}
\end{enumerate}
\end{thm}

\begin{proof} 1) (i) For $j \, \in \, \frac12 \, \nnn_{\rm odd}$ such that 
$\frac12 \leq j \leq m$, we compute 
$$
\sum_{\substack{\ell \in \frac12 \zzz_{\rm odd} \\[1mm] 0<\ell < m}} 
\sin \frac{\pi j\ell}{m} \times {\rm \eqref{m1:eqn:2022-530c1}}
\,\ + \,\ 
\sin \pi j \times {\rm \eqref{m1:eqn:2022-530d1}} \, .
$$
Then 
\begin{subequations}
{\allowdisplaybreaks
\begin{eqnarray}
& & \hspace{-10mm}
\sum_{\substack{\ell \, \in \, \frac12 \zzz_{\rm odd} \\[1mm] 0<\ell < m}} 
\sin \frac{\pi j\ell}{m} \times \text{LHS of \eqref{m1:eqn:2022-530c1}}
\,\ + \,\ 
\sin \pi j \times \text{LHS of \eqref{m1:eqn:2022-530d1}}
\nonumber
\\[0mm]
&=&
\sum_{\substack{\ell \in \frac12 \zzz_{\rm odd} \\[1mm] 0<\ell < m}} 
\sin \frac{\pi j\ell}{m}
\sum_{\substack{k \in \frac12 \zzz_{\rm odd} \\[1mm] 0<k \leq m}} 
\sin \frac{\pi k\ell}{m} \cdot 
g^{(1)[m,p]}_k\Big(-\frac{1}{\tau}\Big) 
\, + \, 
\frac12 \, \sin \pi j \hspace{-2mm}
\sum_{\substack{k \in \frac12 \zzz_{\rm odd} \\[1mm] 0<k \leq m}} 
\hspace{-2mm}
\sin \pi k \cdot g^{(1)[m,p]}_k\Big(-\frac{1}{\tau}\Big) 
\nonumber
\\[1mm]
&=&
\sum_{\substack{k \in \frac12 \zzz_{\rm odd} \\[1mm] 0<k \leq m}} 
\Bigg\{
\underbrace{\sum_{\substack{\ell \in \frac12 \zzz_{\rm odd} 
\\[1mm] 0<\ell < m}} 
\sin \frac{\pi j\ell}{m} \cdot \sin \frac{\pi k\ell}{m}
+
\frac12 \, \sin \pi j \cdot\sin \pi k}_{\substack{
|| \\[0mm] {\displaystyle 
\frac{m}{2} \, c_j \, \delta_{j,k}
}}} 
\Bigg\} \, g^{(1)[m,p]}_k\Big(-\frac{1}{\tau}\Big) 
\nonumber
\\[-5mm]
&=&
\frac{m}{2} \, c_j \, g^{(1)[m,p]}_j\Big(-\frac{1}{\tau}\Big) 
\label{m1:eqn:2022-601a1}
\end{eqnarray}}
and
\begin{eqnarray}
& & \hspace{-10mm}
\sum_{\substack{\ell \, \in \, \frac12 \zzz_{\rm odd} \\[1mm] 0 < \ell < m}} 
\sin \frac{\pi j\ell}{m} \times \text{RHS of \eqref{m1:eqn:2022-530c1}}
\,\ + \,\ 
\sin \pi j \times \text{RHS of \eqref{m1:eqn:2022-530d1}}
\nonumber
\\[0mm]
& & \hspace{-7mm}
= \frac{\tau \, \sqrt{2m}}{2 \sqrt{2m+1}} \, 
\sum_{\substack{\ell \in \frac12 \zzz_{\rm odd} \\[1mm] 0< \ell \leq m}} 
\sin \frac{\pi j\ell}{m}
\sum_{p'=0}^{2m} e^{-\frac{\pi i}{2(2m+1)}(2p+1)(2p'+1)} 
g^{(1)[m,p']}_{\ell}(\tau)
\label{m1:eqn:2022-601a2}
\end{eqnarray}
\end{subequations}
Then, by $\eqref{m1:eqn:2022-601a1} = \eqref{m1:eqn:2022-601a2}$,
we obtain the formula 1) (i). The formulas 1) (ii) and (iii) 
can be proved in similar way. And 
2) follows immediately from Proposition \ref{m1:prop:2022-522a} 
and Lemma \ref{m1:lemma:2022-526d}.
\end{proof}

\begin{rem}
\label{rem:m1:2022-604a}
We note that $g^{(3)[m,0]}_m(\tau)=0$ and this is, of course, in 
consistency with $S$- and $T$-transformation formulas. 
\end{rem}

\section{An example $\sim$ the case $m=\frac12$}
\label{sec:ex}

In this section, we compute the functions $g^{(i)[m,p]}_j(\tau)$ 
and their modular transformation in the case $m=\frac12$ 
to show the following:

\begin{prop} \,\ 
\label{n3:prop:2022-716a}
\begin{enumerate}
\item[{\rm 1)}] \,\ $\bigg[
\sum\limits_{\substack{j, \, r \, \in \zzz \\[1mm]
0 \, < \, r \, \leq \, j}}
-
\sum\limits_{\substack{ j, \, r \, \in \zzz \\[1mm]
j \, < \, r \, \leq 0}}
\bigg]
(-1)^j \, q^{(j+\frac14)^2-\frac12 r^2}
\,\ = \,\ 
\dfrac12 \, \bigg\{ \dfrac{\eta(\tau)^4}{\eta(\frac{\tau}{2})\eta(2\tau)}
\, - \, 
\dfrac{\eta(\frac{\tau}{2})\eta(2\tau)}{\eta(\tau)}\bigg\}$
\item[{\rm 2)}] \,\ $\bigg[
\sum\limits_{\substack{j, \, r \, \in \zzz \\[1mm]
0 \, < \, r \, \leq \, j}}
-
\sum\limits_{\substack{ j, \, r \, \in \zzz \\[1mm]
j \, < \, r \, \leq 0}}
\bigg]
(-1)^r \, q^{(j+\frac14)^2-\frac12 r^2}
\,\ = \,\ 
\dfrac12 \, \bigg\{\eta(\tau) \eta\Big(\dfrac{\tau}{2}\Big) 
\, - \, 
\dfrac{\eta(\tau)^2}{\eta(\frac{\tau}{2})}\bigg\}$
\item[{\rm 3)}] \,\ $\bigg[
\sum\limits_{\substack{j, \, r \, \in \zzz \\[1mm]
0 \, < \, r \, \leq \, j}}
-
\sum\limits_{\substack{ j, \, r \, \in \zzz \\[1mm]
j \, < \, r \, \leq 0}}
\bigg]
(-1)^j \, q^{(j+\frac12)^2-\frac12(r+\frac12)^2}
\,\ = \,\ 
\eta(\tau) \eta(2\tau)$
\end{enumerate}
\end{prop}

\begin{proof} In the case $m=\frac12$, the functions
$g^{(i)[\frac12,p]}_{\frac12}(\tau)$ \,\ ($i\in \{1,2,3\}$ and 
$p \in \{0,1\}$) are as follows:
{\allowdisplaybreaks
\begin{eqnarray*}
g^{(1)[\frac12,1]}_{\frac12}(\tau)
&=& 
g^{(1)[\frac12,0]}_{\frac12}(\tau)
\,\ = \,\ 
i \, \Bigg\{
\bigg[
\sum\limits_{\substack{j, \, r \, \in \zzz \\[1mm]
0 \, < \, r \, \leq \, j}}
-
\sum\limits_{\substack{ j, \, r \, \in \zzz \\[1mm]
j \, < \, r \, \leq 0}}
\bigg]
(-1)^j \, q^{(j+\frac14)^2-\frac12 r^2}
+ \, 
\frac12 \, \theta_{\frac12, \, 1}^{(-)}(\tau,0)\Bigg\}
\\[2mm]
g^{(2)[\frac12,1]}_{\frac12}(\tau)
&=& 
- \, g^{(2)[\frac12,0]}_{\frac12}(\tau)
\,\ = \,\ 
\bigg[
\sum\limits_{\substack{j, \, r \, \in \zzz \\[1mm]
0 \, < \, r \, \leq \, j}}
-
\sum\limits_{\substack{ j, \, r \, \in \zzz \\[1mm]
j \, < \, r \, \leq 0}}
\bigg]
(-1)^r \, q^{(j+\frac14)^2-\frac12 r^2}
+ \, 
\dfrac12 \, \theta_{\frac12, \, 1}(\tau,0)
\\[-2mm]
g^{(3)[\frac12,0]}_{\frac12}(\tau)
&=& 0
\\[2mm]
g^{(3)[\frac12,1]}_{\frac12}(\tau)
&=&
i \, \bigg[
\sum\limits_{\substack{j, \, r \, \in \zzz \\[1mm]
0 \, < \, r \, \leq \, j}}
-
\sum\limits_{\substack{ j, \, r \, \in \zzz \\[1mm]
j \, < \, r \, \leq 0}}
\bigg]
(-1)^j \, q^{(j+\frac12)^2-\frac12(r+\frac12)^2}
\end{eqnarray*}}
Modular transformation of these functions are obtained from 
Theorem \ref{m1:thm:2022-524b} as follows:
$$ \left\{
\begin{array}{rcr}
g^{(1)[\frac12, 1]}_\frac12 \Big(-\dfrac{1}{\tau}\Big)
&=&
- \, i \tau \, g^{(1)[\frac12,1]}_{\frac12}(\tau)
\\[3.5mm]
g^{(2)[\frac12,1]}_\frac12 \Big(-\dfrac{1}{\tau}\Big) 
&=& 
\dfrac{-\tau}{\sqrt{2}} \,\ g^{(3)[\frac12 ,1]}_{\frac12}(\tau)
\\[3.5mm]
g^{(3)[\frac12,1]}_\frac12 \Big(-\dfrac{1}{\tau}\Big) 
&=&
\sqrt{2} \, \tau \, g^{(2)[\frac12,1]}_{\frac12}(\tau)
\end{array} \right. \hspace{5mm}
\left\{
\begin{array}{rcr}
g^{(1)[\frac12,1]}_\frac12(\tau+1) 
&=&
i \, e^{\frac{\pi i}{8}} \, g^{(2)[\frac12, 1]}_\frac12(\tau)
\\[3.5mm]
g^{(2)[\frac12,1]}_\frac12(\tau+1) 
&=&
- \, i \, e^{\frac{\pi i}{8}} \, g^{(1)[\frac12, 1]}_\frac12 (\tau)
\\[3.5mm]
g^{(3)[\frac12,1]}_\frac12 (\tau+1) 
&=&
e^{\frac{\pi i}{4}} \, g^{(3)[\frac12,1]}_\frac12 (\tau)
\end{array} \right.
$$
Putting $
\widetilde{g}_1(\tau) := - i \, g^{(1)[\frac12, 1]}_{\frac12}(\tau), \,\ 
\widetilde{g}_2(\tau) := g^{(2)[\frac12, 1]}_{\frac12}(\tau)$ and 
$\widetilde{g}_3(\tau) := 
\dfrac{- i}{\sqrt{2}} \, g^{(3)[\frac12, 1]}_{\frac12}(\tau) $,
these formulas are rewritten in terms of $\widetilde{g}_i(\tau)$ 
as follows:
{\allowdisplaybreaks
\begin{eqnarray}
\widetilde{g}_1(\tau) 
&=&
\bigg[
\sum\limits_{\substack{j, \, r \, \in \zzz \\[1mm]
0 \, < \, r \, \leq \, j}}
-
\sum\limits_{\substack{ j, \, r \, \in \zzz \\[1mm]
j \, < \, r \, \leq 0}}
\bigg]
(-1)^j \, q^{(j+\frac14)^2-\frac12 r^2}
\, + \, 
\dfrac12 \, \theta_{\frac12, \, 1}^{(-)}(\tau,0)
\,\ = \,\ 
\dfrac12 \, q^{\frac{1}{16}} + \, \cdots 
\nonumber
\\[2mm]
\widetilde{g}_2(\tau) 
&=&
\bigg[
\sum\limits_{\substack{j, \, r \, \in \zzz \\[1mm]
0 \, < \, r \, \leq \, j}}
-
\sum\limits_{\substack{ j, \, r \, \in \zzz \\[1mm]
j \, < \, r \, \leq 0}}
\bigg]
(-1)^r \, q^{(j+\frac14)^2-\frac12 r^2}
\, + \, 
\dfrac12 \, \theta_{\frac12, \, 1}(\tau,0)
\,\ = \,\ 
\dfrac12 \, q^{\frac{1}{16}} + \, \cdots
\label{n3:eqn:2022-716a}
\\[2mm]
\widetilde{g}_3(\tau)
&=&
\dfrac{1}{\sqrt{2}} \, \bigg[
\sum\limits_{\substack{j, \, r \, \in \zzz \\[1mm]
0 \, < \, r \, \leq \, j}}
-
\sum\limits_{\substack{ j, \, r \, \in \zzz \\[1mm]
j \, < \, r \, \leq 0}}
\bigg]
(-1)^j \, q^{(j+\frac12)^2-\frac12(r+\frac12)^2}
\,\ = \,\ 
\dfrac{1}{\sqrt{2}} \, q^{\frac18} + \, \cdots
\nonumber
\end{eqnarray}}
and
$$\left\{
\begin{array}{lclcl}
\widetilde{g}_1\Big(-\dfrac{1}{\tau}\Big) 
\, = \, (-i\tau) \, \widetilde{g}_1(\tau), 
& &
\widetilde{g}_2\Big(-\dfrac{1}{\tau}\Big) 
\, = \, (-i\tau) \, \widetilde{g}_3(\tau), 
& &
\widetilde{g}_3\Big(-\dfrac{1}{\tau}\Big) 
\, = \, (-i\tau) \, \widetilde{g}_2(\tau)
\\[4mm]
\widetilde{g}_1(\tau+1) \, = \, e^{\frac{\pi i}{8}} \, \widetilde{g}_2(\tau), 
& &
\widetilde{g}_2(\tau+1) \, = \, e^{\frac{\pi i}{8}} \, \widetilde{g}_1(\tau), 
& &
\widetilde{g}_3(\tau+1) \, = \, e^{\frac{\pi i}{4}} \, \widetilde{g}_3(\tau)
\end{array} \right.
$$
Now we put $f_i(\tau) \, := \, \dfrac{\widetilde{g}_i(\tau)}{\eta(\tau)^2}$ \,  
for $i \in \{1,2,3\}$. 
Since $\eta(-\frac{1}{\tau}) = (-i\tau)^{\frac12}\eta(\tau)$ and 
$\eta(\tau+1)= e^{\frac{\pi i}{12}} \eta(\tau)$, 
these functions $f_i(\tau)$ satisfy the following equations: 
$$\left\{
\begin{array}{lclcl}
f_1\Big(-\dfrac{1}{\tau}\Big) \, = \, f_1(\tau), 
& &
f_2\Big(-\dfrac{1}{\tau}\Big) \, = \, f_3(\tau), 
& &
f_3\Big(-\dfrac{1}{\tau}\Big) \, = \, f_2(\tau)
\\[4mm]
f_1(\tau+1) \, = \, e^{-\frac{\pi i}{24}} \, f_2(\tau), 
& &
f_2(\tau+1) \, = \, e^{-\frac{\pi i}{24}} \, f_1(\tau), 
& &
f_3(\tau+1) \, = \, e^{\frac{\pi i}{12}} \, f_3(\tau)
\\[3mm]
f_1(\tau) \, = \, \frac12 \, q^{-\frac{1}{48}} + \cdots ,
& &
f_2(\tau) \, = \, \frac12 \, q^{-\frac{1}{48}} + \cdots ,
& &
f_3(\tau) \, = \, \frac{1}{\sqrt{2}} \, q^{\frac{1}{24}} + \cdots 
\end{array} \right.
$$

\noindent
Then, by Lemma 4.8 in \cite{KW1988} (or Lemma 4.9 in \cite{W2001}), 
we have 
$$
f_1(\tau) \, = \, \frac12 \cdot \frac{\eta(\tau)^2}{\eta(\frac{\tau}{2})\eta(2\tau)},
\hspace{10mm}
f_2(\tau) \, = \, \frac12 \cdot \frac{\eta(\frac{\tau}{2})}{\eta(\tau)},
\hspace{10mm}
f_3(\tau) \, = \, \frac12 \cdot \sqrt{2} \, \frac{\eta(2\tau)}{\eta(\tau)}
$$
so
\begin{equation}
\widetilde{g}_1(\tau) \, = \, 
\frac12 \cdot \frac{\eta(\tau)^4}{\eta(\frac{\tau}{2})\eta(2\tau)},
\hspace{5mm}
\widetilde{g}_2(\tau) \, = \, \frac12 \cdot \eta(\tau) \eta\Big(\frac{\tau}{2}\Big),
\hspace{5mm}
\widetilde{g}_3(\tau) \, = \, \frac{1}{\sqrt{2}} \cdot \eta(\tau) \eta(2\tau)
\label{n3:eqn:2022-716b}
\end{equation}

\vspace{1.5mm}

\noindent
Then by \eqref{n3:eqn:2022-716a} and \eqref{n3:eqn:2022-716b}, we have
\begin{enumerate}
\item[{\rm 1)}] \,\ $\widetilde{g}_1(\tau) \,\ = \,\ 
\bigg[
\sum\limits_{\substack{j, \, r \, \in \zzz \\[1mm]
0 \, < \, r \, \leq \, j}}
-
\sum\limits_{\substack{ j, \, r \, \in \zzz \\[1mm]
j \, < \, r \, \leq 0}}
\bigg]
(-1)^j \, q^{(j+\frac14)^2-\frac12 r^2}
\, + \, 
\dfrac12 \, \theta_{\frac12, \, 1}^{(-)}(\tau,0)
\,\ = \,\ 
\dfrac12 \cdot \dfrac{\eta(\tau)^4}{\eta(\frac{\tau}{2})\eta(2\tau)}$
\item[{\rm 2)}] \,\ $\widetilde{g}_2(\tau) \,\ = \,\ 
\bigg[
\sum\limits_{\substack{j, \, r \, \in \zzz \\[1mm]
0 \, < \, r \, \leq \, j}}
-
\sum\limits_{\substack{ j, \, r \, \in \zzz \\[1mm]
j \, < \, r \, \leq 0}}
\bigg]
(-1)^r \, q^{(j+\frac14)^2-\frac12 r^2}
\, + \, 
\dfrac12 \, \theta_{\frac12, \, 1}(\tau,0)
\,\ = \,\ 
\dfrac12 \cdot \eta(\tau) \eta\Big(\dfrac{\tau}{2}\Big) $
\item[{\rm 3)}] \,\ $\widetilde{g}_3(\tau)
\,\ = \,\ 
\dfrac{1}{\sqrt{2}} \, \bigg[
\sum\limits_{\substack{j, \, r \, \in \zzz \\[1mm]
0 \, < \, r \, \leq \, j}}
-
\sum\limits_{\substack{ j, \, r \, \in \zzz \\[1mm]
j \, < \, r \, \leq 0}}
\bigg]
(-1)^j \, q^{(j+\frac12)^2-\frac12(r+\frac12)^2}
\,\ = \,\ 
\dfrac{1}{\sqrt{2}} \cdot \eta(\tau) \eta(2\tau)$
\end{enumerate}
proving Proposition \ref{n3:prop:2022-716a}, since 
$$
\theta_{\frac12,1}(\tau,0) \, = \,
\frac{\eta(\tau)^2}{\eta(\frac{\tau}{2})} \hspace{10mm} 
\text{and} \hspace{10mm}
\theta_{\frac12,1}^{(-)}(\tau,0) \, = \,
\frac{\eta(\frac{\tau}{2})\eta(2\tau)}{\eta(\tau)} \,\ .
$$
\end{proof}

\end{document}